\documentclass{tac}

\newcommand{\revdate}{October 2022}

\usepackage[utf8]{inputenc} 
\usepackage[T1]{fontenc}
\usepackage{amsfonts, amssymb, amsmath}

\usepackage{enumitem}

\usepackage{mathtools} 
\usepackage{tensor} 
\usepackage{scalerel} 

\usepackage{perpage}
\MakePerPage{footnote}

\usepackage{chngcntr}
\counterwithin{table}{section}

\usepackage[ngerman, english]{babel}
\useshorthands{"}
\makeatletter
 \defineshorthand[ngerman]{"/}{\mbox{-}\bbl@allowhyphens}
\makeatother
\addto\extrasenglish{\languageshorthands{ngerman}}

\usepackage[x11names]{xcolor}
\hypersetup{
	pdfauthor={Seerp Roald Koudenburg},
	pdftitle={Augmented virtual double categories},
  pdfkeywords={augmented virtual double category, multicategory, Yoneda structure},
}

\usepackage{tikz} 
\usetikzlibrary{decorations.markings, matrix, arrows, fit, calc}

\tikzstyle{map} = [->, font=\scriptsize]
\tikzstyle{linj} = [left hook->, font=\scriptsize]
\tikzstyle{rinj} = [right hook->, font=\scriptsize]
\tikzstyle{mono} = [>->, font=\scriptsize]
\tikzstyle{epi} = [->>, font=\scriptsize]
\tikzstyle{cell} = [double,double equal sign distance,-implies, shorten >= 3.75pt, shorten <= 3.75pt, font=\scriptsize]
\tikzstyle{eq} = [double,double equal sign distance]
\tikzstyle{ps} = [shorten >= 2pt]

\tikzstyle{iso} = [above, sloped, inner sep=1.5pt]
\tikzstyle{nat} = [above, sloped, inner sep=2pt]
\tikzstyle{desc} = [fill=white, inner sep=2pt]
\tikzstyle{dots} = [black, font=]

\tikzstyle{small} = [font=\scriptsize]

\tikzstyle{textbaseline} = [baseline=-3.2pt]

\tikzstyle{barred} = [decoration={markings, mark=at position 0.5 with {\draw[-] (0,-1.5pt) -- (0,1.5pt);}}, postaction ={decorate}]

\tikzstyle{math35} = [matrix of math nodes, row sep={3.25em,between origins}, column sep={3.5em,between origins}, text height=1.5ex, text depth=0.25ex, nodes in empty cells]
\tikzstyle{minimath} = [matrix of math nodes, row sep={3em,between origins}, column sep={3.25em,between origins}, font=\scriptsize, text height=1ex, text depth=0.25ex, nodes in empty cells]
\tikzstyle{scheme} = [textbaseline, x=1.6em, y=1.6em, yshift=-2.4em, font=\scriptsize, text depth=0ex, every node/.style={overlay}, execute at end picture = { \useasboundingbox ($(current bounding box.north west) + (0,0.4em)$) rectangle ($(current bounding box.south east) - (0,0.4em)$); }]

\makeatletter
\def\slashedarrowfill@#1#2#3#4#5{%
  $\m@th\thickmuskip0mu\medmuskip\thickmuskip\thinmuskip\thickmuskip
   \relax#5#1\mkern-7mu%
   \cleaders\hbox{$#5\mkern-2mu#2\mkern-2mu$}\hfill
   \mathclap{#3}\mathclap{#2}%
   \cleaders\hbox{$#5\mkern-2mu#2\mkern-2mu$}\hfill
   \mkern-7mu#4$%
}
\def\rightslashedarrowfill@{%
  \slashedarrowfill@\relbar\relbar\mapstochar\rightarrow}
\newcommand\xslashedrightarrow[2][]{%
  \ext@arrow 0055{\rightslashedarrowfill@}{#1}{#2}}
\makeatother

\def\slashedrightarrow{\xslashedrightarrow{}}

\newcommand{\conc}{%
  \mathbin{
    \mathchoice
    {\raisebox{1ex}{\scalebox{.7}{$\frown$}}}
    {\raisebox{1ex}{\scalebox{.7}{$\frown$}}}
    {\raisebox{.7ex}{\scalebox{.5}{$\frown$}}}
    {\raisebox{.7ex}{\scalebox{.5}{$\frown$}}}
  }
}

\newcommand{\cocart}{\mathrm{cocart}}
\newcommand{\cart}{\mathrm{cart}}

\providecommand{\cororef}[1]{Corollary~\ref{#1}}
\providecommand{\defref}[1]{Definition~\ref{#1}}
\providecommand{\exref}[1]{Example~\ref{#1}}

\providecommand{\lemref}[1]{Lemma~\ref{#1}}
\providecommand{\propref}[1]{Proposition~\ref{#1}}

\providecommand{\thmref}[1]{Theorem~\ref{#1}}

\providecommand{\secref}[1]{Section~\ref{#1}}

\providecommand{\dfn}{\coloneqq}

\providecommand{\of}{\circ}
\providecommand{\iso}{\cong}

\providecommand{\brar}{\slashedrightarrow}
\providecommand{\xrar}[1]{\xrightarrow{#1}}
\providecommand{\xlar}[1]{\xleftarrow{#1}}
\providecommand{\xbrar}[1]{\xslashedrightarrow{#1}}
\providecommand{\Rar}{\Rightarrow}

\providecommand{\eps}{\varepsilon}

\newcommand{\dash}{\makebox[1.2ex]{\text{--}}}

\providecommand{\tens}{\otimes}

\providecommand{\ul}[1]{\underline{#1}{}}

\providecommand{\brcs}[1]{\lbrace #1 \rbrace}

\providecommand{\brks}[1]{\lbrack #1 \rbrack}

\providecommand{\pars}[1]{\left(#1\right)}
\providecommand{\bigpars}[1]{\bigl(#1\bigr)}
\providecommand{\Bigpars}[1]{\Bigl(#1\Bigr)}
\providecommand{\lns}[1]{\lvert#1\rvert}

\providecommand{\angles}[1]{\langle#1\rangle}

\providecommand{\gen}[1]{\angles{#1}}

\providecommand{\ol}[1]{\overline{#1}}

\providecommand{\set}[1]{\brcs{#1}}

\providecommand{\natarrow}{\Rightarrow}

\providecommand{\map}[3]{#1\colon#2\to#3}
\providecommand{\mono}[3]{#1\colon#2\rightarrowtail#3}

\providecommand{\nat}[3]{#1\colon#2\natarrow#3}
\providecommand{\cell}[3]{#1\colon#2\Rightarrow#3}

\providecommand{\hmap}[3]{#1\colon#2\slashedrightarrow#3}

\providecommand{\inv}[1]{{#1}^{-1}}

\newcommand{\id}{\mathrm{id}}
\newcommand{\yon}{\mathrm{y}}

\providecommand{\ladj}{\dashv}

\DeclareMathOperator{\colim}{colim}
\providecommand{\intl}{\int\limits}

\providecommand{\op}[1]{#1^\textup{op}}
\providecommand{\co}[1]{#1^\textup{co}}

\providecommand{\ps}[1]{\widehat{#1}}

\providecommand{\catvar}[1]{\mathcal{#1}}

\providecommand{\2}{\mathsf 2}

\providecommand{\D}{\catvar D}
\providecommand{\E}{\catvar E}

\providecommand{\K}{\catvar K}
\renewcommand{\L}{\catvar L}

\providecommand{\V}{\catvar V}

\providecommand{\Set}{\mathsf{Set}}
\providecommand{\Cat}{\mathsf{Cat}}
\providecommand{\enCat}[1]{#1\text-\Cat}
\providecommand{\inCat}[1]{\Cat(#1)}
\providecommand{\twoCat}{2\text{-}\Cat}

\providecommand{\MonCat}{\mathsf{MonCat}}

\providecommand{\AugVirtDblCat}{\mathsf{AugVirtDblCat}}
\providecommand{\VirtDblCat}{\mathsf{VirtDblCat}}

\providecommand{\Rel}{\mathsf{Rel}}
\providecommand{\Span}[1]{\mathsf{Span}(#1)}
\providecommand{\Mat}[1]{#1\text-\mathsf{Mat}}
\providecommand{\Mod}{\mathsf{Mod}}

\providecommand{\Prof}{\mathsf{Prof}}
\providecommand{\MonProf}{\mathsf{MonProf}}
\providecommand{\enProf}[1]{#1\text-\Prof}
\providecommand{\ensProf}[1]{#1\text-\mathsf{sProf}}
\providecommand{\inProf}[1]{\Prof(#1)}
\providecommand{\spFib}[1]{\mathsf{spFib}(#1)}

\providecommand{\Rep}{\mathsf{Rep}}
\providecommand{\opRep}{\mathsf{opRep}}

\providecommand{\Alg}[1]{#1\text-\mathsf{Alg}}

\providecommand{\hc}{\odot}

\providecommand{\tab}[1]{\gen{#1}}

\providecommand{\cur}[1]{#1^{\scriptscriptstyle\lambda}}

\DeclareMathOperator{\fc}{\mathsf{fc}}

	\title{Augmented virtual double categories}
	\author{Seerp Roald Koudenburg}
	\address{Mathematics Research and Teaching Group\\Middle East Technical University\\Northern Cyprus Campus\\ 99738 Kalkanl\i, G\"uzelyurt\\Turkish Republic of Northern Cyprus\\via Mersin 10, Turkey}
	\eaddress{roaldkoudenburg@gmail.com}
	\thanks{Parts of this paper were written during my visit to Macquarie University in September--November 2015. I would like to thank Richard Garner, Mark Weber and Ram\'on Abud Alcal\'a for helpful discussions, and I am grateful to the Macquarie University Research Centre for its funding of my visit. This paper has greatly benefitted from discussions with Robert Par\'e, especially those that took place during my visit to Dalhousie University in August 2016. I would like to thank Bob, as well as the @CAT-group for partly funding my visit to Dalhousie. I thank the anonymous referee for the prompt review of this paper and their helpful suggestions, which led to several improvements. Thanks to Vladimir Sotirov for suggesting corrections to Def.\ 1.2 and Ex.\ 1.6 (February 2025).}
\begin{document}
	\maketitle
	\begin{abstract}
		In this article the notion of virtual double category (also known as fc"/multicategory) is extended as follows. While cells in a virtual double category classically have a horizontal multi-source and single horizontal target, the notion of augmented virtual double category introduced here extends the latter notion by including cells with empty horizontal target as well.
		
		Any augmented virtual double category comes with a \emph{built-in} notion of ``locally small object'' and we describe advantages of using augmented virtual double categories as a setting for formal category theory rather than 2-categories, which are classically \emph{equipped} with a notion of ``admissible object'' by means of a Yoneda structure in the sense of Street and Walters.

An object is locally small precisely if it admits a horizontal unit, and we show that the notions of augmented virtual double category and virtual double category coincide in the presence of all horizontal units. Without assuming the existence of horizontal units we show that most of the basic theory for virtual double categories, such as that of restriction and composition of horizontal morphisms, extends to augmented virtual double categories. We introduce and study in augmented virtual double categories the notion of ``pointwise'' composition of horizontal morphisms, which formalises the classical composition of profunctors given by the coend formula.
	\end{abstract}
	\tableofcontents	

	\section*{Introduction}\addcontentsline{toc}{section}{\protect\numberline{}Introduction}
	Analogous to the generalisation of monoidal category to multicategory, Burroni in \cite{Burroni71} generalised the notion of double category to that of virtual double category \cite{Cruttwell-Shulman10} (Burroni used the term `multicat\'egorie'). A virtual double category consists of objects $A$, $B$, $C,\dotsc$, two types of morphism $\map fAC$ and $\hmap JAB$ (which we will draw vertically and horizontally respectively) and cells $\phi$ of the form as on the left below, each with a single morphism $\hmap KCD$ as horizontal target and a (potentially empty) path $\bar J = (A_0 \xbrar{J_1} A_1 \dotsb A_{n-1} \xbrar{J_n} A_n)$ of morphisms as horizontal source.
	\begin{displaymath}
		\begin{tikzpicture}[baseline]
			\matrix(m)[math35, column sep={3.25em,between origins}]{A_0 & A_1 & A_{n-1} & A_n \\ C & & & D \\};
			\path[map]	(m-1-1) edge[barred] node[above] {$J_1$} (m-1-2)
													edge node[left] {$f$} (m-2-1)
									(m-1-3) edge[barred] node[above] {$J_n$} (m-1-4)
									(m-1-4) edge node[right] {$g$} (m-2-4)
									(m-2-1) edge[barred] node[below] {$K$} (m-2-4);
			\path[transform canvas={xshift=1.625em}]	(m-1-2) edge[cell] node[right] {$\phi$} (m-2-2);
			\draw				($(m-1-2)!0.5!(m-1-3)$) node[xshift=-1.5pt] {$\dotsb$};
		\end{tikzpicture} \qquad\qquad\qquad\qquad \begin{tikzpicture}[baseline]
			\matrix(m)[math35, column sep={1.625em,between origins}]
				{A_0 & & A_1 & & A_{n-1} & & A_n \\ & & & C & & & \\};
			\path[map]	(m-1-1) edge[barred] node[above] {$J_1$} (m-1-3)
													edge node[below left] {$f$} (m-2-4)
									(m-1-5) edge[barred] node[above] {$J_n$} (m-1-7)
									(m-1-7) edge node[below right] {$g$} (m-2-4);
			\path				(m-1-4) edge[cell] node[right] {$\psi$} (m-2-4);
			\draw				(m-1-4) node[xshift=-1.5pt] {$\dotsb$};
		\end{tikzpicture}
	\end{displaymath}
	The present article introduces the notion of `augmented virtual double category', which extends that of virtual double category by including cells $\psi$ as on the right above, with empty horizontal targets. The prototypical augmented virtual double category $\Prof$ has as morphisms functors $\map fAC$ and profunctors \mbox{$\map J{\op A \times B}{\Set}$} between categories $A$, $B$, $C, \dotsc$ that need not be locally small, i.e.\ need not have all hom"/sets isomorphic to objects in $\Set$. As does any augmented virtual double category, $\Prof$ contains a virtual double category $U(\Prof)$ consisting of cells of the form $\phi$ above only (see \exref{virtual double category of unary cells} below).
	
	In contrast to the vertical morphisms, horizontal morphisms are not equipped with composition in either notion of virtual double category. In both $\Prof$ and $U(\Prof)$ for example the composite of two profunctors along a properly large category does not exist in general. A fortuitous path $\ul J = (J_1, \dotsc, J_n)$ of horizontal morphisms however may still admit a composite $(J_1 \hc \dotsb \hc J_n)$, defined as such by a universal cell \mbox{$(J_1, \dotsc, J_n) \Rar (J_1 \hc \dotsb \hc J_n)$}. Likewise an object $A$ may admit a horizontal unit morphism $\hmap{I_A}AA$ defined by a universal cell $(A) \Rar I_A$, whose horizontal source is the empty path at $A$. E.g.\ $A \in \Prof$ admits a horizontal unit $I_A$ if and only if $A$ is locally small, in which case $I_A$ consists of its hom"/sets.
	
	A fundamental advantage of working with an augmented virtual double category $\K$ is that its collection of vertical morphisms form a $2$"/category $V(\K)$, whose cells are those of the form $\psi$ above with empty horizontal source $\bar J = (A_0)$. In contrast vertical morphisms in a virtual double category only form a category \emph{a priori}. E.g.\ while $\Prof$ contains all natural transformations $\cell\psi fg$ between functors $f$ and $\map g{A_0}C$, only those with $C$ locally small can be canonically identified with cells in the virtual double category $U(\Prof)$ contained in $\Prof$, namely the cells $\phi$ above with $\ul J = (A_0)$ and $K = I_C$; for details see \exref{(Set, Set')-Prof} below.
	
	One of the main results of this paper (\thmref{unital virtual double categories} below) asserts that the notions of virtual double category and augmented virtual double category are equivalent whenever all horizontal units exist; such (augmented) virtual double categories we will call unital virtual double categories. In any unital virtual double category cells $\cell\psi fg$ as above, with $\ul J = (A_0)$, correspond precisely to cells of the form
	\begin{displaymath}
		\begin{tikzpicture}[textbaseline]
			\matrix(m)[math35]{A_0 & A_0 \\ C & C \\};
			\path[map]	(m-1-1) edge[barred] node[above] {$I_{A_0}$} (m-1-2)
													edge node[left] {$f$} (m-2-1)
									(m-1-2) edge node[right] {$g$} (m-2-2)
									(m-2-1) edge[barred] node[below] {$I_C$} (m-2-2);
			\path[transform canvas={xshift=1.75em}]	(m-1-1) edge[cell] (m-2-1);
		\end{tikzpicture}
	\end{displaymath}
	(see \cororef{vertical cells} below).
	
	A further advantage of using augmented virtual double categories is that they allow for suppressing all `unit coherence cells', such as $\lambda$ in the composite on the left"/hand side below, which are often used in compositions of cells in unital virtual double categories. Indeed using the language of augmented virtual double categories the cell $\psi$ in the left"/hand side, with the horizontal unit $I_C$ as horizontal target, corresponds to the cell $\psi'$ in the right"/hand side, such that the two composites below coincide. Moreover the right"/hand side allows us to consider this composite even when the horizontal unit $I_C$ does not exist. Thus proofs of results for unital virtual double categories can both be significantly shortened as well as be generalised to proofs that apply to (not necessarily unital) augmented virtual double categories. Parts of \lemref{companion identities lemma}, \cororef{horizontal cells} and \lemref{restrictions and composites} below are obtained in this way from analogous results in \cite{Cruttwell-Shulman10}.
	\begin{displaymath}
		\begin{tikzpicture}[textbaseline]
			\matrix(m)[math35]{A_0 & A_n & B_m \\ C & C & D \\ C & & D \\};
			\path[map]	(m-1-1) edge[barred] node[above] {$\ul J$} (m-1-2)
													edge node[left] {$f$} (m-2-1)
									(m-1-2) edge[barred] node[above] {$\ul H$} (m-1-3)
													edge node[right] {$g$} (m-2-2)
									(m-1-3) edge node[right] {$h$} (m-2-3)
									(m-2-1) edge[barred] node[below] {$I_C$} (m-2-2)
													edge node[left] {$\id_C$} (m-3-1)
									(m-2-2) edge[barred] node[below] {$K$} (m-2-3)
									(m-2-3) edge node[right] {$\id_D$} (m-3-3)
									(m-3-1) edge[barred] node[below] {$K$} (m-3-3);
			\path[transform canvas={xshift=1.75em}] (m-1-1) edge[cell] node[right] {$\psi$} (m-2-1)
									(m-1-2) edge[cell] node[right] {$\phi$} (m-2-2);
			\path				(m-2-2) edge[cell] node[right] {$\lambda$} (m-3-2);
		\end{tikzpicture} \quad = \quad \begin{tikzpicture}[textbaseline]
				\matrix(m)[math35, column sep={1.75em,between origins}]{A_0 & & A_n & & B_m \\ & C & & D & \\};
				\path[map]	(m-1-1) edge[barred] node[above] {$\ul J$} (m-1-3)
														edge[transform canvas={xshift=-1pt}] node[left] {$f$} (m-2-2)
										(m-1-3) edge[barred] node[above] {$\ul H$} (m-1-5)
														edge[transform canvas={xshift=1pt}] node[right] {$g$} (m-2-2)
										(m-1-5) edge[transform canvas={xshift=1pt}] node[right] {$h$} (m-2-4)
										(m-2-2) edge[barred] node[below] {$K$} (m-2-4);
				\path[transform canvas={yshift=0.25em}]	(m-1-2) edge[cell] node[right, inner sep=2pt] {$\psi'$} (m-2-2);
				\path[transform canvas={xshift=0.875em}]	(m-1-3) edge[cell] node[right] {$\phi$} (m-2-3);
			\end{tikzpicture}
	\end{displaymath}
	
	Our main purpose for augmented virtual double categories is to use them as a convenient ``double dimensional'' setting for the internalisation of the notion of Yoneda embedding, thus giving an alternative to the classical $2$"/categorical approach of Street and Walters' Yoneda structures \cite{Street-Walters78}. While such work will have to appear as a sequel to the present paper (for a draft see Sections~4 and 5 of \cite{Koudenburg19b}) we will, after having given the outline of this paper below, close this introduction by broadly describing its ideas and some of its benefits.
	
	This article is largely based on Sections~1, 2 and 3 of the draft \cite{Koudenburg19b}. Since the material presented here is significantly more streamlined as well as expanded in several ways, the author encourages readers to consult the present article rather than the latter sections. The first version of \cite{Koudenburg19b} used the term ``hypervirtual double category'' for what is the main notion of this article, where presently we use ``augmented virtual double category'' instead, as suggested to the author by Robert Par\'e.
	
	\subsection*{Outline}
	We start by introducing the notion of augmented virtual double category in \secref{augmented virtual double category section}. Examples are given in \secref{examples section}, including the augmented virtual double category $\enProf\V$ of $\V$"/enriched profunctors (\exref{enriched profunctors}); $\enProf{(\V, \V')}$ of $\V$"/enriched profunctors between $\V'$"/categories, where $\V' \supset \V$ is a universe enlargement of $\V$ in the sense of Section~3.11 of \cite{Kelly82} (\exref{(V, V')-Prof}); $\ensProf\V$ of small $\V$"/enriched profunctors in the sense of \cite{Day-Lack07} (\exref{small V-profunctors}); $\inProf\E$ of profunctors internal to a category $\E$ with pullbacks (\exref{internal profunctors}); $\Rel(\E)$ of relations in a category $\E$ with pullbacks (\exref{internal relations}) and $\spFib\K$ of split two"/sided fibrations in a finitely complete $2$"/category $\K$ (\exref{internal split fibrations}). In \secref{2-category of augmented virtual double categories section} the $2$"/category of augmented virtual double categories, the functors between them and their transformations is introduced, and its equivalences are characterised as functors that are full, faithful and essentially surjective in the appropriate sense.
	
	In \secref{restriction section} the notion of restriction $\hmap{K(f, g)}AB$ of a horizontal morphism $\hmap KCD$ along vertical morphisms $\map fAC$ and $\map gBD$, that was introduced in Section~7 of \cite{Cruttwell-Shulman10} for virtual double categories, is translated to augmented virtual double categories as well as expanded to include that of `nullary restriction' $\hmap{C(f, g)}AB$ of an object $C$ along morphisms $\map fAC$ and $\map gBC$. Both types of restriction are defined by cells with a certain universal property; such cells are called `cartesian', while `weakly cocartesian' cells satisfy a vertical dual property. Full and faithful morphisms are defined in terms of cartesian cells, and the horizontal unit $\hmap{I_A}AA$ of an object $A$ is defined to be the nullary restriction $I_A \dfn A(\id_A, \id_A)$.
	
	In \secref{companion and conjoint section} the `companion' $\hmap{f_*}AC$ and `conjoint' $\hmap{f^*}CA$ of a vertical morphism $\map fAC$ are introduced as the nullary restrictions $f_* \dfn C(f, \id_C)$ and $f^* \dfn C(\id_C, f)$; they can be thought of as the horizontal morphisms that are respectively ``isomorphic'' and ``adjoint'' to $f$. Unlike similar definitions for unital virtual double categories given in \cite{Cruttwell-Shulman10} we need not require that the horizontal unit $I_C$ exists. Analogous to observations for double categories in Section~4 of \cite{Shulman08} we prove that the companion $f_*$ can be equivalently defined in three ways: by a cartesian cell $(f_*) \Rightarrow (A)$, by a weakly cocartesian cell $(A) \Rightarrow (f_*)$, or by a pair of cells $(f_*) \Rightarrow (A)$ and $(A) \Rightarrow (f_*)$ satisfying certain ``companion identities''; a horizontal dual result holds for the conjoint $f^*$. These identities and their horizontal duals imply that companions, conjoints and horizontal units are preserved by any functor of augmented virtual double categories. Horizontal units $I_A$ can both be regarded as the companion and conjoint of the identity $\id_A$; we prove that their defining cells $(I_A) \Rightarrow (A)$ and $(A) \Rightarrow (I_A)$ are both cartesian as well as weakly cocartesian. We prove lemmas that relate the notions of nullary restriction, horizontal unit and full and faithful morphism. We describe adjunctions and absolute left liftings in the $2$"/category $V(\K)$ in terms of companions and conjoints in $\K$.
	
	In \secref{representable morphism section} we consider horizontal morphisms $\hmap JAB$ that are representable by a vertical morphism $\map fAB$, i.e.\ $J \iso f_*$. Given an augmented virtual double category $\K$ in \thmref{lower star} we describe its locally full sub-augmented virtual double category $\Rep(\K)$ of representable horizontal morphisms in terms of its vertical $2$"/category $V(\K)$.

In \secref{composition section} we study composites $(J_1 \hc \dotsb \hc J_n)$ of paths $\ul J = (J_1, \dotsc, J_n)$ of horizontal morphisms. As described above, these are defined by universal cells $\ul J \Rightarrow (J_1 \hc \dotsb \hc J_n)$. The main lemma of \secref{restrictions and extensions in terms of companions and conjoints section} proves that, for morphisms $A \xrar f C \xbrar K D \xlar g B$, the restriction $K(f, g)$ and the composite $f^* \hc K \hc g_*$ coincide. This translates and extends Theorem~7.16 of \cite{Cruttwell-Shulman10} from unital virtual double categories to augmented virtual double categories; here too we need not require the existence of any horizontal units. Internalising the composition of profunctors given by the ``coend formula'', in \secref{pointwise horizontal composites section} we introduce and study `pointwise' horizontal composites. Informally, a horizontal composite is pointwise whenever any of its restrictions are again horizontal composites. Finally in \secref{unital virtual double categories section} we prove the equivalence of the notions of virtual double category and augmented virtual double category in the presence of all horizontal units.
	
	\subsection*{Motivation: internalising Yoneda embeddings}
	Following Wood \cite{Wood82} and Grandis and Par\'e \cite{Grandis-Pare08}, who used `proarrow equipments' and double categories respectively to formalise parts of classical category theory, recently certain unital virtual double categories have been used to study formal category theory in less well behaved settings, as follows. Cruttwell and Shulman in \cite{Cruttwell-Shulman10} internalise the notion of fully faithful morphism in the unital virtual double category $\mathbb M\mathsf{od}(\mathbb X)$ of `modules' in a virtual double category $\mathbb X$, while Riehl and Verity in \cite{Riehl-Verity17} internalise the notions of fully faithful morphism, `exact square' and (pointwise) Kan extension in the unital virtual double category $\ul{\textup{Mod}}_\mathcal K$ of modules between $\infty$"/categories in the homotopy $2$"/category of a `$\infty$"/cosmos' $\mathcal K$.
	
	In line with the previous our goal for augmented virtual double categories is to use them as a setting for internalising the notion of Yoneda embedding, as we will now sketch roughly. We start by recalling the classical internalisation of Yoneda embeddings, in the form of a Yoneda structure on a $2$"/category \cite{Street-Walters78}. First let us recall some details of the classical theory of Yoneda embeddings. Given a properly large, locally small category $A$, recall from \cite{Freyd-Street95} that the category $PA \dfn \Set^{\op A}$ of small set valued presheaves $\op A \rightarrow \Set$ is necessarily locally properly large. Thus the natural $2$"/dimensional environment for classical Yoneda embeddings $\map{yA}A{PA}$ for such $A$, that map each $x \in A$ to the representable presheaf $A(\dash, x)$, is the $2$"/category $\Cat$ of locally large categories, functors and natural transformations. Using that any functor $\map fAB$, with small hom"/sets $B(fx, y)$ for all $x \in A$ and $y \in B$, induces a functor $\map{B(f, 1)}B{PA}$ given by $B(f, 1)(y) = B(f\dash, y)$, Yoneda's lemma can be rephrased internally to $\Cat$ as follows: the canonical natural transformation
	\begin{displaymath}
  	\begin{tikzpicture}
			\matrix(m)[math35, column sep={1.75em,between origins}]{A & & B \\ & PA & \\};
			\path[map]	(m-1-1) edge node[above] {$f$} (m-1-3)
													edge[transform canvas={xshift=-1pt}] node[left] {$yA$} (m-2-2)
									(m-1-3) edge[transform canvas={xshift=1pt}] node[right] {$B(f, 1)$} (m-2-2);
			\path[transform canvas={xshift=0.2em,yshift=-0.1em}]	($(m-1-1)!0.5!(m-2-2)$) edge[cell] node[above left, inner sep=0pt] {$\chi^f$} (m-1-3);
		\end{tikzpicture} 		
	\end{displaymath}
	exhibits $f$ as the `absolute left lifting' (see e.g.\ \cite{Street-Walters78}) of $yA$ along $B(f, 1)$.
	
	A Yoneda structure on a $2$"/category $\mathcal C$ formalises the previous as follows. Firstly it postulates a right ideal\footnote{A right ideal of a category is a class $\mathcal I$ of morphisms closed under precomposition: if $f$ and $g$ are composable then $g \in \mathcal I$ implies $g \of f \in \mathcal I$.}  of `admissible morphisms' $\map fAC$ in $\mathcal C$, which internalises the smallness condition on the functors $f$ above; an object $A$ is then called admissible whenever its identity morphism $\id_A$ is so. Secondly it provides a morphism \mbox{$\map{yA}A{PA}$} for each admissible object $A$, internalising the Yoneda embedding, together with a cell $\chi^f$ as above for each admissible $\map fAB$. The cells $\chi^f$ are required to satisfy three axioms \cite{Street-Walters78}:
	\begin{enumerate}[label=(\arabic*)]
		\item $\chi^f$ exhibits $B(f, 1)$ as the left Kan extension of $f$ along $yA$ (together with (3) below this formalises $yA$ being dense);
		\item $\chi^f$ exhibits $f$ as the absolute left lifting of $yA$ along $B(f, 1)$ (as above);
		\item roughly, the assignment $f \mapsto \chi_f$ is pseudofunctorial.
	\end{enumerate}
	
	The stronger notion of `good Yoneda structure' on a finitely complete $2$"/category $\mathcal C$, introduced by Weber \cite{Weber07}, is defined as above except for replacing axioms (1) and (3) with the following stronger axiom:
	\begin{enumerate}
		\item[(4)] if any cell $\phi$ in $\mathcal C$, of the form as below and with $f$ admissible, exhibits $f$ as the absolute left lifting of $yA$ along $g$ then it exhibits $g$ as the pointwise left Kan extension of $yA$ along $f$ (in the sense of \cite{Street74}).
	\end{enumerate}
	\begin{displaymath}
  	\begin{tikzpicture}
			\matrix(m)[math35, column sep={1.75em,between origins}]{A & & B \\ & PA & \\};
			\path[map]	(m-1-1) edge node[above] {$f$} (m-1-3)
													edge[transform canvas={xshift=-1pt}] node[left] {$yA$} (m-2-2)
									(m-1-3) edge[transform canvas={xshift=1pt}] node[right] {$g$} (m-2-2);
			\path[transform canvas={xshift=0.2em,yshift=-0.1em}]	($(m-1-1)!0.5!(m-2-2)$) edge[cell] node[above left, inner sep=1pt] {$\phi$} (m-1-3);
		\end{tikzpicture} 		
	\end{displaymath}
	While Yoneda embeddings for $2$"/categories, that is categories enriched in $\Cat$, combine to form a Yoneda structure, satisfying axioms (1)--(3) above, they do not satisfy axiom (4); this is explained in Remark~9 of \cite{Walker18}.
	
	The main idea of internalising the notion of Yoneda embedding in an augmented virtual double category $\K$ is the following. Instead of postulating a notion of admissible morphism in $\K$ we internalise that notion, simply by regarding all horizontal morphisms of $\K$ to be admissible; consequently a vertical morphism $\map fAC$ (respectively an object $A$) is considered admissible whenever its companion $\hmap{f_*}AC$ (respectively its horizontal unit $\hmap{I_A}AA$) exists. Compare the prototypal example $\K = \Prof$, in which all horizontal morphisms are small"/set"/valued profunctors. This allows for a simpler notion of Yoneda embedding: we do not have to specify an ideal of admissible morphisms and, instead of having to provide a full coherent family of Yoneda embeddings as in a Yoneda structure, we may simply consider a single Yoneda embedding $\map \yon A{\ps A}$ in $\K$, as follows. To exhibit $\yon$ as a Yoneda embedding amounts to providing, for each horizontal morphism $\hmap JAB$ in $\K$, a cell $\chi$ as below on the left, satisfying the following `Yoneda' and `density' axioms, analogous to (2) and (4) above:
	\begin{enumerate}
		\item[(y)] $\chi$ is a `cartesian cell' (see \secref{restriction section} below), thus exhibiting $J$ as the restriction of the object $\ps A$ along the morphisms $\yon$ and $J^\lambda$;
		\item[(d)] any cartesian cell $\phi$ in $\K$, as on the right, defines $g$ as a pointwise left Kan extension of $\yon$ along $J$ (in the sense of Section~4.2 of \cite{Koudenburg19b}). 
	\end{enumerate}
	\begin{displaymath}
		\begin{tikzpicture}[baseline]
			\matrix(m)[math35, column sep={1.75em,between origins}]{A & & B \\ & \ps A & \\};
			\path[map]	(m-1-1) edge[barred] node[above] {$J$} (m-1-3)
													edge[transform canvas={xshift=-1pt}, ps] node[left] {$\yon$} (m-2-2)
									(m-1-3) edge[transform canvas={xshift=1pt}, ps] node[right] {$J^\lambda$} (m-2-2);
			\path[transform canvas={yshift=0.25em}]	(m-1-2) edge[cell] node[right, inner sep=3pt] {$\chi$} (m-2-2);
		\end{tikzpicture} \qquad\qquad\qquad\qquad\qquad\qquad \begin{tikzpicture}[baseline]
			\matrix(m)[math35, column sep={1.75em,between origins}]{A & & B \\ & \ps A & \\};
			\path[map]	(m-1-1) edge[barred] node[above] {$J$} (m-1-3)
													edge[transform canvas={xshift=-1pt}, ps] node[left] {$\yon$} (m-2-2)
									(m-1-3) edge[transform canvas={xshift=1pt}, ps] node[right] {$g$} (m-2-2);
			\path[transform canvas={yshift=0.25em}]	(m-1-2) edge[cell] node[right, inner sep=3pt] {$\phi$} (m-2-2);
		\end{tikzpicture}
	\end{displaymath}
	In $\K = \Prof$ the Yoneda embedding $\yon \dfn yA$ for a locally small category $A$ is defined as before, with $\ps A = \Set^{\op A}\mspace{-3mu}$, while the functors $J^\lambda$ are defined by $J^\lambda y = J(\dash, y)$ for $y \in B$. The components $\map{\chi_{x, y}}{J(x, y)}{\ps A(\yon x, J^\lambda y)}$ of $\chi$, which axiom (y) requires to be isomorphisms, are supplied by Yoneda's lemma.
	
	To conclude this motivation we list some benefits of using augmented virtual double categories $\K$ to internalise the notion of Yoneda embedding $\map\yon A{\ps A}$.
	\begin{itemize}
		\item If all nullary restrictions $\ps A(\yon ,g)$ exist in $\K$, as they do in all well known examples (e.g.\ $\K = \Prof$), then the assigment $J \mapsto J^\lambda$ induces an equivalence between morphisms of the forms $A \brar B$ and $B \to \ps A$. In contrast the assignment $f \mapsto B(f, 1)$ induced by a Yoneda structure is in general not essentially surjective onto morphisms $B \to PA$ (e.g.\ take $A = 1 = B$ the terminal category in $\mathcal C = \Cat$). 
		
		\item Several types of Yoneda embedding satisfy the axioms (1)--(3) of a Yoneda structure but their appropriate notion of admissible morphism does not form a right ideal, so that the theory of Yoneda structures does not apply fully. For a well"/known example consider a closed symmetric monoidal, small complete category $\V$. The appropriate notion of admissible $\V$"/functor for the $\V$"/enriched Yoneda embeddings $\map\yon A{\ps A\strut^\textup s}$, where $\ps A\strut^\textup s$ denotes the $\V$"/category of `small $\V$"/presheaves on $A$' in the sense of \cite{Day-Lack07}, does not form a right ideal. In contrast, it is not hard to prove that these $\yon$ do form Yoneda embeddings in the augmented virtual double category $\ensProf\V$ (\exref{small V-profunctors} below), so that the theory of \cite{Koudenburg19b} applies fully. Likewise Yoneda embeddings induced by a `KZ doctrine', as studied by Walker in \cite{Walker18}, do not satisfy the right ideal property; they too are likely to form Yoneda embeddings in some appropriately chosen augmented virtual double category.
		
		\item As noted previously $\V$"/enriched Yoneda embeddings, in the classical sense of e.g.\ Section~2.4 of \cite{Kelly82}, form Yoneda structures that do not in general satisfy axiom~(4). On the other hand they do satisfy a stronger, i.e.\ $\V$"/enriched, version of axiom~(1). Thus neither notion of Yoneda structure captures the notion of $\V$"/enriched Yoneda embedding exactly. The augmented virtual double category $\enProf{(\V, \V')}$, as described in \exref{(V, V')-Prof} below, is the right setting in this case: therein axioms (y) and (d) capture correctly the $\V$"/enriched notion of Yoneda embedding. Roughly this is because the pointwise notion of Kan extension in axiom~(d) above, when considered in $\enProf{(\V, \V')}$, coincides with the classical notion of $\V$"/enriched Kan extension (see Section~4.4 of \cite{Koudenburg19b}).
		
		\item Regarding all horizontal morphisms of augmented virtual double categories as admissible allows us to prove results that assert admissibility of morphisms or objects. For instance consider any full and faithful morphism $\map hCE$ in an augmented virtual double category that has all restrictions of the form $K(f, g)$. \lemref{full and faithfulness and horizontal units} below proves that $h$ `reflects admissibility', that is $C$ is admissible (i.e.\ the horizontal unit $I_C$ exists) whenever $E$ is. Even though inside a unital virtual double category, i.e.\ with all horizontal units, our notion of full and faithful coincides with that of \cite{Cruttwell-Shulman10}, notice that in unital virtual double categories this result is meaningless. For another example remember that (good) Yoneda structures provide a Yoneda embedding for each admissible object. Inside augmented virtual double categories a weak converse holds: given a Yoneda embedding $\map yA{\ps A}$ the horizontal unit of $A$ exists whenever all nullary restrictions of the form $\ps A(\yon, g)$ exist; see Section~5.1 of \cite{Koudenburg19b}.
		
		\item For an example of a formalisation of a more involved result, similar to those of the previous item, let $\map fAC$ be a $\V$"/functor and $\map{f^\sharp}{\ps A\strut^\textup s}{\ps C\strut^\textup s}$ be given by left Kan extending small $\V$"/presheaves on $A$ along $f$. In our terms Proposition~3.3 of \cite{Day-Lack07} can be rephrased as follows: $f^\sharp$ has a right adjoint if and only if $f$ is admissible (in other words its companion $\hmap{f_*}AC$ exists) in the augmented virtual double category $\ensProf\V$ (\exref{small V-profunctors} below). This result partially formalises to any $\map fAC$ in a general augmented virtual double category $\K$, assuming that the Yoneda embeddings $\map{y_A}A{\ps A}$ and $\map{y_C}C{\ps C}$ exist: the morphism $f^\sharp$ can then be internalised as being the left Kan extension of $y_C \of f$ along $y_A$ and the implications
		\begin{displaymath}
			f^\sharp \text{ has a right adjoint} \quad \Leftrightarrow \quad (y_C \of f)_* \text{ exists} \quad \Rightarrow \quad f_* \text{ exists}
		\end{displaymath}
		hold under mild conditions on $\K$; see Section~5.2 of \cite{Koudenburg19b}.
		
		\item Axiom (y) above allows us to capture a monoidal variant of Yoneda's lemma as follows. Recall that a monoidal structure $\tens$ on a category $A$ induces a monoidal structure $\ps\tens$ on its category of presheaves $\ps A \dfn \Set^{\op A}$ that is given by Day convolution \cite{Day70}
			\begin{equation} \label{Day convolution}
		(p \mathbin{\ps\tens} q)(x) \dfn \int^{u, v \in A} A(x, u \tens v) \times pu \times qv, \qquad\quad\quad \text{where $p, q \in \ps A$}.
		\end{equation}
		With respect to $\ps\tens$ the Yoneda embedding $\map\yon A{\ps A}$ forms a monoidal functor, i.e.\ it admits a coherent family of isomorphisms $\bar\yon \colon \yon x \mathbin{\ps\tens} \yon y \iso \yon(x \tens y)$ where $x$, $y \in A$. Thus a monoidal functor $\map{(\yon, \bar\yon)}{(A, \tens)}{(\ps A, \ps\tens)}$ satisfies the following monoidal variant of the Yoneda axiom (y) for profunctors: any lax monoidal profunctor \mbox{$\hmap JAB$} (i.e.\ equipped with coherent maps \mbox{$\map{\bar J}{J(x_1, y_1) \times J(x_2, y_2)}{J(x_1 \tens x_2, y_1 \tens y_2)}$}) induces a lax monoidal functor $\map{\cur J}B{\ps A}$ such that $\ps A(\yon\dash, \cur J\dash) \iso J$ as lax monoidal profunctors. In detail, we can take $\cur J$ to be as defined before: $\cur Jy \dfn J(\dash, y)$, and take the coherence morphisms $\nat{\bar{\cur J}}{\cur J y_1 \mathbin{\ps\tens} \cur J y_2}{\cur J(y_1 \tens y_2)}$ to be induced by the composites
		\begin{multline*}
			A(x, u \tens v) \times J(u, y_1) \times J(v, y_2) \\
			\xrar{\id \times \bar J} A(x, u \tens v) \times J(u \tens v, y_1 \tens y_2) \to J(x, y_1 \tens y_2),
		\end{multline*}
		where the unlabelled morphism is induced by the functoriality of $J$ in $A$.
	
		In fact, we may consider the augmented virtual double category $\MonProf$ of lax monoidal functors and lax monoidal profunctors between (possibly large) monoidal categories, and show that the monoidal Yoneda embedding $(\yon, \bar\yon)$ satisfies both axioms (y) and (d) therein. As described in the first item above, together these axioms imply an equivalence between lax monoidal profunctors $A \brar B$ and lax monoidal functors $B \to \ps A$.\footnote{This observation is not new: it follows from Pisani's study of exponentiable multicategories in Section~2 of \cite{Pisani14}.} We remark that the monoidal Yoneda embeddings $(\yon, \bar\yon)$ do not combine to form a Yoneda structure on the two $2$"/categories consisting of either lax or colax monoidal functors between (possibly large) monoidal categories: this is because colax monoidal structures on a functor $\map fAB$ correspond to \emph{lax} monoidal structures on the corresponding functor $\map{B(f, 1)}B{\ps A}$ and similarly lax monoidal structures on $f$ do in general not induce monoidal structures on $B(f, 1)$.
	
		\item One of the main results of \cite{Koudenburg19b} formalises Day convolution for monoidal structures to algebraic structures defined by monads $T$ on augmented virtual double categories $\K$. More precisely, given a Yoneda embedding $\map\yon A{\ps A}$ in $\K$ it gives conditions under which a $T$"/algebra structure $\map a{TA}A$ on $A$ induces such a structure $\ps a$ on $\ps A$, in a way that makes $\yon$ into a Yoneda embedding in the augmented virtual double category $\Alg T$ of $T$"/algebras. Formalising a result of Im and Kelly \cite{Im-Kelly86} one can then show, for instance, that $\yon$, as a $T$"/morphism, defines $(\ps A, \ps a)$ as the free cocompletion of $(A, a)$ in $\Alg T$; see Section~5.4 of \cite{Koudenburg19b}.
	\end{itemize}
	
	\section{Augmented virtual double categories}\label{augmented virtual double category section}
	The definition of augmented virtual double category below uses the notion of \emph{directed graph}, by which we mean a parallel pair of functions
	\begin{displaymath}
	  A = \bigl(\mspace{-11mu} \begin{tikzpicture}[textbaseline]
		\matrix(m)[math35, column sep=1em]{A_1 & A_0 \\};
		\path[map]	(m-1-1) edge[above, transform canvas={yshift=2pt}] node {$s$} (m-1-2)
												edge[below, transform canvas={yshift=-2pt}] node {$t$} (m-1-2);
	\end{tikzpicture}\mspace{-11mu}\bigr)
	\end{displaymath}
	from a class $A_1$ of \emph{edges} to a class $A_0$ of \emph{vertices}. An edge $e$ with $(s,t)(e) = (x,y)$ is denoted $x \xrar e y$; the vertices $x$ and $y$ are called its \emph{source} and \emph{target}. Any category $\catvar C$ has an underlying graph $\catvar C_1 \rightrightarrows \catvar C_0$ with $\catvar C_1$ and $\catvar C_0$ the classes of morphisms and objects of $\catvar C$ respectively. Conversely, remember that any graph $A$ generates a \emph{free category} $\fc A$, with as objects the vertices of $A$ and as morphisms $x \to y$ (possibly empty) paths \mbox{$\ul e = (x = x_0 \xrar{e_1} x_1 \xrar{e_2} \dotsb \xrar{e_n} x_n = y)$} of edges in $A$; we write $\lns{\ul e} \dfn n$ for their lengths. Composition in $\fc A$ is given by concatenation
	\begin{displaymath}
	  (\ul e, \ul f) \mapsto \ul e \conc \ul f \dfn (x_0 \xrar{e_1} \dotsb \xrar{e_n} x_n = y_0 \xrar{f_1} \dotsb \xrar{f_m} y_m)  
	\end{displaymath}
	of paths, while the empty path $(x)$ forms the identity at $x \in A_0$.
	
	\begin{notation}
	  For any integer $n \geq 1$ we write $n' \dfn n - 1$.
	\end{notation}
	
	\begin{definition} \label{augmented virtual double category}
		An \emph{augmented virtual double category} $\K$ consists of
		\begin{itemize}[label=-]
			\item a class $\K_0$ of \emph{objects} $A$, $B, \dotsc$
			\item a category $\K_\textup v$ with $\K_{\textup v0} = \K_0$, whose morphisms $\map fAC$, $\map gBD, \dotsc$ are called \emph{vertical morphisms};
			\item a directed graph $\K_\textup h$ with $\K_{\textup h0} = \K_0$, whose edges are called \emph{horizontal morphisms} and denoted by slashed arrows $\hmap JAB$, $\hmap KCD, \dotsc$;
			\item a class of \emph{cells} $\phi$, $\psi, \dotsc$ that are of the form
				\begin{equation} \label{cell}
					\begin{tikzpicture}[textbaseline]
						\matrix(m)[math35]{A_0 & A_n \\ C & D \\};
						\path[map]	(m-1-1) edge[barred] node[above] {$\ul J$} (m-1-2)
																edge node[left] {$f$} (m-2-1)
												(m-1-2) edge node[right] {$g$} (m-2-2)
												(m-2-1) edge[barred] node[below] {$\ul K$} (m-2-2);
						\path[transform canvas={xshift=1.75em}]	(m-1-1) edge[cell] node[right] {$\phi$} (m-2-1);
					\end{tikzpicture}
				\end{equation}
				where $\ul J$ and $\ul K$ are (possibly empty) paths in $\K_\textup h$ with $\lns{\ul K} \leq 1$;
			\item	for any path of cells
				\begin{equation} \label{composable sequence}
					\begin{tikzpicture}[textbaseline]
						\matrix(m)[math35, column sep={4.5em,between origins}]
							{	A_{10} & A_{1m_1} & A_{2m_2} &[1em] A_{n'm_{n'}} & A_{nm_n} \\
								C_0 & C_1 & C_2 & C_{n'} & C_n \\ };
						\path[map]	(m-1-1) edge[barred] node[above] {$\ul J_1$} (m-1-2)
																edge node[left] {$f_0$} (m-2-1)
												(m-1-2) edge[barred] node[above] {$\ul J_2$} (m-1-3)
																edge node[right] {$f_1$} (m-2-2)
												(m-1-3) edge node[right] {$f_2$} (m-2-3)
												(m-1-4) edge[barred] node[above] {$\ul J_n$} (m-1-5)
																edge node[left] {$f_{n'}$} (m-2-4)
												(m-1-5) edge node[right] {$f_n$} (m-2-5)
												(m-2-1) edge[barred] node[below] {$\ul K_1$} (m-2-2)
												(m-2-2) edge[barred] node[below] {$\ul K_2$} (m-2-3)
												(m-2-4) edge[barred] node[below] {$\ul K_n$} (m-2-5);
						\path[transform canvas={xshift=2.25em}]	(m-1-1) edge[cell] node[right] {$\phi_1$} (m-2-1)
												(m-1-2)	edge[cell] node[right] {$\phi_2$} (m-2-2)
												(m-1-4) edge[cell] node[right] {$\phi_n$} (m-2-4);
						\draw				($(m-1-3)!0.5!(m-2-4)$) node {$\dotsb$};
					\end{tikzpicture}
				\end{equation}
				of length $n \geq 1$ and a cell $\psi$ as on the left below, a \emph{vertical composite} as on the right;
				\begin{equation} \label{vertical composite}
					\begin{tikzpicture}[textbaseline]
						\matrix(m)[math35, column sep={8em,between origins}]{C_0 & C_n \\ E & F \\};
						\path[map]	(m-1-1) edge[barred] node[above] {$\ul K_1 \conc \ul K_2 \conc \dotsb \conc \ul K_n$} (m-1-2)
																edge node[left] {$h$} (m-2-1)
												(m-1-2) edge node[right] {$k$} (m-2-2)
												(m-2-1) edge[barred] node[below] {$\ul L$} (m-2-2);
						\path[transform canvas={xshift=4em}]	(m-1-1) edge[cell] node[right] {$\psi$} (m-2-1);
					\end{tikzpicture} \qquad\quad\qquad \begin{tikzpicture}[textbaseline]
						\matrix(m)[math35, column sep={8em,between origins}]{A_{10} & A_{nm_n} \\ E & F \\};
						\path[map]	(m-1-1) edge[barred] node[above] {$\ul J_1 \conc \ul J_2 \conc \dotsb \conc \ul J_n$} (m-1-2)
																edge node[left] {$h \of f_0$} (m-2-1)
												(m-1-2) edge node[right] {$k \of f_n$} (m-2-2)
												(m-2-1) edge[barred] node[below] {$\ul L$} (m-2-2);
						\path[transform canvas={xshift=1.15em}]	(m-1-1) edge[cell] node[right] {$\psi \of (\phi_1, \dotsc, \phi_n)$} (m-2-1);
					\end{tikzpicture}
				\end{equation}
			\item	\emph{horizontal identity cells} as on the left below, one for each $\hmap JAB$;
				\begin{displaymath}
					\begin{tikzpicture}[baseline]
						\matrix(m)[math35]{A & B \\ A & B \\};
						\path[map]	(m-1-1) edge[barred] node[above] {$(J)$} (m-1-2)
																edge node[left] {$\id_A$} (m-2-1)
												(m-1-2) edge node[right] {$\id_B$} (m-2-2)
												(m-2-1) edge[barred] node[below] {$(J)$} (m-2-2);
						\path[transform canvas={xshift=1.75em}]	(m-1-1) edge[cell] node[right] {$\id_J$} (m-2-1);
					\end{tikzpicture} \qquad\qquad\qquad\qquad \begin{tikzpicture}[baseline]
						\matrix(m)[math35]{A & A \\ C & C \\};
						\path[map]	(m-1-1) edge[barred] node[above] {$(A)$} (m-1-2)
																edge node[left] {$f$} (m-2-1)
												(m-1-2) edge node[right] {$f$} (m-2-2)
												(m-2-1) edge[barred] node[below] {$(C)$} (m-2-2);
						\path[transform canvas={xshift=1.75em}]	(m-1-1) edge[cell] node[right] {$\id_f$} (m-2-1);
					\end{tikzpicture}
				\end{displaymath}
			\item \emph{vertical identity cells} as on the right above, one for each $\map fAC$, with empty horizontal source $(A)$ and target $(C)$, that are preserved by vertical composition: $\id_h \of (\id_f) = \id_{h \of f}$; we write $\id_A \dfn \id_{\id_A}$.
		\end{itemize}
		The vertical composition above is required to satisfy the \emph{associativity axiom}
		\begin{multline} \label{associativity axiom}
			\chi \of \bigpars{\psi_1 \of (\phi_{11}, \dotsc, \phi_{1m_1}), \dotsc, \psi_n \of (\phi_{n1}, \dotsc, \phi_{nm_n})} \\
			= \bigpars{\chi \of (\psi_1, \dotsc, \psi_n)} \of (\phi_{11}, \dotsc, \phi_{nm_n}),
		\end{multline}
		whenever both sides\ make sense, as well as the \emph{unit axioms}
		\begin{flalign*}
			&& \id_C \of (\phi) = \phi, \quad \id_K \of (\phi) = \phi, \quad \phi \of (\id_A) &= \phi, \quad \phi \of (\id_{J_1}, \dotsc, \id_{J_n}) = \phi & \\
			\text{and} && \psi \of (\phi_1, \dotsc, \phi_i, \id_{f_i}, \phi_{i+1}, \dotsc, \phi_n) &= \psi \of (\phi_1, \dotsc, \phi_i, \phi_{i+1}, \dotsc, \phi_n) &
		\end{flalign*}
		whenever these make sense and where, in the last axiom, $0 \leq i \leq n$ (in the cases $i = 0$ and $i = n$ the identity cells $\id_{f_0}$ and $\id_{f_n}$ are respectively the first and last cell in the path that is composed with $\psi$).
	\end{definition}
	
	For a cell $\phi$ as in \eqref{cell} above we call the vertical morphisms $f$ and $g$ its \emph{vertical source} and \emph{target} respectively, the path of horizontal morphisms $\ul J = (J_1, \dotsc, J_n)$ its \emph{horizontal source} and $\ul K$ its \emph{horizontal target}. We write $\lns \phi \dfn (\lns{\ul J}, \lns{\ul K})$ for the \emph{arity} of $\phi$. An $(n,1)$"/ary cell will be called \emph{unary}, $(n,0)$-ary cells \emph{nullary} and $(0,0)$-ary cells \emph{vertical}.
	
	When writing down paths $(J_1, \dotsc, J_n)$ of length $n \leq 1$ we will often leave out parentheses and simply write $J_1 \dfn (A_0 \xbrar{J_1} A_1)$ or $A_0 \dfn (A_0)$. Likewise in the composition of cells: $\psi \of \phi_1 \dfn \psi \of (\phi_1)$. We will often denote unary cells simply by $\cell\phi{(J_0, \dotsc, J_n)}K$ and nullary cells by $\cell\psi{(J_0, \dotsc, J_n)}C$, leaving out their vertical source and target. When drawing compositions of cells it is often helpful to depict them in full detail and, in the case of nullary cells, draw their empty horizontal target as a single object, as shown below.
	\begin{displaymath}
		\begin{tikzpicture}[baseline]
			\matrix(m)[math35]{A_0 \\ C \\};
			\path[map]	(m-1-1) edge[bend right=45] node[left] {$f$} (m-2-1)
													edge[bend left=45] node[right] {$g$} (m-2-1);
			\path				(m-1-1) edge[cell] node[right] {$\psi$} (m-2-1);
		\end{tikzpicture} \quad \begin{tikzpicture}[baseline]
			\matrix(m)[math35, column sep={1.625em,between origins}]
				{A_0 & & A_1 & \dotsb & A_{n'} & & A_n \\ & & & C & & & \\};
			\path[map]	(m-1-1) edge[barred] node[above] {$J_1$} (m-1-3)
													edge node[below left] {$f$} (m-2-4)
									(m-1-5) edge[barred] node[above] {$J_n$} (m-1-7)
									(m-1-7) edge node[below right] {$g$} (m-2-4);
			\path				(m-1-4) edge[cell] node[right] {$\psi$} (m-2-4);
		\end{tikzpicture}	\quad \begin{tikzpicture}[baseline]
			\matrix(m)[math35, column sep={1.75em,between origins}]{& A_0 & \\ C & & D \\};
			\path[map]	(m-1-2) edge[transform canvas={xshift=-1pt}] node[left] {$f$} (m-2-1)
													edge[transform canvas={xshift=1pt}] node[right] {$g$} (m-2-3)
									(m-2-1) edge[barred] node[below] {$K$} (m-2-3);
			\path				(m-1-2) edge[cell, transform canvas={yshift=-0.25em}] node[right, inner sep=2.5pt] {$\phi$} (m-2-2);
		\end{tikzpicture} \quad \begin{tikzpicture}[baseline]
			\matrix(m)[math35, column sep={3.25em,between origins}]{A_0 & A_1 & A_{n'} & A_n \\ C & & & D \\};
			\path[map]	(m-1-1) edge[barred] node[above] {$J_1$} (m-1-2)
													edge node[left] {$f$} (m-2-1)
									(m-1-3) edge[barred] node[above] {$J_n$} (m-1-4)
									(m-1-4) edge node[right] {$g$} (m-2-4)
									(m-2-1) edge[barred] node[below] {$K$} (m-2-4);
			\path[transform canvas={xshift=1.625em}]	(m-1-2) edge[cell] node[right] {$\phi$} (m-2-2);
			\draw				($(m-1-2)!0.5!(m-1-3)$) node {$\dotsb$};
		\end{tikzpicture}
	\end{displaymath}
	
	A cell with identities as vertical source and target is called \emph{horizontal}. A horizontal cell $\cell\phi JK$ with unary horizontal source is called \emph{invertible} if there exists a horizontal cell $\cell\psi KJ$ such that $\phi \of \psi = \id_K$ and $\psi \of \phi = \id_J$; in that case we write $\inv\phi \dfn \psi$. When drawing diagrams we shall often depict identity morphisms by equal signs ($=$), while in identity cells we will leave out the arrows $(\Downarrow\! \id)$, leaving them empty instead. Because composition of cells is associative we will leave out bracketings when writing down composites.
	
	For convenience we use the `whisker' notation from $2$-category theory and define
	\begin{displaymath}
		h \of (\phi_1, \dotsc, \phi_n) \dfn \id_h \of (\phi_1, \dotsc, \phi_n) \qquad \text{and} \qquad \psi \of f \dfn \psi \of \id_f,
	\end{displaymath}
	whenever the right-hand side makes sense. Moreover, for any path
	\begin{displaymath}
		\begin{tikzpicture}
			\matrix(m)[math35]{A_0 & A_1 & A_{n'} & A_n & B_1 & B_{m'} & B_m \\ C & & & D & & & G \\};
			\path[map]	(m-1-1) edge[barred] node[above] {$J_1$} (m-1-2)
													edge node[left] {$f$} (m-2-1)
									(m-1-3) edge[barred] node[above] {$J_n$} (m-1-4)
									(m-1-4) edge[barred] node[above] {$H_1$} (m-1-5)
													edge node[right] {$g$} (m-2-4)
									(m-1-6) edge[barred] node[above] {$H_m$} (m-1-7)
									(m-1-7) edge node[right] {$h$} (m-2-7)
									(m-2-1) edge[barred] node[below] {$\ul K$} (m-2-4)
									(m-2-4) edge[barred] node[below] {$\ul L$} (m-2-7);
			\path[transform canvas={xshift=1.75em}]	(m-1-2) edge[cell] node[right] {$\phi$} (m-2-2)
									(m-1-5) edge[cell] node[right] {$\psi$} (m-2-5);
			\draw				($(m-1-2)!0.5!(m-1-3)$) node {$\dotsb$}
									($(m-1-5)!0.5!(m-1-6)$) node {$\dotsb$};
		\end{tikzpicture}
	\end{displaymath}
	with $\lns{\ul K} + \lns{\ul L} \leq 1$ we define the \emph{horizontal composite} $\cell{\phi \hc \psi}{\ul J \conc \ul H}{\ul K \conc \ul L}$ by
	\begin{displaymath}
		\phi \hc \psi \dfn \id_{\ul K \conc \ul L} \of (\phi, \psi),
	\end{displaymath}
	where $\id_{\ul K \conc \ul L}$ is to be interpreted as the identity $\map{\id_C}CC$ in the case that $\ul K \conc \ul L = (C)$. The following lemma follows easily from the associativity and unit axioms for vertical composition.
	\begin{lemma} \label{horizontal composition}
		Horizontal composition $(\phi, \psi) \mapsto \phi \hc \psi$, as defined above, satisfies the \emph{associativity} and \emph{unit axioms}
		\begin{displaymath}
			(\phi \hc \psi) \hc \chi = \phi \hc (\psi \hc \chi), \qquad (\id_f \hc \phi) = \phi \qquad \text{and} \qquad (\phi \hc \id_g) = \phi
		\end{displaymath}
		whenever these make sense. Moreover, horizontal and vertical composition satisfy the \emph{interchange axioms}
		\begin{flalign*}
			&& \bigpars{\psi \of (\phi_1, \dotsc, \phi_n)} \hc \bigl(\chi \of (\xi_1, \dotsc, \xi_m)\bigr ) &= (\psi \hc \chi) \of (\phi_1, \dotsc, \phi_n, \xi_1, \dotsc, \xi_m) & \\
			\text{and} && \psi \of \bigpars{\phi_1, \dotsc, (\phi_{i'} \hc \phi_{i}), \dotsc, \phi_n} &= \psi \of (\phi_1, \dotsc, \phi_{i'}, \phi_i, \dotsc, \phi_n) &
		\end{flalign*}
		whenever they make sense.
	\end{lemma}
	
	The following examples relate augmented virtual double categories to some classical $2$"/dimensional categorical notions. Further examples are given in the next section.
	\begin{example} \label{virtual double category of unary cells}
	  By restricting to augmented virtual double categories in which all nullary cells are vertical identities, that is $\id_f$ for some vertical morphism $\map fAC$, we recover the classical notion of \emph{virtual double category}, in the sense of \cite{Cruttwell-Shulman10} or Section~5.1 of \cite{Leinster04} (where it is called $\fc$-multicategory). Virtual double categories were originally introduced by Burroni \cite{Burroni71} who called them `multicat\'egories'. It follows that every augmented virtual double category $\K$ contains a virtual double category $U(\K)$ consisting of its objects, vertical and horizontal morphisms, and unary cells.
	\end{example}
	
	\begin{example} \label{vertical 2-category}
	  Augmented virtual double categories with no horizontal morphisms, so that all cells are vertical, correspond precisely to \emph{$2$"/categories}, with the compositions $\hc$ and $\of$\ corresponding to the vertical and horizontal composition in $2$"/categories respectively. Thus every augmented virtual double category $\K$ contains a \emph{vertical $2$"/category} $V(\K)$, consisting of its objects, vertical morphisms and vertical cells. As remarked in the Introduction virtual double categories do not canonically contain $2$"/categories of vertical morphisms unless they have all horizontal units (see Proposition~6.1 of \cite{Cruttwell-Shulman10}). In \thmref{unital virtual double categories} below we will see that the notions of augmented virtual double category and virtual double category coincide in the presence of horizontal units (see \defref{cocartesian paths}).
	\end{example}
	
	\begin{example}
	  Restricting to augmented virtual double categories $\K$ with $\K_\textup v = 1$, the terminal category, and whose only nullary cell is the identity cell $\id_*$ for the unique object $* \in \K$, recovers the notion of \emph{multicategory} (see e.g.\ Section~2.1 of \cite{Leinster04}). Similarly augmented virtual double categories $\K$ with $\K_\textup v = 1$ whose only \emph{vertical} cell is $\id_*$ can be regarded as multicategories $\K$ equipped with a bimodule $\K \brar 1$, in the sense of Definition~2.3.6 of \cite{Leinster04}, where $1$ denotes the terminal multicategory.
	\end{example}
	
	\begin{example} \label{augmented virtual double category from a unital virtual double category}
		Let us recall the notion of a \emph{horizontal unit} $\hmap{I_A}AA$ for an object $A$ in a virtual double category $\K$ from e.g.\ Section~8 of \cite{Hermida00} or Section~5 of \cite{Cruttwell-Shulman10}; see also \secref{composition section} below. It is defined by an \emph{cocartesian cell} $\eta_A$ as on the left below, satisfying the following universal property: any cell $\phi$ in $\K$, with $A$ an object in its horizontal source as in the middle below, factors uniquely through $\eta_A$ as a cell $\phi'$ as shown, where the empty cells denote paths of identity cells.
		\begin{equation} \label{factorisation through unit}
			\begin{tikzpicture}[textbaseline]
					\matrix(m)[math35, column sep={1.75em,between origins}]{& A & \\ A & & A \\};
					\path[map]	(m-2-1) edge[barred] node[below] {$I_A$} (m-2-3);
					\path				(m-1-2) edge[eq, transform canvas={xshift=-1pt}] (m-2-1)
															edge[eq, transform canvas={xshift=1pt}] (m-2-3)
											(m-1-2) edge[cell, transform canvas={yshift=-0.333em, xshift=-0.4em}] node[right] {$\eta_A$} (m-2-2);
				\end{tikzpicture} \qquad\qquad \begin{tikzpicture}[textbaseline]
					\matrix(m)[math35, column sep={1.75em,between origins}]{X_0 & & A & & Y_m \\ & C & & D & \\};
					\path[map]	(m-1-1) edge[barred] node[above] {$\ul J$} (m-1-3)
															edge[transform canvas={xshift=-1pt}] node[left] {$f$} (m-2-2)
											(m-1-3) edge[barred] node[above] {$\ul H$} (m-1-5)
											(m-1-5) edge[transform canvas={xshift=1pt}] node[right] {$g$} (m-2-4)
											(m-2-2) edge[barred] node[below] {$K$} (m-2-4);
					\path				(m-1-3) edge[cell] node[right] {$\phi$} (m-2-3);
				\end{tikzpicture} = \begin{tikzpicture}[textbaseline]
					\matrix(m)[math35, column sep={1.75em,between origins}]
						{ & X_0 & & A & & Y_m & \\
							X_0 & & A & & A & & Y_m \\
							& C & & & & D & \\};
					\path[map]	(m-1-2) edge[barred] node[above] {$\ul J$} (m-1-4)
											(m-1-4) edge[barred] node[above] {$\ul H$} (m-1-6)
											(m-2-1) edge[barred] node[below] {$\ul J$} (m-2-3)
															edge[transform canvas={xshift=-1pt}] node[left] {$f$} (m-3-2)
											(m-2-3) edge[barred] node[below, inner sep=2pt] {$I_A$} (m-2-5)
											(m-2-5) edge[barred] node[below] {$\ul H$} (m-2-7)
											(m-2-7) edge[transform canvas={xshift=1pt}] node[right] {$g$} (m-3-6)
											(m-3-2) edge[barred] node[below] {$K$} (m-3-6);
					\path				(m-1-2) edge[eq, transform canvas={xshift=-1pt}] (m-2-1)
											(m-1-4)	edge[eq, transform canvas={xshift=-1pt}] (m-2-3)
															edge[eq, transform canvas={xshift=1pt}] (m-2-5)
											(m-1-6) edge[eq, transform canvas={xshift=1pt}] (m-2-7)
											(m-1-4) edge[cell, transform canvas={yshift=-0.333em, xshift=-0.4em}] node[right] {$\eta_A$} (m-2-4)
											(m-2-4) edge[cell] node[right] {$\phi'$} (m-3-4);
				\end{tikzpicture}
		\end{equation}
		
		We call a virtual double category \emph{unital} if each of its objects admits a horizontal unit. Choosing a cocartesian cell $\eta_A$ for each object $A$ in a unital virtual double category $\K$ allows $\K$ to be regarded as an augmented virtual double category $N(\K)$, as we shall now explain. The augmented virtual double category $N(\K)$ has as objects, morphisms and unary cells the objects, morphisms and cells of $\K$ while the nullary cells $\xi$ of $N(\K)$, of the shape as on the left below, are the cells $\xi$ of $\K$ that are of the shape as on the right.\!\footnote{Notice that a cell $\cell\xi{(J_1, \dotsc, J_n)}{I_C}$ in $\K$ as above appears as a cell in $N(\K)$ in two ways: once as a unary cell $\cell\xi{(J_1, \dotsc, J_n)}{I_C}$ and once as a nullary cell $\cell\xi{(J_1, \dotsc, J_n)}C$.}
		\begin{displaymath}
			\begin{tikzpicture}[baseline]
				\matrix(m)[math35, column sep={1.75em,between origins}]
					{A_0 & & A_1 & \dotsb & A_{n'} & & A_n \\ & & & C & & & \\};
				\path[map]	(m-1-1) edge[barred] node[above] {$J_1$} (m-1-3)
														edge node[below left] {$f$} (m-2-4)
										(m-1-5) edge[barred] node[above] {$J_n$} (m-1-7)
										(m-1-7) edge node[below right] {$g$} (m-2-4);
				\path				(m-1-4) edge[cell] node[right] {$\xi$} (m-2-4);
			\end{tikzpicture}	\qquad\qquad\qquad \begin{tikzpicture}[baseline]
				\matrix(m)[math35]{A_0 & A_1 & A_{n'} & A_n \\ C & & & C \\};
				\path[map]	(m-1-1) edge[barred] node[above] {$J_1$} (m-1-2)
														edge node[left] {$f$} (m-2-1)
										(m-1-3) edge[barred] node[above] {$J_n$} (m-1-4)
										(m-1-4) edge node[right] {$g$} (m-2-4)
										(m-2-1) edge[barred] node[below] {$I_C$} (m-2-4);
				\path[transform canvas={xshift=1.75em}]	(m-1-2) edge[cell] node[right] {$\xi$} (m-2-2);
				\draw				($(m-1-2)!0.5!(m-1-3)$) node {$\dotsb$};
			\end{tikzpicture}
		\end{displaymath}
		Composition $\psi \of (\phi_1, \dotsc, \phi_n)$ in $N(\K)$, with $\cell{\phi_i}{\ul J_i}{\ul K_i}$ as in \eqref{composable sequence} and $\cell\psi{\ul K_1 \conc \dotsb \conc \ul K_n}{\ul L}$ as in \eqref{vertical composite}, is defined as the composite in $\K$
		\begin{displaymath}
			\psi \of (\phi_1, \dotsc, \phi_n) \dfn \psi' \of (\phi_1, \dotsc, \phi_n)
		\end{displaymath}
		where $\psi'$ is defined as follows. Writing $\ul\eta_{\ul \phi}$ for the path of cocartesian cells $(\eta_{\phi_1}, \dotsc, \eta_{\phi_n})$, where $\eta_{\phi_i} \dfn \eta_{C_{i'}}$ if $\phi_i$ is nullary with horizontal target $C_{i'}$ and $\eta_{\phi_i} \dfn \id_{K_i}$ if $\phi_i$ is unary with horizontal target $\hmap{K_i}{C_{i'}}{C_i}$, the cell $\psi'$ is the unique factorisation in $\psi = \psi' \of \ul\eta_{\ul \phi}$. This factorisation $\psi'$ exists by the universal property of the $\eta_{C_{i'}}$, and it contains a unit $I_{C_{i'}}$ in its horizontal source for each nullary cell $\phi_i$ in $\ul\phi$. Finally the horizontal identity cells $\id_J$ in $N(\K)$ are simply those of $\K$, while the vertical identity cells $\id_f$ in $N(\K)$, one for each $\map fAC$, are the composites $\id_f \dfn \eta_C \of f$ in $\K$. That the composition for $N(\K)$ as defined above satisfies the associativity and unit axioms is a straightforward consequence of those axioms in $\K$, combined with the uniqueness of the factorisations $\psi'$.
		
		In \thmref{unital virtual double categories} below we will see that the assigment $\K \mapsto N(\K)$ is part of an equivalence between unital virtual double categories and augmented virtual double categories that have all horizontal units.
	\end{example}
	
	Every augmented virtual double category has a horizontal dual as follows.
	\begin{definition} \label{horizontal dual}
		Let $\K$ be an augmented virtual double category. The \emph{horizontal dual} of $\K$ is the augmented virtual double category $\co\K$ that has the same objects and vertical morphisms, that has a horizontal morphism $\hmap{\co J}AB$ for each $\hmap JBA$ in $\K$, and a cell $\co\phi$ as on the left below for each cell $\phi$ in $\K$ as on the right.
		\begin{displaymath}
			\begin{tikzpicture}[baseline]
				\matrix(m)[math35, column sep={3.25em,between origins}]{A_0 & A_1 & A_{n'} & A_n \\ C & & & D \\};
				\path[map]	(m-1-1) edge[barred] node[above] {$\co J_1$} (m-1-2)
														edge node[left] {$f$} (m-2-1)
										(m-1-3) edge[barred] node[above] {$\co J_n$} (m-1-4)
										(m-1-4) edge node[right] {$g$} (m-2-4)
										(m-2-1) edge[barred] node[below] {$\co{\ul K}$} (m-2-4);
				\path[transform canvas={xshift=1.625em}]	(m-1-2) edge[cell] node[right] {$\co\phi$} (m-2-2);
				\draw				($(m-1-2)!0.5!(m-1-3)$) node {$\dotsb$};
			\end{tikzpicture} \qquad\qquad\qquad \begin{tikzpicture}[baseline]
				\matrix(m)[math35, column sep={3.25em,between origins}]{A_n & A_{n'} & A_1 & A_0 \\ D & & & C \\};
				\path[map]	(m-1-1) edge[barred] node[above] {$J_n$} (m-1-2)
														edge node[left] {$g$} (m-2-1)
										(m-1-3) edge[barred] node[above] {$J_1$} (m-1-4)
										(m-1-4) edge node[right] {$f$} (m-2-4)
										(m-2-1) edge[barred] node[below] {$\ul K$} (m-2-4);
				\path[transform canvas={xshift=1.625em}]	(m-1-2) edge[cell] node[right] {$\phi$} (m-2-2);
				\draw				($(m-1-2)!0.5!(m-1-3)$) node {$\dotsb$};
			\end{tikzpicture}
		\end{displaymath}
		Identities and compositions in $\co\K$ are induced by those of $\K$:
		\begin{displaymath}
			\id_{\co J} \dfn \co{(\id_J)}, \quad \id_f \dfn \co{(\id_f)} \quad \text{and} \quad \co\psi \of (\co\phi_1, \dotsc, \co\phi_n) \dfn \co{\bigpars{\psi \of (\phi_n, \dotsc, \phi_1)}}.
		\end{displaymath}
	\end{definition}
	
	We end this section with a remark on the associativity of composition of cells in augmented virtual double categories.
	\begin{remark}
	  Consider a configuration of composable cells as in the scheme below, where $\phi_2$ is a nullary cell and the other cells are unary. Notice that there are two ways of vertically composing these cells if we compose the top rows first: in that case the cell $\phi_2$ can be composed either with $\psi_1$ or with $\psi_2$. In contrast, if we start by first composing the bottom two rows then there is only one way to form the composite.
	  
	  This example shows why the associativity axiom \eqref{associativity axiom} for vertical composition has to be ``read from left to right'': when read in the other direction, in general, there might be multiple ways in which the cells $(\phi_1, \dotsc, \phi_m)$ of top row can be ``distributed'' over the cells $(\psi_1, \dotsc, \psi_n)$ in the middle row.
	  
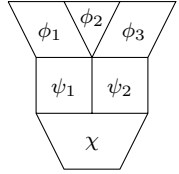
\begin{displaymath}
	    \begin{tikzpicture}[scheme, x=1.8em, y=1.8em]
	      \draw (0,3) -- (3,3) -- (2.5,2) -- (2.5, 1) -- (2,0) -- (1,0) -- (0.5,1) -- (0.5,2) -- (0,3)
	            (0.5,2) -- (2.5,2)
	            (1,3) -- (1.5,2) -- (1.5,1)
	            (2,3) -- (1.5,2)
	            (0.5,1) -- (2.5,1);
	      \draw (0.75,2.5) node {$\phi_1$}
	            (1.5,2.75) node {$\phi_2$}
	            (2.25,2.5) node {$\phi_3$}
	            (1,1.5) node {$\psi_1$}
	            (2,1.5) node {$\psi_2$}
	            (1.5,0.5) node {$\chi$};
	    \end{tikzpicture}
	  \end{displaymath}
	  
	  Formally the above observation is a manifestation of the fact that, when regarded as monoids, augmented virtual double categories $\K$ (with a fixed directed graph $\K_\textup h$)\ are monoids in a \emph{skew-monoidal category}, in the sense of Szlach\'anyi \cite{Szlachanyi12}, instead of monoids in an ordinary monoidal category. This is made precise in \cite{Koudenburg19a}.
	\end{remark}
	
	\section{Examples} \label{examples section}
		Our main source of augmented virtual double categories will be virtual double categories, as will be explained in this section. Briefly, given a virtual double category $\K$ we will consider `monoids' and `bimodules' in $\K$, as recalled from Section~5.3 of \cite{Leinster04} (or Section~2 of \cite{Cruttwell-Shulman10}) in the definition below, and these arrange into a virtual double category $\Mod(\K)$. The latter admits all horizontal units so that we can apply \exref{augmented virtual double category from a unital virtual double category}, thus obtaining an augmented virtual double category $(N \of \Mod)(\K)$. Often we will then consider a sub"/augmented virtual double category of $(N \of \Mod)(\K)$ by ``restricting the size of bimodules''. For instance, while the canonical notion of bimodule between large categories (i.e.\ categories internal to a category $\Set'$ of `large sets') is a profunctor $\hmap JAB$ with images $J(a, b)$ that are possibly large, we take the viewpoint (see \exref{(Set, Set')-Prof} below) that it is preferable to consider profunctors with all images $J(a, b)$ small.
		
	\begin{definition}[Leinster] \label{monoids and bimodules}
		Let $\K$ be a virtual double category.
		\begin{enumerate}[label =-]
			\item A \emph{monoid} $A$ in $\K$ is a quadruple $A = (A, \alpha, \bar\alpha, \tilde\alpha)$ consisting of a horizontal morphism $\hmap\alpha AA$ in $\K$ equipped with \emph{multiplication} and \emph{unit} cells
			\begin{displaymath}
				\begin{tikzpicture}[textbaseline]
					\matrix(m)[math35, column sep={1.75em,between origins}]{A & & A & & A \\ & A & & A & \\};
					\path[map]	(m-1-1) edge[barred] node[above] {$\alpha$} (m-1-3)
											(m-1-3) edge[barred] node[above] {$\alpha$} (m-1-5)
											(m-2-2) edge[barred] node[below] {$\alpha$} (m-2-4);
					\path				(m-1-1) edge[eq] (m-2-2)
											(m-1-5) edge[eq] (m-2-4)
											(m-1-3) edge[cell] node[right] {$\bar\alpha$} (m-2-3);
				\end{tikzpicture} \qquad\qquad \text{and} \qquad\qquad \begin{tikzpicture}[textbaseline]
					\matrix(m)[math35, column sep={1.75em,between origins}]{& A & \\ A & & A, \\};
					\path[map]	(m-2-1) edge[barred] node[below] {$\alpha$} (m-2-3);
					\path				(m-1-2) edge[eq, transform canvas={xshift=-1pt}] (m-2-1)
															edge[eq, transform canvas={xshift=1pt}] (m-2-3)
											(m-1-2) edge[cell, transform canvas={yshift=-0.333em}] node[right] {$\tilde\alpha$} (m-2-2);
				\end{tikzpicture}
			\end{displaymath}
			that satisfy the associativity axiom $\bar\alpha \of (\bar\alpha, \id_\alpha) = \bar\alpha \of (\id_\alpha, \bar\alpha)$ and the unit axioms $\bar\alpha \of (\tilde\alpha, \id_\alpha) = \id_\alpha = \bar\alpha \of (\id_\alpha, \tilde\alpha)$.
			\item	A \emph{morphism} $A \to C$ of monoids is a vertical morphism $\map fAC$ in $\K$ that is equipped with a cell
			\begin{displaymath}
				\begin{tikzpicture}[textbaseline]
					\matrix(m)[math35]{A & A \\ C & C \\};
					\path[map]	(m-1-1) edge[barred] node[above] {$\alpha$} (m-1-2)
															edge node[left] {$f$} (m-2-1)
											(m-1-2) edge node[right] {$f$} (m-2-2)
											(m-2-1) edge[barred] node[below] {$\gamma$} (m-2-2);
					\path[transform canvas={xshift=1.75em}]	(m-1-1) edge[cell] node[right] {$\bar f$} (m-2-1);
				\end{tikzpicture}
			\end{displaymath}
			satisfying the associativity and unit axioms $\bar\gamma \of (\bar f, \bar f) = \bar f \of \bar \alpha$ and $\tilde \gamma \of f = \bar f \of \tilde \alpha$.
			\item A \emph{bimodule} $A \brar B$ between monoids is a horizontal morphism $\hmap JAB$ in $\K$ that is equipped with left and right \emph{action} cells
			\begin{displaymath}
				\begin{tikzpicture}[textbaseline]
					\matrix(m)[math35, column sep={1.75em,between origins}]{A & & A & & B \\ & A & & B & \\};
					\path[map]	(m-1-1) edge[barred] node[above] {$\alpha$} (m-1-3)
											(m-1-3) edge[barred] node[above] {$J$} (m-1-5)
											(m-2-2) edge[barred] node[below] {$J$} (m-2-4);
					\path				(m-1-1) edge[eq] (m-2-2)
											(m-1-5) edge[eq] (m-2-4)
											(m-1-3) edge[cell] node[right] {$\lambda$} (m-2-3);
				\end{tikzpicture} \qquad\qquad \text{and} \qquad\qquad \begin{tikzpicture}[textbaseline]
					\matrix(m)[math35, column sep={1.75em,between origins}]{A & & B & & B \\ & A & & B, & \\};
					\path[map]	(m-1-1) edge[barred] node[above] {$J$} (m-1-3)
											(m-1-3) edge[barred] node[above] {$\beta$} (m-1-5)
											(m-2-2) edge[barred] node[below] {$J$} (m-2-4);
					\path				(m-1-1) edge[eq] (m-2-2)
											(m-1-5) edge[eq] (m-2-4)
											(m-1-3) edge[cell] node[right] {$\rho$} (m-2-3);
				\end{tikzpicture}
			\end{displaymath}
			satisfying the usual associativity, unit and compatibility axioms for bimodules:
			\begin{align*}
				\lambda \of (\bar\alpha, \id_J) & = \lambda \of (\id_\alpha, \lambda); & \rho \of (\id_J, \bar\beta) &= \rho \of (\rho, \id_\beta); \\
				\lambda \of (\tilde\alpha, \id_J) & = \id_J = \rho \of (\id_J, \tilde\beta); &\rho \of (\lambda, \id_\beta) &= \lambda \of (\id_\alpha, \rho).
			\end{align*}
			\item	A cell
			\begin{displaymath}
				\begin{tikzpicture}
					\matrix(m)[math35, column sep={3.25em,between origins}]{A_0 & A_1 & A_{n'} & A_n \\ C & & & D \\};
					\path[map]	(m-1-1) edge[barred] node[above] {$J_1$} (m-1-2)
															edge node[left] {$f$} (m-2-1)
											(m-1-3) edge[barred] node[above] {$J_n$} (m-1-4)
											(m-1-4) edge node[right] {$g$} (m-2-4)
											(m-2-1) edge[barred] node[below] {$K$} (m-2-4);
					\path[transform canvas={xshift=1.625em}]	(m-1-2) edge[cell] node[right] {$\phi$} (m-2-2);
					\draw				($(m-1-2)!0.5!(m-1-3)$) node {$\dotsb$};
				\end{tikzpicture}
			\end{displaymath}
			of bimodules, where $n \geq 1$, is a cell $\phi$ in $\K$ between the underlying morphisms satisfying the \emph{external equivariance} axioms
			\begin{align*}
				\phi \of (\lambda_{J_1}, \id_{J_2}, \dotsc, \id_{J_n}) &= \lambda_K \of (\bar f, \phi) \\
				\phi \of (\id_{J_1}, \dotsc, \id_{J_{n'}}, \rho_{J_n}) &= \rho_K \of (\phi, \bar g)
			\end{align*}
			and the \emph{internal equivariance} axioms
			\begin{multline*}
				\phi \of (\id_{J_1}, \dotsc, \id_{J_{i''}}, \rho_{J_{i'}}, \id_{J_i}, \id_{J_{i+1}}, \dotsc, \id_{J_n}) \\
					= \phi \of (\id_{J_1}, \dotsc, \id_{J_{i''}}, \id_{J_{i'}}, \lambda_{J_i}, \id_{J_{i+1}}, \dotsc, \id_{J_n})
			\end{multline*}
			for $2 \leq i \leq n$.
			\item A cell
			\begin{displaymath}
				\begin{tikzpicture}
					\matrix(m)[math35, column sep={1.75em,between origins}]{& A & \\ C & & D \\};
					\path[map]	(m-1-2) edge[transform canvas={xshift=-1pt}] node[left] {$f$} (m-2-1)
															edge[transform canvas={xshift=1pt}] node[right] {$g$} (m-2-3)
											(m-2-1) edge[barred] node[below] {$K$} (m-2-3);
					\path				(m-1-2) edge[cell, transform canvas={yshift=-0.25em}] node[right, inner sep=2.5pt] {$\phi$} (m-2-2);
				\end{tikzpicture}	
			\end{displaymath}
			of bimodules is a cell $\phi$ in $\K$ between the underlying morphisms satisfying the \emph{external equivariance} axiom $\lambda \of (\bar f, \phi) = \rho \of (\phi, \bar g)$.
		\end{enumerate}
		Monoids in $\K$, their morphisms and bimodules, as well as the cells between them, form a virtual double category $\Mod(\K)$, whose composition and identities are simply those of $\K$. In fact the assignment $\K \mapsto \Mod(\K)$ extends to an endo"/$2$"/functor on the $2$"/category $\VirtDblCat$ of virtual double categories; see Proposition~3.9 of \cite{Cruttwell-Shulman10}.
	\end{definition}
	As is shown in Proposition~5.5 of \cite{Cruttwell-Shulman10} the virtual double category $\Mod(\K)$ has all horizontal units, in the sense of \exref{augmented virtual double category from a unital virtual double category}. Indeed the unit $I_A$ for a monoid $A = (A, \alpha, \bar\alpha, \tilde\alpha)$ in $\K$ is the bimodule $\hmap{I_A \dfn \alpha}AA$, whose actions $\lambda$ and $\rho$ are both given by multiplication $\cell{\bar\alpha}{(\alpha, \alpha)}\alpha$. The cocartesian cell $\cell{\eta_A}A\alpha$ is the unit cell $\eta_A \dfn \tilde\alpha$: the factorisation of a cell $\phi$ through $\eta_A$, that is of the form as in \eqref{factorisation through unit}, is obtained by composing $\phi$ with the right or left action of $A$ on either bimodule $\hmap{J_n}{X_{n'}}A$ or $\hmap{H_1}A{Y_1}$ in its horizontal source.
	\begin{example} \label{augmented virtual double category of monoids}
		Let $\K$ be a virtual double category. Applying \exref{augmented virtual double category from a unital virtual double category} to the unital virtual double category $\Mod(\K)$ of bimodules in $\K$ we obtain the augmented virtual double category $(N \of \Mod)(\K)$ whose objects, morphisms and unary cells are the same as those of $\Mod(\K)$, as in \defref{monoids and bimodules}, while the nullary cells $\xi$ of $(N \of \Mod)(\K)$, of the shape as on the left below, are cells of bimodules $\xi$ of the shape as on the right, where $\hmap\gamma CC$ is the horizontal unit bimodule for the monoid $C = (C, \gamma, \bar\gamma, \tilde\gamma)$ as described above.
		\begin{displaymath}
			\begin{tikzpicture}[baseline]
				\matrix(m)[math35, column sep={1.75em,between origins}]
					{A_0 & & A_1 & \dotsb & A_{n'} & & A_n \\ & & & C & & & \\};
				\path[map]	(m-1-1) edge[barred] node[above] {$J_1$} (m-1-3)
														edge node[below left] {$f$} (m-2-4)
										(m-1-5) edge[barred] node[above] {$J_n$} (m-1-7)
										(m-1-7) edge node[below right] {$g$} (m-2-4);
				\path				(m-1-4) edge[cell] node[right] {$\xi$} (m-2-4);
			\end{tikzpicture}	\qquad\qquad\qquad \begin{tikzpicture}[baseline]
				\matrix(m)[math35]{A_0 & A_1 & A_{n'} & A_n \\ C & & & C \\};
				\path[map]	(m-1-1) edge[barred] node[above] {$J_1$} (m-1-2)
														edge node[left] {$f$} (m-2-1)
										(m-1-3) edge[barred] node[above] {$J_n$} (m-1-4)
										(m-1-4) edge node[right] {$g$} (m-2-4)
										(m-2-1) edge[barred] node[below] {$\gamma$} (m-2-4);
				\path[transform canvas={xshift=1.75em}]	(m-1-2) edge[cell] node[right] {$\xi$} (m-2-2);
				\draw				($(m-1-2)!0.5!(m-1-3)$) node {$\dotsb$};
			\end{tikzpicture}	
		\end{displaymath}
	\end{example}
	
	The remainder of this section consists of examples of (augmented) virtual double categories. They can be split into two kinds: Examples \ref{enriched profunctors}---\ref{small V-profunctors}	are examples of enriched structures, while Examples \ref{internal profunctors}---\ref{internal split fibrations} are examples of internal structures.
	\begin{notation}
	  Throughout this article we assume given a category $\Set'$ of \emph{large sets}, as well as a full subcategory $\Set \subsetneqq \Set'$ of \emph{small sets}, such that the collection of morphisms of $\Set$ forms an object in $\Set'$. A large set $A \in \Set'$ will be called \emph{properly large} if it is not isomorphic to any small set.
	\end{notation}

	\begin{example} \label{enriched profunctors}
		Let $\V = (\V, \tens, I)$ be a monoidal category. The virtual double category $\Mat\V$ of \emph{$\V$-matrices} has large sets and functions as objects and vertical morphisms, while a horizontal morphism $\hmap JAB$ is a $\V$-matrix, given by a family $J(x, y)$ of $\V$"/objects indexed by pairs $(x, y) \in A \times B$. A cell 
		\begin{displaymath}
		  \begin{tikzpicture}[baseline]
				\matrix(m)[math35]{A_0 & A_1 & A_{n'} & A_n \\ C & & & D \\};
				\path[map]	(m-1-1) edge[barred] node[above] {$J_1$} (m-1-2)
														edge node[left] {$f$} (m-2-1)
										(m-1-3) edge[barred] node[above] {$J_n$} (m-1-4)
										(m-1-4) edge node[right] {$g$} (m-2-4)
										(m-2-1) edge[barred] node[below] {$K$} (m-2-4);
				\path[transform canvas={xshift=1.75em}]	(m-1-2) edge[cell] node[right] {$\phi$} (m-2-2);
				\draw				($(m-1-2)!0.5!(m-1-3)$) node {$\dotsb$};
			\end{tikzpicture}
		\end{displaymath}
		in $\Mat\V$ is a family of $\V$-maps
		\begin{displaymath}
			\map{\phi_{(x_0, \dotsc, x_n)}}{J_1(x_0, x_1) \tens \dotsb \tens J_n(x_{n'}, x_n)}{K(fx_0, gx_n)}
		\end{displaymath}
		indexed by $(n+1)$"/tuples $(x_0, \dotsc, x_n) \in A_0 \times \dotsb \times A_n$, where the tensor product is taken to be the monoidal unit $I$ in the case that $n = 0$.
		
		The augmented virtual double category $\enProf\V \dfn (N \of \Mod)(\Mat\V)$ of monoids and bimodules in $\Mat\V$ is that of large \emph{$\V$-enriched} categories, $\V$"/functors, $\V$"/profunctors and $\V$"/natural transformations. In some more detail: a $\V$"/profunctor $\hmap JAB$, between $\V$"/categories $A$ and $B$, consists of a family of $\V$-objects $J(x, y)$, indexed by pairs of objects $x \in A$ and $y \in B$, that is equipped with associative and unital actions
		\begin{displaymath}
			\map\lambda{A(x_1, x_2) \tens J(x_2, y)}{J(x_1, y)} \qquad \text{and} \qquad \map\rho{J(x,y_1) \tens B(y_1, y_2)}{J(x, y_2)}
		\end{displaymath}
		satisfying the usual compatibility axiom for bimodules; see e.g.\ Section~3 of \cite{Lawvere73}. If $\V$ is closed symmetric monoidal, so that it can be considered as enriched over itself, then $\V$-profunctors $\hmap JAB$ can be identified with $\V$-functors of the form \mbox{$\map J{\op A \tens B}\V$,} where $\op A$ denotes the dual of $A$ (see e.g.\ Section~1.4 of \cite{Kelly82}). In \exref{restrictions of V-profunctors} we will see that $\enProf\V$ has all horizontal units so that by \thmref{unital virtual double categories} it can equivalently be regarded as a virtual double category.
		
		A vertical cell $\cell\phi fg$ in $\enProf\V$, between $\V$-functors $f$ and $\map gAC$, is a $\V$"/natural transformation $f \Rar g$ in the usual sense; see for instance Section~1.2 of \cite{Kelly82}. We conclude that the vertical $2$-category $V(\enProf\V)$ contained in $\enProf\V$ (\exref{vertical 2-category}) equals the $2$-category $\enCat\V$ of $\V$"/categories, $\V$"/functors and $\V$"/natural transformations.
		
		Taking $\V = \Set$ in the above we obtain the augmented virtual double category $\enProf\Set$ of locally small (i.e.\ $\Set$"/enriched) categories, functors, \emph{$\Set$"/profunctors} \mbox{$\map J{\op A \times B}\Set$} and transformations. Likewise $\enProf{\Set'}$ is the augmented virtual double category of categories (possibly with large hom-sets), functors, \emph{$\Set'$"/profunctors} \mbox{$\map J{\op A \times B}\Set'$} and transformations. We will call a $\Set'$"/category $A$ \emph{locally properly large} if $A(x_1, x_2)$ is properly large for some $x_1$, $x_2 \in A$. Likewise a $\Set'$"/profunctor $\hmap JAB$ is \emph{properly large} if $J(x, y)$ is properly large for some $x \in A$ and $y \in B$.
	\end{example}
	
	\begin{example} \label{quantale-enriched profunctors}
		A \emph{quantale} $\V$ (see e.g.\ Section II.1.10 of \cite{Hofmann-Seal-Tholen14}) is a complete lattice equipped with a monoid structure $\tens$ that preserves suprema on both sides. Equivalently, a quantale can be thought of as a thin category $\V$ that is complete (hence cocomplete) and equipped with a closed monoidal structure.
		
		The extended positive real line $\V = \brks{0, \infty}$ for example, equipped with the reversed order $\geq$, forms a quantale whose monoid structure is given by addition $(+, 0)$ while its closed structure is truncated subtraction $\brks{x, y} \dfn \max(y - x, 0)$. Categories enriched in $\brks{0, \infty}$ form Lawvere's paradigmatic example of enriched category theory \cite{Lawvere73}: they can be regarded as \emph{generalised metric spaces}, that is sets $A$ equipped with a (not necessarily symmetric) distance function \mbox{$A \times A \to \brks{0, \infty}$} (which we again denote by $A$). Both vertical morphisms $\map fAC$ and horizontal morphisms $\hmap JAB$ in $\enProf{\brks{0, \infty}}$ are required to be \emph{non"/expanding}, that is $A(x_1, x_2) \geq C(fx_1, fx_2)$ and 
		\begin{displaymath}
			A(x_1, x_2) + J(x_2, y) \geq J(x_1, y) \qquad \text{and} \qquad J(x, y_1) + B(y_1, y_2) \geq J(x, y_2)
		\end{displaymath}
		respectively, for all $x_1, x_2, x \in A$ and $y, y_1, y_2 \in B$.
		
		Notice that, because quantales $\V$ are thin categories, their induced augmented virtual double categories $\enProf\V$ are \emph{locally thin}: any cell in $\enProf\V$ is uniquely determined by its (horizontal and vertical) sources and targets. Locally thin augmented virtual double categories of the form $\enProf\V$, where $\V$ is a quantale, form a natural setting for the study of `monoidal topology' \cite{Hofmann-Seal-Tholen14}, see for instance \cite{Koudenburg18}.
	\end{example}
	
	The following example motivates our choice of augmented virtual double categories as the optimal `double dimensional' environment for classical category theory.
	\begin{example} \label{(Set, Set')-Prof}
		Taking $\V = \Set'$ in \exref{enriched profunctors}, we write $\enProf{(\Set, \Set')}$ for the locally full sub"/augmented virtual double category of $\enProf{\Set'}$ that is generated by $\Set$"/profunctors. In detail: $\enProf{(\Set, \Set')}$ consists of all $\Set'$"/categories and functors, only those profunctors $\hmap JAB$ with $J(x, y) \in \Set$ for all \mbox{$(x, y) \in A \times B$}, and all cells between such $\Set$"/profunctors (including the nullary and vertical cells).
		
		Thus we have a chain of sub"/augmented virtual double categories
		\begin{displaymath}
			 \enProf\Set \subsetneqq \enProf{(\Set, \Set')} \subsetneqq \enProf{\Set'},
		\end{displaymath}
		and we take the view that the classical theory of locally small categories is best considered in $\enProf{(\Set, \Set')}$, motivated as follows. Recall from \cite{Freyd-Street95} that, for a locally small category $A$, the category $\Set^{\op A}$ of presheaves on $A$ is locally small if and only if $A$ is essentially small. Thus, on one hand, presheaves on a locally small category $A$ in general do not form an object in $\enProf\Set$, while they do form one in $\enProf{(\Set, \Set')}$. On the other hand, writing $\map\yon A{\Set^{\op A}}$ for the Yoneda embedding, Yoneda's lemma supplies, for each horizontal morphism $\hmap JAB$ in $\enProf{(\Set, \Set')}$, a functor $\map{\cur J}B{\Set^{\op A}}$ equipped with a natural isomorphism of $\Set$"/profunctors $J \iso \Set^{\op A}(\yon, \cur J)$\footnote{Indeed, take $\cur J(y) \dfn J(\dash, y)$ for $y \in B$.}; of course such a $\cur J$ does not exist for the properly large profunctors $J$ contained in $\enProf{\Set'}$. Thus the objects in $\enProf{(\Set, \Set')}$ are ``large enough'' for it to contain all presheaf categories $\Set^{\op A}$ with $A$ locally small while its horizontal morphisms are ``small enough'' to allow for a simple universal property of the Yoneda embeddings $\map yA{\Set^{\op A}}$, that is given in terms of \emph{all} horizontal morphisms of $\enProf{(\Set, \Set')}$ instead of a certain subclass of ``admissible'' ones, the latter such as in the definition of Yoneda structure \cite{Street-Walters78}; in particular this universal property is straightforward to formalise.
		
		For an example of an advantage of working in the augmented virtual double category $\enProf{(\Set, \Set')}$ rather than in the virtual double category $U\pars{\enProf{(\Set, \Set')}}$ that it contains (\exref{virtual double category of unary cells}) notice that, for any two functors $f$ and $\map gAC$ into a locally properly large category $C$, the natural transformations $\nat\phi fg$ are contained in the former but cannot be considered the latter. Indeed in $\enProf{(\Set, \Set')}$ they exist as vertical cells $\cell\phi fg$, but these are removed when passing to $U\pars{\enProf{(\Set, \Set')}}$. And while such natural transformations correspond to cells in $\enProf{\Set'}$ of the form below, where $I_C$ is the `unit profunctor' given by the hom-sets $I_C(x,y) = C(x,y)$, the properly large profunctor $I_C$ is not contained in $\enProf{(\Set, \Set')}$ (see \exref{restrictions in (V, V')-Prof}) and thus neither in $U\pars{\enProf{(\Set, \Set')}}$.
		\begin{displaymath}
			\begin{tikzpicture}
				\matrix(m)[math35, column sep={1.75em,between origins}]{& A & \\ C & & C \\};
				\path[map]	(m-1-2) edge node[left] {$f$} (m-2-1)
														edge node[right] {$g$} (m-2-3)
										(m-2-1) edge[barred] node[below] {$I_C$} (m-2-3);
				\path				(m-1-2) edge[cell, transform canvas={yshift=-0.25em}] node[right, inner sep=2.5pt] {$\phi$} (m-2-2);
			\end{tikzpicture}
		\end{displaymath}
	\end{example}
	
	Considering $\Set'$"/categories is one way of dealing with the size of the categories $\Set^{\op A}$ of presheaves on locally small categories $A$. Another way is to restrict to `small' presheaves on $A$ instead, as recalled in \exref{small V-profunctors} below. The next example generalises the construction of $\enProf{(\Set, \Set')}$ above to the enriched setting.
	\begin{example} \label{(V, V')-Prof}
	Analogous to the previous example we can consider sub"/augmented virtual double categories $\enProf{(\mathsf{Ab}, \mathsf{Ab}')} \subset \enProf{\mathsf{Ab}'}$, $\enProf{(\Cat, \Cat')} \subset \enProf{\Cat'}$, etc., where $\mathsf{Ab} \subset \mathsf{Ab}'$, $\Cat \subset \Cat'$, etc., are embeddings obtained by considering abelian groups, categories, etc., in both categories of sets $\Set$ and $\Set'$ respectively. Again we prefer to work in e.g.\ $\enProf{(\mathsf{Ab}, \mathsf{Ab}')}$ instead of $\enProf{\mathsf{Ab}}$ or $\enProf{\mathsf{Ab'}}$, for reasons similar to the ones given in the previous example.
	
	More generally we will follow Kelly's approach in Section 3.11 of \cite{Kelly82}, which is based on \cite{Day70}, and enrich both in a monoidal category $\V$ as well as in a `universe enlargement' of $\V$, as follows. A \emph{universe enlargement} of a large (not necessarily closed) monoidal category $\V$ is a monoidal full embedding $\V \subset \V'$ of $\V$ into a closed\footnote{$\V' = (\V', \tens', I')$ is closed if, for every object $x' \in \V'$, the endofunctor $x' \tens \dash$ has a right adjoint $\brks{x', \dash}$; see e.g.\ Section~1.5 of \cite{Kelly82}.} monoidal category $\V'$ that satisfies the following axioms:
	\begin{enumerate}[label=(\alph*)]
		\item $\V'$ is locally large, that is $\V'(x',y') \in \Set'$ for all $x', y' \in \V'$; 
		\item $\V'$ is large complete and large cocomplete;
		\item $\V \subset \V'$ preserves all limits.
	\end{enumerate}
	
	One can show that the embeddings $\Set \subset \Set'$, $\mathsf{Ab} \subset \mathsf{Ab}'$ and $\Cat \subset \Cat'$ are universe enlargements in the above sense, as long as $\Set$ has infinite sets. More generally Kelly shows that the Yoneda embedding $\map\yon \V{{\Set'}^{\op\V}}$ defines the category ${\Set'}^{\op\V}$ of $\Set'$"/presheaves on $\V$ as a universe enlargement of $\V$, with the monoidal structure $\tens'$ on the category ${\Set'}^{\op\V}$ given by `Day convolution' \cite{Day70} (or see \eqref{Day convolution} above). If $\V$ is closed monoidal then, besides (a)---(c) above, the Yoneda embedding $\map\yon \V{{\Set'}^{\op\V}}$ also is a closed monoidal embedding, that is $\yon(\brks{x, y}) \iso \brks{\yon x, \yon y}'$ coherently for all $x$, $y \in \V$. In that case, as is shown in Section~3.12 of \cite{Kelly82}, the factorisation of $\yon$ through the full subcategory \mbox{$\V' \subset \Set'^{\op \V}$} of $\Set'$"/presheaves that preserve all large limits in $\op \V$ is a universe enlargement $\V \subset \V'$ that, besides preserving all limits, preserves large colimits as well.
	
	Returning to a universe enlargement $\V \subset \V'$ with $\V$ not necessarily closed, consider a $\V'$-profunctor $\hmap JAB$ in $\enProf{\V'}$ (see \exref{enriched profunctors}). We will call $J$ a \emph{$\V$-profunctor} whenever $J(x, y)$ is a $\V$-object for all pairs $x \in A$ and $y \in B$. Analogous to the definition of $\enProf{(\Set, \Set')}$ in the previous example we denote by $\enProf{(\V, \V')}$ the sub"/augmented virtual double category of $\enProf{\V'}$ that consists of all $\V'$-categories and $\V'$-functors, as well as $\V$-profunctors and their transformations.
	\end{example}
		
	In the next example we recall the notion of `small $\V$"/profunctor' and show that such profunctors form a sub"/augmented virtual double category of $\enProf\V$ (\exref{enriched profunctors}). In doing so we use the classical coend formula that defines compositions of $\V$"/profunctors, which we first recall briefly. Let \mbox{$\ul J = (A_0 \xbrar{J_1} A_1, \dotsc, A_{n'} \xbrar{J_n} A_n)$} be a non"/empty path of $\V$"/profunctors and let $x \in A_0$ and $y \in A_n$ be objects. Inspired by Mac Lane's construction of ends as limits in Section~IX.5 of \cite{MacLane98} we consider the following diagram functor $\map{\ul J^\S(x,y)}{\ul J^\textup S}\V$. The objects of $J^\textup S$ are of two kinds: they are either $n'$"/tuples of pairs $\bigpars{(v_1, w_1), (v_2, w_2), \dotsc, (v_{n'}, w_{n'})}$, with each pair $(v_i, w_i)$ objects in $A_i$, or they are $n'$"/tuples $(u_1, u_2, \dotsc, u_n)$ of objects $u_i \in A_i$. The non"/identity morphisms of $J^\textup S$ are the legs of spans of the form
	\begin{displaymath}
		(v_1, v_2, \dotsc, v_{n'}) \leftarrow \bigpars{(v_1, w_1), (v_2, w_2), \dotsc, (v_{n'}, w_{n'})} \rightarrow (w_1, w_2, \dotsc, w_{n'});
	\end{displaymath}
	consequently in any composable pair of morphisms in $J^\textup S$ either morphism necessarily is an identity. Having defined $J^\textup S$ we next denote by $\ul J^\S(x, y)$ the diagram $J^\textup S \rightarrow \V$ that maps each span above to the following span in $\V$.
		\begin{displaymath}
			\begin{tikzpicture}[textbaseline]
				\matrix(m)[minimath, column sep={11em,between origins}]
					{	& J_1(x, v_1) \tens A_1(v_1, w_1) \tens J_2(w_1, v_2) \tens A_2(v_2, w_2) \tens \dotsb \tens A_n(v_{n'}, w_{n'}) \tens J_n(w_{n'}, y) & \\
						J_1(x, v_1) \tens J_2(v_1, v_2) \tens \dotsb \tens J_n(v_{n'}, y) & & J_1(x, w_1) \tens J_2(w_1, w_2) \tens \dotsb \tens J_n(w_{n'}, y) \\ };
				\path[map]	(m-1-2) edge node[left, inner sep=15pt] {$\id \tens \lambda \tens \dotsb \tens \lambda$} (m-2-1)
														edge node[right, inner sep=20pt] {$\rho \tens \dotsb \tens \rho \tens \id$} (m-2-3);
			\end{tikzpicture}
		\end{displaymath}
		In the case $\ul J = (J_1, J_2)$ the colimits of $(J_1, J_2)^\S(x, y)$, if they exist for all \mbox{$x \in A_0$} and $y \in A_2$, combine to form the composite $\V$"/profunctor `$\hmap{J_2 \of J_1}{A_0}{A_2}$' as defined in Section~3 of \cite{Lawvere73}. For general $\ul J$, if $\V$ is closed symmetric monoidal, so that each $\hmap{J_i}{A_{i'}}{A_i}$ can be identified with a $\V$-functor $\map{J_i}{\op{A_{i'}} \tens A_i}{\V}$, then the colimit of $\ul J^\S(x, y)$ is easily checked to coincide with the iterated \emph{coend} on the left"/hand side below; for the definition of the dual notion `end' see e.g.\ Section~2.1 of \cite{Kelly82}. We will use the coend notation
	\begin{displaymath}
		\intl^{u_1 \in A_1} \dotsb \intl^{u_{n'} \in A_{n'}} J_1(x, u_1) \tens \dotsb \tens J_n(u_{n'}, y) \dfn \colim \ul J^\S(x,y)
	\end{displaymath}
	for the colimit of $\ul J^\S(x, y)$ regardless of whether the monoidal category $\V$ is closed symmetric. If $\tens$ preserves large colimits on both sides then the coends above, if they exist for all $x \in A_0$ and $x \in A_n$, combine into a $\V$"/profunctor $A_0 \brar A_n$.
	\begin{example} \label{small V-profunctors}
	  Let $\V = (\V, \tens, I)$ be a monoidal category such that $v \tens \dash$ preserves large colimits for each $v \in \V$. A $\V$"/profunctor $\hmap JAB$ in $\enProf\V$ (\exref{enriched profunctors}) is called \emph{small} if for each object $y \in B$ there exists a small sub"/$\V$"/category $A_y \subseteq A$ such that the coends below exists together with isomorphisms
	  \begin{displaymath}
	    J(x, y) \iso \intl^{x' \in A_y} A(x, x') \tens J(x', y)
	  \end{displaymath}
	  that are equivariant in $x \in A$ (\exref{monoids and bimodules}). For example if $\V = \Set$ and $A$ is any large set seen as a discrete category, then a $\Set$"/profunctor $\hmap JAB$, with $B$ any category, is small precisely if for each $y \in B$ the set
	  \begin{displaymath}
	    \set{x \in A \mid J(x, y) \neq \emptyset}
	  \end{displaymath}
	  is small, which in that case we can take as $A_y$. In general notice that any $\V$"/profunctor $\hmap JAB$ is small whenever $A$ is a small $\V$"/category.
	  
	  We denote by $\ensProf\V \subseteq \enProf\V$ the sub"/augmented virtual double category consisting of all $\V$"/categories, all $\V$"/functors, only small $\V$"/profunctors, and all cells between them (including the nullary and vertical ones). We will see in \exref{restrictions of small profunctors} that $\ensProf\V$ has horizontal units and, in \exref{composites of small profunctors}, that, unlike $\enProf\V$ and $\enProf{(\V, \V')}$, it has all horizontal composites (see \secref{composition section}) whenever $\V$ is small cocomplete such that its monoidal product $\tens$ preserves large colimits on both sides. Thus in that case $\ensProf\V$ is a pseudo double category in the sense of \cite{Grandis-Pare99} (or see \secref{composition section} below).
	  
	  To see that, when $\V$ is closed symmetric monoidal, the above notion agrees with the usual notion of smallness for $\V$"/profunctors notice that, by equation~(4.25) of \cite{Kelly82}, for each $y \in B$ the isomorphisms above exhibit the $\V$"/presheaf \mbox{$\map{J(\dash, y)}{\op A}\V$} as the left Kan extension of $J(\dash, y)$ along the inclusion $A_y \subseteq A$. Hence each $J(\dash, y)$ is an `accessible' $\V$"/presheaf in the sense of Proposition~4.83 of \cite{Kelly82}; more recently (e.g.\ \cite{Day-Lack07}) such $\V$"/presheaves have been termed \emph{small}. Assuming that the $\V$"/category $\brks{B, \V}$ of $\V$"/functors $B \to \V$ exists (see Section~2 of \cite{Kelly82}), it follows that $\hmap JAB$ is small in the above sense precisely if the corresponding $\V$"/functor $\op A \to \brks{B, \V}$ is `pointwise small' in the sense of \cite{Day-Lack07}.
	\end{example}
	
	\begin{example} \label{internal profunctors}
	  Let $\E$ be a category with pullbacks. The augmented virtual double category $\Span\E$ of \emph{spans} in $\E$ has as objects and vertical morphisms the objects and morphisms of $\E$, while its horizontal morphisms $\hmap JAB$ are spans \mbox{$A \leftarrow J \rightarrow B$} in $\E$. A unary cell $\phi$ in $\Span\E$, as on the left below, is a morphism \mbox{$\map \phi{J_1 \times_{A_1} \dotsb \times_{A_{n'}} J_n}K$} in $\E$ lying over $f$ and $g$, where the wide pullback is taken to be $A_0$ if the horizontal source of $\phi$ is empty. Nullary cells in $\Span\E$ on the other hand are uniquely determined by their boundary: a cell $\psi$ as in the middle exists precisely if the square on the right commutes.
 		\begin{displaymath}
		  \begin{tikzpicture}[baseline]
				\matrix(m)[math35]{A_0 & A_n \\ C & D \\};
				\path[map]	(m-1-1) edge[barred] node[above] {$\ul J$} (m-1-2)
														edge node[left] {$f$} (m-2-1)
										(m-1-2) edge node[right] {$g$} (m-2-2)
										(m-2-1) edge[barred] node[below] {$K$} (m-2-2);
				\path[transform canvas={xshift=1.75em}]	(m-1-1) edge[cell] node[right] {$\phi$} (m-2-1);
			\end{tikzpicture} \qquad\qquad\qquad \begin{tikzpicture}[baseline]
				\matrix(m)[math35, column sep={1.75em,between origins}]
					{A_0 & & A_n \\ & C & \\};
				\path[map]	(m-1-1) edge[barred] node[above] {$\ul J$} (m-1-3)
														edge node[left] {$f$} (m-2-2)
										(m-1-3) edge node[right] {$g$} (m-2-2);
				\path				(m-1-2) edge[cell, transform canvas={yshift=0.25em}] node[right, inner sep=2pt] {$\psi$} (m-2-2);
			\end{tikzpicture} \qquad\qquad\qquad \begin{tikzpicture}[baseline]
				\matrix(m)[math35]
					{ & J_1 \times_{A_1} \dotsb \times_{A_{n'}} J_n & \\ A_0 & & A_n \\ & C & \\};
				\path[map]	(m-1-2) edge (m-2-1)
														edge (m-2-3)
										(m-2-1) edge node[below left] {$f$} (m-3-2)
										(m-2-3) edge node[below right] {$g$} (m-3-2);
			\end{tikzpicture}
		\end{displaymath}
		
		The virtual double category $U(\Span\E)$ (\exref{virtual double category of unary cells}) contained in $\Span\E$ is the same as that considered in Example~2.7 of \cite{Cruttwell-Shulman10}. The augmented virtual double category $\inProf\E \dfn (N \of \Mod)(U(\Span\E))$ of monoids and bimodules in $U(\Span\E)$ is that of \emph{internal} categories, functors, profunctors and transformations in $\E$. The vertical $2$"/category $V(\inProf\E)$ contained inside $\inProf\E$ (\exref{vertical 2-category}) is the $2$"/category $\inCat\E$ of internal categories, functors and transformations in $\E$; the latter in the classical sense of \cite{Street74}.
		
		We will see that $\Span\E$ has all horizontal units (\exref{restrictions of spans}) and composites (\exref{Span(E) has composites}), so that it can be equivalently regarded as a pseudo double category (see \secref{composition section}).
	\end{example}
	
	\begin{example} \label{internal relations}
	  As in the previous example let $\E$ be a category with pullbacks. A span $A \xlar{j_0} J \xrar{j_1} B$ in $\E$ is called a \emph{relation} (see e.g.\ \cite{Carboni-Kasangian-Street84}) if any two horizontal cells $\cell{\phi, \psi}HJ$ in $\Span\E$ are equal, that is $j_0$ and $j_1$ are jointly monic. We denote by $\Rel(\E) \subseteq \Span\E$ the sub"/augmented virtual double category generated by the relations in $\E$. Like $\Span\E$, $\Rel(\E)$ has all horizontal units (see \exref{restrictions of spans}), so that it can be equivalently regarded as a virtual double category by \thmref{unital virtual double categories}.  Notice that $\Rel(\E)$ is a locally thin augmented virtual double category in the sense of \exref{quantale-enriched profunctors}.
	  
	  We remark that in order to be able to arrange relations in $\E$ into a bicategory or a pseudo double category (see \cite{Grandis-Pare99} or \secref{composition section} below) one needs $\E$ to be \emph{regular} (see e.g.\ \cite{Carboni-Kasangian-Street84}); in constrast, to form $\Rel(\E)$ as an (augmented) virtual double category it suffices that $\E$ has pullbacks.
	\end{example}	
	
	\begin{example} \label{internal split fibrations}
	  Let $\K$ be a finitely complete $2$"/category, that is $\K$ has all finite conical limits as well as cotensors with the ``walking arrow'' category $\2 \dfn (0 \rightarrow 1)$. `Split bifibrations' in $\K$, introduced in \cite{Street74} and recently called \emph{split two"/sided fibrations}, can be regarded as profunctors internal to $\K_0$, the category underlying $\K$, as follows. For $A \in \K$ the cotensor $\Phi A \dfn \brks{\2, A}$ is defined by a cell
	  \begin{displaymath}
			\begin{tikzpicture}[textbaseline]
						\matrix(m)[math35]{\Phi A \\ A \\};
						\path[map]	(m-1-1) edge[bend right=45] node[left] {$d_0$} (m-2-1)
																edge[bend left=45] node[right] {$d_1$} (m-2-1);
						\path				(m-1-1) edge[cell] (m-2-1);
			\end{tikzpicture}
		\end{displaymath}
		whose universal property induces on the span $A \xlar{d_0} \Phi A \xrar{d_1} A$ the structure of a category internal to $\K_0$. In fact, Proposition~2 of \cite{Street74} shows that choosing a cotensor $\Phi A$ for each $A \in \K$ induces a functor $\map\Phi{\K_0}{\inCat{\K_0} = V(\inProf{\K_0})}$ (\exref{internal profunctors}). Given objects $A$, $B \in \K$, an internal profunctor $\hmap J{\Phi A}{\Phi B}$ in $\inProf{K_0}$ is precisely a `split bifibration' $A \leftarrow J \rightarrow B$ in $\K$ in the sense of \cite{Street74}, which follows easily from Proposition~12 therein. Likewise horizontal cells $\cell\phi JK$ in $\inProf{K_0}$, with $\hmap K{\Phi A}{\Phi B}$, are morphisms of bifibrations in the sense of \cite{Street74}.
		
		In light of the above we denote by $\spFib\K$ the augmented virtual double category whose objects and vertical morphisms are those of $\K_0$, and whose horizontal morphisms $\hmap JAB$ are profunctors $\hmap J{\Phi A}{\Phi B}$ internal to $\K_0$. Cells in $\spFib\K$, with vertical source $\map f{A_0}C$ and target $\map g{A_n}D$, are cells in $\inProf{\K_0}$ with vertical source $\map{\Phi f}{\Phi A_0}{\Phi C}$ and target $\map {\Phi g}{\Phi A_n}{\Phi D}$; their compositions are defined as in $\inProf{\K_0}$. 
	\end{example}
		
	\section{The \texorpdfstring{$2$}{2}-category of augmented virtual double categories} \label{2-category of augmented virtual double categories section}
	Having introduced the notion of augmented virtual double categories next we consider the functors between them, as well as their transformations.
	\begin{definition}
		A \emph{functor} $\map F\K\L$ between augmented virtual double categories consists of a functor $\map F{\K_\textup v}{\L_\textup v}$ as well as assignments mapping the horizontal morphisms and cells of $\K$ to those of $\L$, as shown below, in a way that preserves vertical composition and identity cells strictly.
		\begin{align*}
			\hmap JAB \quad\qquad&\mapsto\quad\qquad \hmap{FJ}{FA}{FB} \\
			\begin{tikzpicture}[textbaseline, ampersand replacement=\&]
				\matrix(m)[math35]{A_0 \& A_1 \& A_{n'} \& A_n \\ C \& \& \& D \\};
				\path[map]	(m-1-1) edge[barred] node[above] {$J_1$} (m-1-2)
														edge node[left] {$f$} (m-2-1)
										(m-1-3) edge[barred] node[above] {$J_n$} (m-1-4)
										(m-1-4) edge node[right] {$g$} (m-2-4)
										(m-2-1) edge[barred] node[below] {$K$} (m-2-4);
				\path[transform canvas={xshift=1.75em}]	(m-1-2) edge[cell] node[right] {$\phi$} (m-2-2);
				\draw				($(m-1-2)!0.5!(m-1-3)$) node {$\dotsb$};
			\end{tikzpicture} \qquad&\mapsto\qquad \begin{tikzpicture}[textbaseline, ampersand replacement=\&]
				\matrix(m)[math35]{FA_0 \& FA_1 \& FA_{n'} \& FA_n \\ FC \& \& \& FD \\};
				\path[map]	(m-1-1) edge[barred] node[above] {$FJ_1$} (m-1-2)
														edge node[left] {$Ff$} (m-2-1)
										(m-1-3) edge[barred] node[above] {$FJ_n$} (m-1-4)
										(m-1-4) edge node[right] {$Fg$} (m-2-4)
										(m-2-1) edge[barred] node[below] {$FK$} (m-2-4);
				\path[transform canvas={xshift=1.75em}]	(m-1-2) edge[cell] node[right] {$F\phi$} (m-2-2);
				\draw				($(m-1-2)!0.5!(m-1-3)$) node {$\dotsb$};
			\end{tikzpicture} \\
			\begin{tikzpicture}[textbaseline, ampersand replacement=\&]
					\matrix(m)[math35, column sep={1.75em,between origins}]
						{A_0 \& \& A_1 \& \dotsb \& A_{n'} \& \& A_n \\ \& \& \& C \& \& \& \\};
					\path[map]	(m-1-1) edge[barred] node[above] {$J_1$} (m-1-3)
															edge node[below left] {$f$} (m-2-4)
											(m-1-5) edge[barred] node[above] {$J_n$} (m-1-7)
											(m-1-7) edge node[below right] {$g$} (m-2-4);
					\path				(m-1-4) edge[cell] node[right] {$\phi$} (m-2-4);
				\end{tikzpicture} \qquad&\mapsto\qquad \begin{tikzpicture}[textbaseline, ampersand replacement=\&]
					\matrix(m)[math35, column sep={1.75em,between origins}]
						{FA_0 \& \& FA_1 \& \dotsb \& FA_{n'} \& \& FA_n \\ \& \& \& FC \& \& \& \\};
					\path[map]	(m-1-1) edge[barred] node[above] {$FJ_1$} (m-1-3)
															edge node[below left] {$Ff$} (m-2-4)
											(m-1-5) edge[barred] node[above] {$FJ_n$} (m-1-7)
											(m-1-7) edge node[below right] {$Fg$} (m-2-4);
					\path				(m-1-4) edge[cell] node[right] {$F\phi$} (m-2-4);
				\end{tikzpicture}
		\end{align*}
	\end{definition}
	Notice that $\map F\K\L$ preserving vertical composition $\of$ implies that $F$ preserves horizontal composition $\hc$ (see \lemref{horizontal composition}).
	\begin{definition} \label{transformation}
		A \emph{transformation} $\nat\xi FG$ of functors $F$, $\map G\K\L$ of augmented virtual double categories consists of a natural transformation $\nat\xi{F_\textup v}{G_\textup v}$ as well as a family of $(1,1)$-ary cells
		\begin{displaymath}
			\begin{tikzpicture}
				\matrix(m)[math35]{FA & FB \\ GA & GB \\};
				\path[map]	(m-1-1) edge[barred] node[above] {$FJ$} (m-1-2)
														edge node[left] {$\xi_A$} (m-2-1)
										(m-1-2) edge node[right] {$\xi_B$} (m-2-2)
										(m-2-1) edge[barred] node[below] {$GJ$} (m-2-2);
				\path[transform canvas={xshift=1.75em}]	(m-1-1) edge[cell] node[right] {$\xi_J$} (m-2-1);
			\end{tikzpicture}
		\end{displaymath}
		in $\L$, one for each $\hmap JAB \in \K$, that satisfies the \emph{naturality axiom}
		\begin{displaymath}
			G\phi \of \xi_{\ul J} = \xi_{\ul K} \of F\phi
		\end{displaymath}
		for all cells $\cell\phi{\ul J}{\ul K}$ in $\K$, where $\xi_{\ul J} \dfn (\xi_{J_1}, \dotsc, \xi_{J_n})$ if $\ul J = (J_1, \dotsc, J_n)$ and $\xi_{\ul J} \dfn \xi_A$ if $\ul J = (A)$.
	\end{definition}
	
	In \exref{virtual double category of unary cells} we saw that restricting to augmented virtual double categories with only vertical identity cells as nullary cells recovers the notion of virtual double category. Likewise, under this restriction the definitions above reduce to that of functor and transformation for virtual double categories as given in Section 3 of \cite{Cruttwell-Shulman10}. The latter combine into a $2$-category of virtual double categories which we denote $\VirtDblCat$.	Remember that every augmented virtual double category $\K$ contains a $2$"/category $V(\K)$ (\exref{vertical 2-category}) and a virtual double category $U(\K)$ (\exref{virtual double category of unary cells}). In the following proposition, which is easily checked, $\twoCat$ denotes the $2$"/category of $2$"/categories, strict $2$"/functors and $2$"/natural transformations.
	\begin{proposition} \label{2-category of augmented virtual double categories}
		Augmented virtual double categories, the functors between them and their transformations form a $2$-category $\AugVirtDblCat$. Both the assignments \mbox{$\K \mapsto V(\K)$} and $\K \mapsto U(\K)$ extend to strict $2$-functors
		\begin{displaymath}
			\map V\AugVirtDblCat\twoCat \qquad \text{and} \qquad \map U\AugVirtDblCat\VirtDblCat. 
		\end{displaymath}
	\end{proposition}
	
	\begin{example}
		Every lax monoidal functor $\map F\V{\catvar W}$ between monoidal categories induces a functor $\map{\Mat F}{\Mat\V}{\Mat{\catvar W}}$ between the virtual double categories of matrices in $\V$ and $\catvar W$ (\exref{enriched profunctors}) in the evident way. Likewise the components of any monoidal transformation $\nat\xi FG$ form the cell-components of an induced transformation $\nat{\Mat\xi}{\Mat F}{\Mat G}$. The assignments $F \mapsto \Mat F$ and \mbox{$\xi \mapsto \Mat\xi$} combine to form a strict $2$-functor $\map{\Mat{(\dash)}}{\MonCat_\textup l}{\VirtDblCat}$, where $\MonCat_\textup l$ denotes the $2$-category of monoidal categories, lax monoidal functors and monoidal transformations.
	\end{example}
	
	\begin{example} \label{Span is a 2-functor}
		Similarly any pullback-preserving functor $\map F\D\E$, between categories with pullbacks, induces a functor $\map{\Span F}{\Span\D}{\Span\E}$ between the augmented virtual double categories of spans in $\D$ and $\E$ (see \exref{internal profunctors}). This too extends to a strict $2$-functor $\map{\Span{\dash}}{\Cat_\textup{pb}}{\AugVirtDblCat}$, where $\Cat_\textup{pb}$ denotes the $2$"/category of categories with pullbacks, pullback-preserving functors and all natural transformations between them.
	\end{example}
		
	By an equivalence of augmented virtual double categories we, of course, mean an internal equivalence in the $2$"/category $\AugVirtDblCat$. The goal of the remainder of this section is to prove that, like in classical category theory (see e.g.\ Section~IV.4 of \cite{MacLane98}), giving an equivalence $\K \simeq \L$ of augmented virtual double categories is the same as giving a functor $\map F\K\L$ that is `full, faithful and essentially surjective'. The following definitions generalise analogous definitions for functors between double categories given in Section~7 of \cite{Shulman08}.
	
	We start with the notions full and faithful. Let $\map F\K\L$ be a functor between augmented virtual double categories. Its restriction $J \mapsto FJ$ to horizontal morphisms preserves sources and targets, so that it extends to an assignment
	\begin{displaymath}
		\ul J = (J_1, \dotsc, J_n) \mapsto F\ul J \dfn (FJ_1, \dotsc, FJ_n)	
	\end{displaymath}
	on paths. Likewise, for any pair of morphisms $\map f{A_0}C$ and $\map g{A_n}D$ in $\K$, together with paths $\hmap{\ul J}{A_0}{A_n}$ and $\hmap{\ul K}CD$ where $\lns K \leq 1$, the functor $F$ restricts to an assignment below, between classes of cells with sources and targets as shown.
	\begin{displaymath}
		\Bigl\lbrace\begin{tikzpicture}[textbaseline]
			\matrix(m)[math35]{A_0 & A_n \\ C & D \\};
			\path[map]	(m-1-1) edge[barred] node[above] {$\ul J$} (m-1-2)
													edge node[left] {$f$} (m-2-1)
									(m-1-2) edge node[right] {$g$} (m-2-2)
									(m-2-1) edge[barred] node[below] {$\ul K$} (m-2-2);
			\path[transform canvas={xshift=1.75em}]	(m-1-1) edge[cell] node[right] {$\phi$} (m-2-1);
		\end{tikzpicture}\Bigr\rbrace \qquad \xrar F \qquad \Bigl\lbrace\begin{tikzpicture}[textbaseline]
			\matrix(m)[math35]{FA_0 & FA_n \\ FC & FD \\};
			\path[map]	(m-1-1) edge[barred] node[above] {$F\ul J$} (m-1-2)
													edge node[left] {$Ff$} (m-2-1)
									(m-1-2) edge node[right] {$Fg$} (m-2-2)
									(m-2-1) edge[barred] node[below] {$F\ul K$} (m-2-2);
			\path[transform canvas={xshift=1.75em}]	(m-1-1) edge[cell] node[right] {$\psi$} (m-2-1);
		\end{tikzpicture}\Bigr\rbrace
	\end{displaymath}
	\begin{definition} \label{full and faithful functor}
		A functor $\map F\K\L$ between augmented virtual double categories is called \emph{locally faithful} (resp.\ \emph{locally full}) if, for any $\map f{A_0}C$, $\map g{A_n}D$, \mbox{$\hmap{\ul J}{A_0}{A_n}$} and $\hmap{\ul K}CD$ with $\lns K \leq 1$ in $\K$, the assignment above is injective (resp.\ surjective). If moreover the restriction $\map F{\K_\textup v}{\L_\textup v}$, to the vertical categories, is faithful (resp.\ full), then $F$ is called \emph{faithful} (resp.\ \emph{full}).
	\end{definition}
	
	The following is Definition 7.6 of \cite{Shulman08} applied verbatim to the setting of augmented virtual double categories.
	\begin{definition} \label{essentially surjective}
		A functor $\map F\K\L$ of augmented virtual double categories is called \emph{essentially surjective} if we can simultaneously make the following choices:
		\begin{enumerate}[label=-]
			\item for each object $A \in \L$, an object $A' \in \K$ and an isomorphism $\sigma_A\colon FA' \iso A$;
			\item for each horizontal morphism $\hmap JAB$ in $\L$, a morphism $\hmap{J'}{A'}{B'}$ in $\K$ and an invertible cell
			\begin{displaymath}
				\begin{tikzpicture}
					\matrix(m)[math35]{FA' & FB' \\ A & B. \\};
					\path[map]	(m-1-1) edge[barred] node[above] {$FJ'$} (m-1-2)
															edge node[left] {$\sigma_A$} (m-2-1)
											(m-1-2) edge node[right] {$\sigma_B$} (m-2-2)
											(m-2-1) edge[barred] node[below] {$J$} (m-2-2);
					\path[transform canvas={xshift=1.75em}]	(m-1-1) edge[cell] node[right] {$\sigma_J$} (m-2-1);
				\end{tikzpicture}
			\end{displaymath}
		\end{enumerate}
	\end{definition}
	
	\begin{proposition} \label{equivalences}
		A functor $\map F\K\L$ between augmented virtual double categories is part of an equivalence $\K \simeq \L$ if and only if it is full, faithful and essentially surjective.
	\end{proposition}
	\begin{proof}[(sketch)]
		The `only if'-part is straightforward; we will sketch the `if'"/part. First, because $F$ is essentially surjective, we can choose objects $A' \in \K$, for each $A \in \L$, and horizontal morphisms $\hmap{J'}{A'}{B'} \in \K$, for each $\hmap JAB \in \L$, as in the definition above, together with isomorphisms $\sigma_A\colon FA' \iso A$ and $\sigma_J\colon FJ' \iso J$. Using the full and faithfulness of $F$ these choices can be extended to a functor $\map{(\dash)'}\L\K$ as follows: for each vertical morphism $\map fAC$ in $\L$ we define $\map{f'}{A'}{C'}$ to be the unique map in $\K$ such that $Ff = \inv\sigma_C \of f \of \sigma_A$, and for each cell $\cell\phi{\ul J}{\ul K}$ in $\L$ we define $\cell{\phi'}{\ul J'}{\ul K'}$ to be the unique cell in $\K$ such that $F\phi' = \inv\sigma_{\ul K} \of \phi \of \sigma_{\ul J}$, where the notation $\sigma_{\ul J}$ is as in \defref{transformation}. Using that $F$ is faithful it is easily checked that these assignments preserve the composition and identities of $\L$.
		
		Finally the isomorphisms $(\sigma_A)_{A \in \L}$ and $(\sigma_J)_{J \in \L}$ combine to form a transformation $\sigma \colon F \of (\dash)' \iso \id_\L$. Conversely, a transformation $\eta\colon\id_K\iso (\dash)' \of F$ is obtained by defining $\eta_A$, where $A \in \K$, to be unique with $F\eta_A = \inv\sigma_{FA}$ and defining $\eta_J$, where $\hmap JAB$ in $\K$, such that $F\eta_J = \inv\sigma_{FJ}$.
	\end{proof}
	
	\section{Restriction of horizontal morphisms}\label{restriction section}
	In this section we consider the restriction of horizontal morphisms along vertical morphisms, a construction that is often used in the study of formal category theory internal to (generalised) double categories. Restrictions of horizontal morphisms are defined by `cartesian cells' as in the following definition, which generalises the notions of $(1,1)$"/ary cartesian cell considered in Section~7 of \cite{Cruttwell-Shulman10}, for virtual double categories, and in Section~4 of \cite{Shulman08}, for double categories, to $(n, m)$"/ary cartesian cells where $n, m \leq 1$.
	
	\begin{definition} \label{cartesian cells}
		A cell $\cell\psi{\ul J}{\ul K}$ with $\lns{\ul J} \leq 1$, as in the right-hand side below, is called \emph{cartesian} if any cell $\chi$, as on the left-hand side, factors uniquely through $\psi$ as a cell $\phi$ as shown.
		\begin{displaymath}
			\begin{tikzpicture}[textbaseline]
				\matrix(m)[math35]{X_0 & X_1 & X_{n'} & X_n \\ A & & & B \\ C & & & D \\};
				\path[map]	(m-1-1) edge[barred] node[above] {$H_1$} (m-1-2)
														edge node[left] {$h$} (m-2-1)
										(m-1-3) edge[barred] node[above] {$H_n$} (m-1-4)
										(m-1-4) edge node[right] {$k$} (m-2-4)
										(m-2-1) edge node[left] {$f$} (m-3-1)
										(m-2-4) edge node[right] {$g$} (m-3-4)
										(m-3-1) edge[barred] node[below] {$\ul K$} (m-3-4);
				\draw				($(m-1-2)!0.5!(m-1-3)$) node {$\dotsc$};
				\path[transform canvas={yshift=-1.625em}]	($(m-1-1.south)!0.5!(m-1-4.south)$) edge[cell] node[right] {$\chi$} ($(m-2-1.north)!0.5!(m-2-4.north)$);
			\end{tikzpicture} \quad = \quad \begin{tikzpicture}[textbaseline]
				\matrix(m)[math35]{X_0 & X_1 & X_{n'} & X_n \\ A & & & B \\ C & & & D \\};
				\path[map]	(m-1-1) edge[barred] node[above] {$H_1$} (m-1-2)
														edge node[left] {$h$} (m-2-1)
										(m-1-3) edge[barred] node[above] {$H_n$} (m-1-4)
										(m-1-4) edge node[right] {$k$} (m-2-4)
										(m-2-1) edge[barred] node[below] {$\ul J$} (m-2-4)
														edge node[left] {$f$} (m-3-1)
										(m-2-4) edge node[right] {$g$} (m-3-4)
										(m-3-1) edge[barred] node[below] {$\ul K$} (m-3-4);
				\draw				($(m-1-2)!0.5!(m-1-3)$) node {$\dotsc$};
				\path				($(m-1-1.south)!0.5!(m-1-4.south)$) edge[cell] node[right] {$\phi$} ($(m-2-1.north)!0.5!(m-2-4.north)$)
										($(m-2-1.south)!0.5!(m-2-4.south)$) edge[cell, transform canvas={yshift=-2pt}] node[right] {$\psi$} ($(m-3-1.north)!0.5!(m-3-4.north)$);
			\end{tikzpicture}
		\end{displaymath}
		Vertically dual, provided that $\lns{\ul J} = 1$ the cell $\phi$ is called \emph{weakly cocartesian}\footnote{The stronger notion of \emph{cocartesian} cell will be defined in \defref{cocartesian paths}.} if any cell $\chi$ factors uniquely through $\phi$ as a cell $\psi$ as shown.
	\end{definition}
	
	If a $(1,n)$-ary cartesian cell $\psi$ of the form above exists then its horizontal source \mbox{$\hmap JAB$} is called the \emph{restriction} of $\hmap{\ul K}CD$ along $f$ and $g$, and denoted $\ul K(f, g) \dfn J$. If $\ul K = (C \xbrar K D)$ then we call $K(f, g)$ \emph{unary}; in the case that $\ul K = (C)$ we call $C(f, g)$ \emph{nullary}. Restrictions of the form $\ul K(f, \id)$ and $\ul K(\id, g)$ are called restrictions \emph{on the left} and \emph{right}. We will call the nullary restriction \mbox{$\hmap{C(\id, \id)}CC$} the \emph{(horizontal) unit} of the object $C$ and denote it $I_C \dfn C(\id, \id)$; if $I_C$ exists then we call $C$ \emph{unital}. In \secref{composition section} below we will see that the horizontal morphims $I_C$ form the units for composition of horizontal morphisms. In \thmref{unital virtual double categories} we will see that the notions of an augmented virtual double category with all horizontal units is equivalent to that of a virtual double category with all horizontal units; the latter in the sense of Section~5 of \cite{Cruttwell-Shulman10}. Consequently we call an (augmented) virtual double category a \emph{unital virtual double category} whenever it has all horizontal units. Recall that, as motivated in the Introduction, an advantage of taking augmented virtual double categories as a setting for the formalisation of Yoneda embeddings, rather than $2$"/categories, is that the ``built"/in'' notion of unital object can be taken to replace the notion of `admissible' object as defined by a Yoneda structure~\cite{Street-Walters78}.
	
	By their universal property any two cartesian cells defining the same restriction factor through each other as invertible horizontal cells. We will often not name cartesian cells, but simply depict them as on the left below.
	\begin{displaymath}
		\begin{tikzpicture}
			\matrix(m)[math35]{A & B \\ C & D \\};
			\path[map]	(m-1-1) edge[barred] node[above] {$\ul J$} (m-1-2)
													edge node[left] {$f$} (m-2-1)
									(m-1-2) edge node[right] {$g$} (m-2-2)
									(m-2-1) edge[barred] node[below] {$\ul K$} (m-2-2);
			\draw				($(m-1-1)!0.5!(m-2-2)$) node[font=\scriptsize] {$\cart$};
		\end{tikzpicture} \qquad\qquad\qquad\qquad\qquad\qquad \begin{tikzpicture}
			\matrix(m)[math35]{X_0 & X_n \\ A & B \\};
			\path[map]	(m-1-1) edge[barred] node[above] {$\ul H$} (m-1-2)
													edge node[left] {$h$} (m-2-1)
									(m-1-2) edge node[right] {$k$} (m-2-2)
									(m-2-1) edge[barred] node[below] {$J$} (m-2-2);
			\draw				($(m-1-1)!0.5!(m-2-2)$) node[font=\scriptsize] {$\cocart$};
		\end{tikzpicture}
	\end{displaymath}
	
	If the weakly cocartesian cell on the right above exists then we call its horizontal target $J$ the \emph{extension} of $\ul H$ along $h$ and $k$. Like restrictions, extensions are unique up to isomorphism. When considered in a virtual double category, by restricting the factorisations of \defref{cartesian cells} to unary cells $\chi$, our notion of cocartesian cell coincides with that of weakly cocartesian cell considered in Remark 5.8 of \cite{Cruttwell-Shulman10}.	We shall see in \cororef{extensions and composites} that the extension of $\ul H$ along $h$ and $k$ above coincides with the `horizontal composite' $(A(\id, h) \hc H_1 \hc \dotsb \hc H_n \hc B(k, \id))$ whenever it exists, where \mbox{$\hmap{A(\id, h)}A{X_0}$} and $\hmap{B(k, \id)}{X_n}B$ are nullary restrictions as defined above. Analogously, in \lemref{restrictions and composites} we will see that the restriction of $\ul K$ along $f$ and $g$, defined by the cartesian cell on the left above, coincides with the composite $(C(f, \id) \hc \ul K \hc D(\id, g))$.
	
	The following examples describe restrictions and horizontal units in various augmented double categories. At the end of this section weakly cocartesian cells of a certain shape are characterised in $\enProf\V$ (\exref{enriched profunctors}), $\Span\E$ (\exref{internal profunctors}) and $\Rel(\E)$ (\exref{internal relations}).
	
	\begin{example} \label{restrictions of V-profunctors}
		In the augmented virtual double category $\enProf\V$ of $\V$-profunctors (\exref{enriched profunctors}) unary restrictions $K(f, g)$ are indeed obtained by restricting the profunctor $K$: they consist of the family of $\V$-objects $K(fx, gy)$, for all $x \in A$, $y \in B$. Likewise the nullary restriction $C(f, g)$ of two $\V$-functors $f$ and $g$, with common target $C$, is given by the hom-objects $C(fx, gy)$; in particular $I_C(x, y) = C(x, y)$ defines the unit profunctor $I_C$. From this it easily follows that a cell $\cell\psi{\ul J}{\ul K}$ in $\enProf\V$, as in \defref{cartesian cells} above, is cartesian if and only if all its components $\map{\psi_{x, y}}{\ul J(x, y)}{\ul K(fx, gy)}$ are invertible.
	\end{example}
	
	\begin{example} \label{restrictions of spans}
		Let $\E$ be a category with pullbacks. One easily checks that in $\Span\E$ (\exref{internal profunctors}) the restriction of a span $\hmap KCD$ along morphisms $\map fAC$ and $\map gBD$ is the ``wide pullback'' of the diagram $A \xrar f C \leftarrow K \rightarrow D \xlar g B$. Similarly the nullary restriction $C(f, g)$ is obtained by pulling back the cospan $A \xrar f C \xlar g B$, while horizontal units are unit spans $I_A = (A \xlar{\id} A \xrar{\id} A)$. It is clear that the latter spans form relations in $\E$ so that, by \lemref{locally full and faithful functors reflect cartesian cells} below, they also form nullary restrictions and horizontal units in $\Rel(\E) \subseteq \Span\E$, the sub"/augmented virtual double category of relations in $\E$ (\exref{internal relations}).
	\end{example}
	
	\begin{example} \label{isomorphisms have cartesian identity cells}
		For any isomorphism $\map fAC$ in an augmented virtual double category the vertical identity cell $\id_f$ is cartesian. Similarly invertible vertical cells of the form $s \iso \id_A$ or $\id_A \iso s$, where $\map sAA$, are cartesian: factorisations through $s \iso \id_A$ for instance are obtained by composing on the left with its inverse $\id_A \iso s$.
	\end{example}
	
	The following straightforward lemma is useful for constructing restrictions in locally full sub"/augmented virtual double categories. A vertically dual result holds for weakly cocartesian cells, see \lemref{locally full and faithful functors reflect cocartesian paths} below.
	\begin{lemma} \label{locally full and faithful functors reflect cartesian cells}
		Any locally full and faithful functor $\map F\K\L$ (\defref{full and faithful functor}) reflects cartesian cells, that is a cell $\phi \in \K$ is cartesian whenever its image $F\phi$ is cartesian in $\L$.
	\end{lemma}
	
	\begin{example} \label{restrictions in (V, V')-Prof}
		For every universe enlargement $\V \subset \V'$ (\exref{(V, V')-Prof}) the locally full embedding $\enProf{(\V, \V')} \hookrightarrow \enProf{\V'}$ reflects cartesian cells. Since $K(f, g)$ is a $\V$"/profunctor whenever $K$ is, it follows that $\enProf{(\V, \V')}$ has all unary restrictions. Similarly the nullary restriction $\hmap{C(f, g)}AB$ exists in $\enProf{(\V, \V')}$ whenever the hom"/objects $C(fx, gy)$ are isomorphic to $\V$"/objects for all $x \in A$ and \mbox{$y \in B$}, and a $\V'$"/category $C$ is unital in $\enProf{(\V, \V')}$ whenever all its hom"/objects are isomorphic to $\V$"/objects. For example, in $\enProf{(\Set, \Set')}$ (\exref{(Set, Set')-Prof}) all nullary restrictions $C(f, g)$ exist as soon as the category $C$ is locally small. We will see in \exref{necessary condition for the existence of companions and conjoints in (V, V')-Prof} below that for the aforementioned sufficient condition for `one"/sided' nullary restrictions $C(f, \id)$ and $C(\id, g)$ in $\enProf{(\V, \V')}$, as well as that for unitality of $\V'$"/categories, are necessary conditions as well.
	\end{example}
	
	In the next example we denote by $I$ the unit $\V$-category, consisting of a single object $*$ and hom-object $I(*, *) = I$, the unit of $\V$. We can identify $\V$-functors $I \to A$ with objects in $A$ and $\V$"/profunctors $I \brar I$ with $\V$"/objects; cells between such profunctors can be identified with $\V$"/maps.
	
	\begin{example} \label{restrictions of small profunctors}
		The full embedding $\ensProf\V \hookrightarrow \enProf\V$ (\exref{small V-profunctors}) reflects cartesian cells by the previous lemma; we claim that it preserves cartesian cells as well. To see this consider any cartesian cell $\map\psi J{\ul K}$ in $\ensProf\V$, which defines a small $\V$"/profunctor $\hmap JAB$ as the restriction $\ul K(f, g)$ say. By \exref{restrictions of V-profunctors} it suffices to show that the components $\map{\psi_{(x, y)}}{J(x, y)}{\ul K(fx, gy)}$ are invertible for all $x \in A$ and $y \in B$. Notice that the cartesian cell $\cell\phi{J(x, y)}J$, which restricts $J$ along $\map xIA$ and $\map yIB$, is reflected by the embedding. It follows from \lemref{pasting lemma for cartesian cells} below that the composite $\cell{\psi \of \phi}{J(x, y)}{\ul K}$, which consists of the single component $\psi_{(x, y)}$, is a cartesian cell in $\ensProf\V$ that defines $\hmap{J(x, y)}II$ as the restriction of $\ul K$ along $\map{fx}IC$ and $\map{gy}ID$. But the latter restriction is reflected by the embedding too so that, by \exref{restrictions of V-profunctors} and the fact that restrictions are unique up to isomorphism, we may conclude that $\psi_{(x, y)}$ is invertible.

		It follows from the above that the restriction $\ul K(f, g)$ exists in $\ensProf\V$ if and only if $\ul K(f, g)$, when constructed in $\enProf\V$ as $\ul K(f, g)(x, y) = \ul K(fx, gy)$, is a small $\V$"/profunctor; in that case the two restrictions coincide. Clearly this is so for all unary restrictions $K(\id, g)$ on the right. It is easy to show that all unit $\V$"/profunctors $I_C$ are small too so that, using \cororef{restrictions in terms of units} below, we conclude that $\ensProf\V$ has all nullary restrictions $C(\id, g)$ on the right as well.
		
		To see that $\ensProf\V$ does not have restrictions $\ul K(f, \id)$ on the left in general take $\V = \Set$ and consider the terminal endoprofunctor $\hmap 111$ on the terminal category $1 = \set{*}$, i.e.\ $1(*, *)$ is the singleton set. It follows from the characterisation of small $\Set$"/profunctors given in \exref{small V-profunctors} that the restriction of $1$ along a terminal functor $\map !A1$, where $A$ is any properly large set regarded as a discrete category, is not small.
	\end{example}
	
	\begin{example} \label{restrictions of bimodules}
		Unary restrictions in the augmented virtual double category $(N \of \Mod)(\K)$ of bimodules in a virtual double category $\K$ (\exref{augmented virtual double category of monoids}) can be created in $\K$. For a bimodule $\hmap{(K, \lambda, \rho)}CD$ and monoid morphisms $\map{(f, \bar f)}AC$ and $\map{(g, \bar g)}BD$ this means that the restriction $K(f, g)$ in $\K$, if it exists, admits a bimodule structure that is unique in making its defining cartesian cell into a cartesian cell in $(N \of \Mod)(\K)$. Proving this is straightforward; see Proposition~11.10~of~\cite{Shulman08} for the analogous result in the case of pseudo double categories.
		
		Similarly the nullary restriction $C(f, g)$ in $(N \of \Mod)(\K)$, of a monoid $C = (C, \gamma, \bar\gamma, \tilde\gamma)$ and along monoid morphisms $\map{(f, \bar f)}AC$ and $\map{(g, \bar g)}BC$, can be created in $\K$ from the restriction $\gamma(f, g)$, if it exists. In particular every monoid $A = (A, \alpha, \bar\alpha, \tilde\alpha)$ has a horizontal unit given by $I_A = (\alpha, \bar\alpha, \bar\alpha)$ in $(N \of \Mod)(\K)$; in other words $(N \of \Mod)(\K)$ is a unital virtual double category.
	\end{example}
		
	\begin{example}
		By the previous example unary restrictions of internal profunctors in $\inProf\E$ (\exref{internal profunctors}) can be created as in $\Span\E$, that is as wide pullbacks (\exref{restrictions of spans}). Since the embedding $\spFib\K \hookrightarrow \inProf{\K_0}$ (\exref{internal split fibrations}) is locally full and faithful as well as surjective on horizontal morphisms, by \lemref{locally full and faithful functors reflect cartesian cells} above the restrictions of split two"/sided fibrations in $\K$ are given by wide pullbacks as well; this partially recovers Corollary~13 of \cite{Street74}.
	\end{example}
	
	It is clear from e.g.\ \cite{Cruttwell-Shulman10}, \cite{Koudenburg14}, \cite{Koudenburg18} and \cite{Shulman08} that the notion of a (generalised) double category that has all restrictions is a useful one. In \cite{Cruttwell-Shulman10} virtual double categories are called `virtual equipments' if they have all restrictions and all horizontal units---a term derived from Wood's `bicategories equipped with proarrows' \cite{Wood82}. As we have seen in the examples above some important augmented virtual double categories do not have all nullary restrictions (e.g.\ $\enProf{(\V, \V')}$) or only have restrictions on the right (e.g.\ $\ensProf\V$). This is why we consider the following generalisations of the notion of `equipment' as appropriate for augmented virtual double categories.
	\begin{definition} \label{augmented virtual equipment}
		An augmented virtual double category $\K$ is said to have \emph{restrictions on the left (resp.\ right)} if it has all unary restrictions of the form $K(f, \id)$ (resp.\ $K(\id, g)$). An \emph{augmented virtual equipment} is an augmented virtual double category that has all unary restrictions $K(f, g)$. A \emph{unital virtual equipment} is a unital virtual double category that has all restrictions $\ul K(f, g)$.
	\end{definition}
	For a unital virtual double category $\K$ to be a unital virtual equipment it suffices that $\K$ has all unary restrictions, by \cororef{restrictions in terms of units} below. Under the equivalence of \thmref{unital virtual double categories} our notion of unital virtual equipment coincides with that of `virtual equipment' studied in \cite{Cruttwell-Shulman10}. Table~\ref{notions} lists most of the (generalised) equipments that are considered in this paper.
	
	\begin{table}
		\centering
		\begin{tabular}{ll}
			\textbf{Notion}						& \textbf{Example(s)} \\
			\hline
			\rule[-1.2em]{0pt}{3em}\begin{tabular}{@{}l@{}}virtual double category \\ \quad w/ restrictions\end{tabular} & $\Mat\V$ \\
			\hline
			\rule[-1em]{0pt}{2.6em}augmented virtual equipment			& $\enProf{(\V, \V')}$ \\
			\hline
			\rule[-1.2em]{0pt}{3em}\begin{tabular}{@{}l@{}}unital virtual double category \\ \quad w/ restrictions on the right\end{tabular} & $\ensProf\V$ \\
			\hline
			\rule[-2.5em]{0pt}{5.6em}unital virtual equipment	& \begin{tabular}{@{}l@{}}$\Rel(\E)$ \\ $\inProf{\E}$ \\ $\spFib\K$ \\ $\enProf\V $\end{tabular}\\
			\hline
			\rule[-1.2em]{0pt}{3em}\begin{tabular}{@{}l@{}}pseudo double category \\ \quad w/ restrictions on the right\end{tabular} & $\ensProf\V$ ($\V$ is small cocomplete and closed) \\
			\hline
			\rule{0pt}{3.6em}equipment												& \begin{tabular}{@{}l@{}}$\Span{\E}$\\ $\Rel(\E)$ ($\E$ is regular) \\ $\inProf\E$ ($\E$ has coequalisers pres.\ by pullback)  \\ $\spFib\K$ ($\K$ has coequalisers pres.\ by pullback) \\ $\enProf{\V'}$ ($\V'$ large cocomplete and closed) \end{tabular}
		\end{tabular}
		\caption{Most examples of \secref{examples section} grouped according to whether they have all unary restrictions $K(f, g)$ (`equipment') and/or all horizontal units $I_A$ (`unital'). In the bottom two rows a `pseudo double category'/`equipment' is a unital virtual double category/unital virtual equipment that has all `horizontal composites', see \secref{composition section}; in these examples horizontal composites are in fact `pointwise' in the sense of \secref{pointwise horizontal composites section}.}
		\label{notions}
	\end{table}
	
	The following example demonstrates the relation between cartesian vertical identity cells and full and faithful vertical morphisms.
	\begin{example}
		In the augmented virtual double category $\enProf{(\V, \V')}$ (\exref{(V, V')-Prof}) the identity cell $\id_f$ of a $\V'$"/functor $\map fAC$ is cartesian if $f$ is full and faithful. Indeed if the actions $\map{\bar f}{A(x, y)}{C(fx, fy)}$ of $f$ on hom"/objects are invertible then the unique factorsation of a cell \mbox{$\cell\chi{(H_1, \dotsc, H_n)}C$} through $\id_f$, as in \defref{cartesian cells}, is obtained by composing the components of $\chi$ with the inverses of $\bar f$. The converse holds as soon as the nullary restriction $\hmap{C(f, f)}AA$ exists in $\enProf{(\V, \V')}$: the inverses of $\bar f$ can then be recovered as the components of the factorisation of the cartesian cell that defines $C(f, f)$ through $\id_f$.
	\end{example}
	
	In view of the previous we make the following definition.
	\begin{definition} \label{full and faithful morphism}
		A vertical morphism $\map fAC$ is called \emph{full and faithful} if its identity cell $\id_f$ is cartesian.
	\end{definition}
	In Section~8 of \cite{Cruttwell-Shulman10} a notion of full and faithfulness is introduced for morphisms of monoids, in terms of the unital virtual double category $\Mod(\K)$ of bimodules in a virtual double category $\K$ (\exref{augmented virtual double category of monoids}). Under the equivalence of \thmref{unital virtual double categories} this notion coincides with ours, as follows from the discussion following \cororef{vertical cells}.
	
	\begin{example} \label{isomorphisms are full and faithful}
		Isomorphisms are full and faithful by \exref{isomorphisms have cartesian identity cells}.
	\end{example}
	
	The converse to the following lemma holds as soon as $\K$ has `all weakly cocartesian paths of $(0, 1)$"/ary cells', see \propref{converses with weakly cocartesian paths of (0,1)-ary cells} below.
	\begin{lemma} \label{full and faithful in V(K)}
		If $\map fAC$ is full and faithful in the augmented virtual double category $\K$ then it is full and faithful in the $2$-category $V(\K)$ (\exref{vertical 2-category}): for any $X \in \K$ the functor $\map{V(\K)(X,f)}{V(\K)(X,A)}{V(\K)(X,C)}$, given by postcomposition with $f$, is full and faithful (see e.g.\ \cite{Street74}).
	\end{lemma}
	
	Cartesian cells satisfy the following pasting lemma. As a consequence, taking restrictions is `pseudofunctorial' in the sense that $\ul K(f, g)(h, k) \iso \ul K(f \of h, g \of k)$ and $K(\id, \id) \iso K$.
	\begin{lemma}[Pasting lemma] \label{pasting lemma for cartesian cells}
		If the cell $\phi$ in the composite below is cartesian then the composite $\phi \of \psi$ is cartesian if and only if $\psi$ is.
		\begin{displaymath}
			\begin{tikzpicture}
    		\matrix(m)[math35]{X & Y \\ A & B \\ C & D \\};
    		\path[map]  (m-1-1) edge[barred] node[above] {$\ul H$} (m-1-2)
        		                edge node[left] {$h$} (m-2-1)
            		    (m-1-2) edge node[right] {$k$} (m-2-2)
            		    (m-2-1) edge[barred] node[below] {$\ul J$} (m-2-2)
            		            edge node[left] {$f$} (m-3-1)
            		    (m-2-2) edge node[right] {$g$} (m-3-2)
            		    (m-3-1) edge[barred] node[below] {$\ul K$} (m-3-2);
    		\path[transform canvas={xshift=1.75em}]
        		        (m-1-1) edge[cell] node[right] {$\psi$} (m-2-1)
        		        (m-2-1) edge[transform canvas={yshift=-0.25em}, cell] node[right] {$\phi$} (m-3-1);
  		\end{tikzpicture}
		\end{displaymath}
	\end{lemma}
	
	Restricting to the case where the cartesian cell $\phi$ above defines a horizontal unit $I_C$ we find that nullary restrictions can be obtained as unary restrictions of horizontal units, as in the following corollary. Consequently in an augmented virtual equipment all nullary restrictions $C(f, g)$ exist whenever the object $C$ is unital.
	\begin{corollary} \label{restrictions in terms of units}
		Let $\map fAC$ and $\map gBC$ be morphisms into a unital object $C$. The nullary restriction $C(f, g)$ exists if and only if the unary restriction $I_C(f, g)$ does, and in that case they are isomorphic.
		\begin{displaymath}
  		\begin{tikzpicture}[textbaseline]
				\matrix(m)[math35, column sep={1.75em,between origins}]{A & & B \\ & C & \\};
				\path[map]	(m-1-1) edge[barred] node[above] {$J$} (m-1-3)
														edge[transform canvas={xshift=-1pt}] node[left] {$f$} (m-2-2)
										(m-1-3) edge[transform canvas={xshift=1pt}] node[right] {$g$} (m-2-2);
				\path[transform canvas={yshift=0.25em}]	(m-1-2) edge[cell] node[right, inner sep=3pt] {$\phi$} (m-2-2);
			\end{tikzpicture} \quad = \quad \begin{tikzpicture}[textbaseline]
  			\matrix(m)[math35, column sep={1.75em,between origins}]{A & & B \\ C & & C \\ & C & \\};
  			\path[map]	(m-1-1) edge[barred] node[above] {$J$} (m-1-3)
  													edge node[left] {$f$} (m-2-1)
  									(m-1-3) edge node[right] {$g$} (m-2-3)
  									(m-2-1) edge[barred] node[below, inner sep=1.5pt] {$I_C$} (m-2-3);
  			\path				(m-2-1)	edge[eq, transform canvas={xshift=-2pt}] (m-3-2)
  									(m-2-3) edge[eq, transform canvas={xshift=2pt}] (m-3-2);
  			\path[transform canvas={xshift=1.75em}]
										(m-1-1) edge[cell] node[right] {$\phi'$} (m-2-1);
				\draw				([yshift=0.25em]$(m-2-2)!0.5!(m-3-2)$) node[font=\scriptsize] {$\cart$};
			\end{tikzpicture}
  	\end{displaymath}
  	In detail a nullary cell $\phi$, as on the left-hand side above, is cartesian if and only if its factorisation $\phi'$ through $I_C$ is cartesian.
	\end{corollary}
	
	Horizontal composition with the (co)unit of an adjunction preserves nullary cartesian cells as follows.
	\begin{lemma} \label{horizontal composition with (co-)units preserves nullary cartesian cells}
		In an augmented virtual double category $\K$ consider the composite below. If $\eta$ is the unit of an adjunction $f \ladj g$ (in $V(\K)$) then $\phi$ is cartesian precisely if the composite is so. A horizontally dual result holds for composition on the right with a counit.
		\begin{displaymath}
			\begin{tikzpicture}
  			\matrix(m)[math35, column sep={0.875em,between origins}]
  				{	A & & & & B \\
  					& C & & & \\
  					& & E & & \\
  					& C & & & \\ };
  			\path[map]	(m-1-1) edge[barred] node[above] {$\ul J$} (m-1-5)
  													edge[transform canvas={xshift=-1pt}] node[left] {$h$} (m-2-2)
  									(m-1-5) edge[transform canvas={xshift=1pt}] node[right] {$k$} (m-3-3)
  									(m-2-2) edge node[right, inner sep=0.75pt, yshift=2pt] {$f$} (m-3-3)
  									(m-3-3) edge[bend left=12] node[right] {$g$} (m-4-2);
  			\path				(m-1-3) edge[cell, transform canvas={yshift=-0.5em}] node[right] {$\phi$} (m-2-3)
  									(m-2-2) edge[cell, transform canvas={yshift=-1.625em, xshift=-4pt}] node[right, inner sep=2.5pt] {$\eta$} (m-3-2)
  									(m-2-2) edge[eq, bend right=32] (m-4-2);
  		\end{tikzpicture}
  	\end{displaymath}
  \end{lemma}
  \begin{proof}
  	Consider the commutative diagram of assignments below, between collections of cells in $\K$ that are of the shape as shown.
  	\begin{displaymath}
  		\begin{tikzpicture}
  			\matrix(k)[math35, column sep={0.5833em,between origins}, xshift=-11em]
  				{	X & & & & & & Y \\
  					& A & & & & & \\
  					& & C & & & & \\
  					& & & E & & & \\ };
  			\draw				($(k-1-7)!0.5!(k-4-4)$) node(B) {$B$};
  			\path[map]	(k-1-1) edge[barred] node[above] {$\ul H$} (k-1-7)
  													edge[transform canvas={xshift=-1pt}] node[left] {$p$} (k-2-2)
  									(k-1-7) edge[transform canvas={xshift=1pt}] node[right] {$q$} (B)
  									(k-2-2) edge[transform canvas={xshift=-1pt}] node[left] {$h$} (k-3-3)
  									(k-3-3) edge[transform canvas={xshift=-1pt}] node[left] {$f$} (k-4-4)
  									(B) edge[transform canvas={xshift=1pt}] node[right] {$k$} (k-4-4);
  			\path				(k-1-4) edge[cell, transform canvas={yshift=-0.75em}] node[right] {$\chi$} (k-2-4);
  			
  			\matrix(m)[math35]{X & Y \\ A & B \\};
				\path[map]	(m-1-1) edge[barred] node[above] {$\ul H$} (m-1-2)
														edge node[left] {$p$} (m-2-1)
										(m-1-2) edge node[right] {$q$} (m-2-2)
										(m-2-1) edge[barred] node[below] {$\ul J$} (m-2-2);
				\path[transform canvas={xshift=1.75em}]	(m-1-1) edge[cell] node[right] {$\psi$} (m-2-1);
  			
  			\matrix(n)[math35, column sep={0.5833em,between origins}, xshift=14.5em]
  				{	X & & & & & & Y \\
  					& & & & & B & \\
  					& & & & E & & \\
  					& & & C & & & \\ };
  			\draw				($(n-1-1)!0.5!(n-4-4)$) node(A) {$A$};
  			\path[map]	(n-1-1) edge[barred] node[above] {$\ul H$} (n-1-7)
  													edge[transform canvas={xshift=-1pt}] node[left] {$p$} (A)
  									(n-1-7) edge[transform canvas={xshift=1pt}] node[right] {$q$} (n-2-6)
  									(n-2-6) edge[transform canvas={xshift=1pt}] node[right] {$k$} (n-3-5)
  									(n-3-5) edge[transform canvas={xshift=1pt}] node[right] {$g$} (n-4-4)
  									(A) edge[transform canvas={xshift=-1pt}] node[left] {$h$} (n-4-4);
  			\path				(n-1-4) edge[cell, transform canvas={yshift=-0.75em}] node[right] {$\xi$} (n-2-4);
  			
  			\draw[font=\Large]	(-2.8em,0) node {$\lbrace$}
										(2.8em,0) node {$\rbrace$}
										(-13.75em,0) node {$\lbrace$}
										(-8.5em,0) node {$\rbrace$}
										(12em,0) node {$\lbrace$}
										(17.25em,0) node {$\rbrace$};
				\path[map]	(-3.8em,0) edge node[above] {$\phi \of \dash$} (-7.5em,0)
										(3.8em,0) edge node[above] {$\bigpars{(\eta \of h) \hc (g \of \phi)} \of \dash$} (11em,0)
										(-8em,-2em) edge[bend right=25] node[below] {$(\eta \of h \of p) \hc (g \of \dash)$} (11.5em,-2em);	
  		\end{tikzpicture}
		\end{displaymath}
		Notice that the $\phi$ is cartesian precisely if the top left assignment is a bijection, and that the composite $(\eta \of h) \hc (g \of \phi)$ is cartesian precisely if the top right assignment is a bijection. The proof follows from the fact that the bottom assignment is a bijection: its inverse is given by composing the cells $\xi$ on the right with the counit of $f \ladj g$.
	\end{proof}
	
	Closing this section we characterise weakly cocartesian cells (of a certain shape) in the virtual double categories $\enProf\V$, $\Span\E$ and $\Rel(\E)$.
	\begin{example} \label{weakly cocartesian cells in V-Prof}
		Consider cells $\phi$ in $\enProf\V$ (\exref{enriched profunctors}) of the form below, where $I$ denotes the unit $\V$"/category as recalled before \exref{restrictions of small profunctors}. Notice that such cells correspond precisely to cocones $\ul H^\S(*, *) \Rar J$, where $J \dfn J(*,*)$ and $\ul H^\S(*, *)$ is the diagram defining the iterated coend $\int^{u_1 \in X_1} \!\dotsb \int^{u_{n'} \in X_{n'}} H_1(*, u_1) \tens \dotsb \tens H_n(u_{n'}, *)$; see the definition preceding \exref{small V-profunctors}. Indeed, the internal equivariance axioms satisfied by such $\phi$ (\defref{monoids and bimodules}) correspond precisely to the naturality of such cocones. Weakly cocartesian cells thus correspond to colimiting cocones, that is the cell $\phi$ below is weakly cocartesian in $\enProf\V$ precisely if it defines the $\V$-object $J$ as the afore-mentioned coend.
		\begin{displaymath}
			\begin{tikzpicture}[textbaseline]
				\matrix(m)[math35]{I & X_1 & X_{n'} & I \\ I & & & I \\};
				\path[map]	(m-1-1) edge[barred] node[above] {$H_1$} (m-1-2)
										(m-1-3) edge[barred] node[above] {$H_n$} (m-1-4)
										(m-2-1) edge[barred] node[below] {$J$} (m-2-4);
				\path				(m-1-1) edge[eq] (m-2-1)
										(m-1-4) edge[eq] (m-2-4);
				\path[transform canvas={xshift=1.75em}]	(m-1-2) edge[cell] node[right] {$\phi$} (m-2-2);
				\draw				($(m-1-2)!0.5!(m-1-3)$) node {$\dotsb$};
			\end{tikzpicture}
		\end{displaymath}
	\end{example}
	
	\begin{example} \label{weakly cocartesian cells in Span(E)}
		In $\Span\E$ (\exref{internal profunctors}) a horizontal cell $\phi$, of the form as below, is weakly cocartesian if and only if its underlying morphism $\map\phi{H_1 \times_{X_1} \dotsb \times_{X_{n'}} H_n}J$ is an isomorphism.
		\begin{displaymath}
			\begin{tikzpicture}
				\matrix(m)[math35]{X_0 & X_1 & X_{n'} & X_n \\ X_0 & & & X_n \\};
				\path[map]	(m-1-1) edge[barred] node[above] {$H_1$} (m-1-2)
										(m-1-3) edge[barred] node[above] {$H_n$} (m-1-4)
										(m-2-1) edge[barred] node[below] {$J$} (m-2-4);
				\path				(m-1-1) edge[eq] (m-2-1)
										(m-1-4) edge[eq] (m-2-4);
				\path[transform canvas={xshift=1.75em}]	(m-1-2) edge[cell] node[right] {$\phi$} (m-2-2);
				\draw				($(m-1-2)!0.5!(m-1-3)$) node {$\dotsb$};
			\end{tikzpicture}
		\end{displaymath}
	\end{example}
	
	\begin{example} \label{weakly cocartesian cells in Rel(E)}
		Consider a horizontal cell $\phi$ in the virtual double category $\Rel(\E)$ (\exref{internal relations}) of the form as above. Recall that it is given by a morphism of spans $\map\phi{H_1 \times_{X_1} \dotsb \times_{X_{n'}} H_n}J$ as in the bottom left commuting triangle in the diagram below, where the relation $X_0 \xlar{j_0} J \xrar{j_1} X_n$ is drawn as the (jointly monic) pair of morphisms \mbox{$\mono{(j_0, j_1)}J{(X_0, X_n)}$}; composition in this diagram is defined in the obvious way.
		\begin{displaymath}
			\begin{tikzpicture}
				\matrix(m)[math35, column sep=2em, row sep=2.75em]{H_1 \times_{X_1} \dotsb \times_{X_{n'}} H_n &[-1em] & K \\ J & (X_0, X_n) & (C, D) \\};
				\path[map]	(m-1-1) edge[transform canvas={yshift=1pt}] node[above] {$\chi$} (m-1-3)
														edge (m-2-2)
														edge node[left] {$\phi$} (m-2-1)
										(m-1-3) edge[mono] node[right] {$(k_0, k_1)$} (m-2-3)
										(m-2-1) edge[dashed] node[above] {$\psi$} (m-1-3)
														edge[mono, transform canvas={yshift=-1pt}] node[below] {$(j_0, j_1)$} (m-2-2)
										(m-2-2) edge node[below] {$(f, g)$} (m-2-3);
			\end{tikzpicture}
		\end{displaymath}
		The cell $\phi$ is weakly cocartesian (\defref{cartesian cells}) if for any relation $C \xlar{k_0} K \xrar{k_1} D$ and any morphism of spans $\map\chi{H_1 \times_{X_1} \dotsb \times_{X_{n'}} H_n}K$, as in the commuting trapezium on the right in the diagram above, there exists a (necessarily unique) lift $\psi$ making the diagram commute. Such lifts exist in particular when $\phi$ is a \emph{strong epimorphism} in the sense of Section~3 of \cite{Carboni-Kasangian-Street84}: in that case lifts $\map\psi JK$ exist in any commuting square of the form $(k_0, k_1) \of \chi = (p, q) \of \phi$ where $C \xlar p J \xrar q D$ is any span.
	\end{example}
	
	\section{Companions and conjoints} \label{companion and conjoint section}
	Here we study nullary restrictions of the forms $C(f, \id)$ and $C(\id, f)$ where $\map fAC$ is any vertical morphism. These have been called respectively `horizontal companions' and `horizontal adjoints' in the setting of double categories \cite{Grandis-Pare04}; we follow Section~7 of \cite{Cruttwell-Shulman10} and call them `companions' and `conjoints'. We remark that the latter only defines companions and conjoints in virtual double categories that have all horizontal units; in contrast the definition for augmented virtual double categories below does not require horizontal units. As foreshadowed in the discussion following \defref{cartesian cells}, companions and conjoints can be regarded as building blocks for restrictions and extensions as will be explained in \secref{restrictions and extensions in terms of companions and conjoints section}.
	
	\begin{definition}
		Let $\map fAC$ be a vertical morphism in an augmented virtual double category. The nullary restriction $\hmap{C(f, \id)}AC$ is called the \emph{companion} of $f$ and denoted $f_*$. Likewise $\hmap{C(\id, f)}CA$ is called the \emph{conjoint} of $f$ and denoted $f^*$.
	\end{definition}
	
	Notice that the notions of companion and conjoint are interchanged when moving from $\K$ to its horizontal dual $\co\K$ (\defref{horizontal dual}).
	
	\begin{example}
		It follows from \exref{restrictions of V-profunctors} that the companion $f_*$ and conjoint $f^*$ of a $\V$"/functor $\map fAC$ in $\enProf\V$ are the representable $\V$"/profunctors given by $f_*(x, z) = C(fx, z)$ and $f^*(z, x) = C(z, fx)$. Companions and conjoints in $\enProf{(\V, \V')}$ (\exref{(V, V')-Prof}) are characterised in \exref{necessary condition for the existence of companions and conjoints in (V, V')-Prof} below.
	\end{example}
	
	\begin{example} \label{companions and conjoints in Span(E)}
		From \exref{restrictions of spans} it follows that in the augmented virtual double categories $\Span\E$ (\exref{internal profunctors}) and $\Rel(\E)$ (\exref{internal relations}) the companion and conjoint of a morphism $\map fAC$ are the relations $f_* = (A \xlar{\id} A \xrar f C)$ and $f^* = (C \xlar f A \xrar{\id} A)$ in $\E$.
	\end{example}
	
	While the companion and conjoint of a morphism $\map fAC$ have been defined as nullary restrictions along $f$, the following lemma and its horizontal dual show that they can equivalently be defined as extensions. More precisely it gives, for a horizontal morphism $\hmap JAC$, a bijective correspondence between cartesian cells $\psi$ defining $J$ as the companion of $f$ and weakly cocartesian cells $\phi$ defining $J$ as the extension of $(A)$ along $\id_A$ and $f$, in such a way that each corresponding pair $(\psi, \phi)$ satisfies the identities below; these are called the \emph{companion identities}. Horizontally dual identities are satisfied by pairs of corresponding cartesian and weakly cocartesian cells defining a conjoint; these are called the \emph{conjoint identities}. In \cororef{cocartesian cells for companions and conjoints} below we will see that any weakly cocartesian cell defining a companion or conjoint satisfies the stronger notion of `cocartesian cell' in the sense of \secref{composition section}.
	\begin{lemma} \label{companion identities lemma}
		Consider the factorisation of a vertical identity cell on the left below. The following conditions are equivalent: $\psi$ is cartesian; the identity on the right below holds; $\phi$ is weakly cocartesian.
	 	\begin{displaymath}
	  	\begin{tikzpicture}[textbaseline]
						\matrix(m)[math35]{A \\ C \\};
						\path[map]	(m-1-1) edge[bend right=45] node[left] {$f$} (m-2-1)
																edge[bend left=45] node[right] {$f$} (m-2-1);
						\path[transform canvas={xshift=-0.5em}]	(m-1-1) edge[cell] node[right] {$\id_f$} (m-2-1);
			\end{tikzpicture} \quad = \quad \begin{tikzpicture}[textbaseline]
    		\matrix(m)[math35, column sep={1.75em,between origins}]{& A & \\ A & & C \\ & C & \\};
    		\path[map]	(m-1-2) edge node[right] {$f$} (m-2-3)
    								(m-2-1) edge[barred] node[below, inner sep=2pt] {$J$} (m-2-3)
    												edge node[left] {$f$} (m-3-2);
    		\path				(m-1-2) edge[eq, transform canvas={xshift=-1pt}] (m-2-1)
    								(m-2-3) edge[eq, transform canvas={xshift=1pt}] (m-3-2);
    		\path				(m-1-2) edge[cell, transform canvas={yshift=-0.25em}] node[right, inner sep=2pt] {$\phi$} (m-2-2)
    								(m-2-2) edge[cell, transform canvas={yshift=0.1em}]	node[right, inner sep=2pt] {$\psi$} (m-3-2);
  		\end{tikzpicture} \qquad\qquad\quad \begin{tikzpicture}[textbaseline]
  			\matrix(m)[math35, column sep={1.75em,between origins}]{& A & & C \\ A & & C & \\};
  			\path[map]	(m-1-2) edge[barred] node[above] {$J$} (m-1-4)
  													edge node[above right, inner sep=0.5pt] {$f$} (m-2-3)
  									(m-2-1) edge[barred] node[below] {$J$} (m-2-3);
  			\path				(m-1-2) edge[eq, transform canvas={xshift=-1pt}] (m-2-1)
  									(m-1-4) edge[eq, transform canvas={xshift=1pt}] (m-2-3);
  			\path				(m-1-2) edge[cell, transform canvas={yshift=-0.25em}]	node[right, inner sep=2pt] {$\phi$} (m-2-2)
  									(m-1-3) edge[cell, transform canvas={yshift=0.25em}] node[right, inner sep=2pt] {$\psi$} (m-2-3);
  		\end{tikzpicture} \mspace{6mu} = \quad \begin{tikzpicture}[textbaseline]
  			\matrix(m)[math35]{A & C \\ A & C \\};
  			\path[map]	(m-1-1) edge[barred] node[above] {$J$} (m-1-2)
  									(m-2-1) edge[barred] node[below] {$J$} (m-2-2);
  			\path				(m-1-1) edge[eq] (m-2-1)
  									(m-1-2) edge[eq] (m-2-2);
  			\path[transform canvas={xshift=1.75em, xshift=-5.5pt}]	(m-1-1) edge[cell] node[right] {$\id_J$} (m-2-1);
  		\end{tikzpicture}
	  \end{displaymath}
	\end{lemma}
	\begin{proof}[(sketch)]
		Notice that both sides of the identity on the right coincide after composing them with $\psi$ below or with $\phi$ above. Hence the identity itself follows when $\psi$ is cartesian or $\phi$ is weakly cocartesian, by the uniqueness of factorisations through (weakly co)cartesian cells.
		
		Conversely if both identities hold then the unique factorisation of any nullary cell $\chi$ through $\psi$, as in \defref{cartesian cells} but with $g = \id_C$, is obtained by composing $\chi$ on the left with $\phi$; this shows that $\psi$ is cartesian. Unique factorisations through $\phi$ are likewise obtained by composing with $\psi$ on the right, showing that $\phi$ is weakly cocartesian.
	\end{proof}
	
	As an immediate consequence we find that, unlike functors between virtual double categories, functors of augmented virtual double categories preserve companions, conjoints and horizontal units.
	\begin{corollary} \label{functors preserve companions and conjoints}
		Any functor of augmented virtual double categories preserves the cartesian and weakly cocartesian cells that define companions, conjoints or horizontal units.
	\end{corollary}
	\begin{proof}
		This follows immediately from the fact that functors preserve vertical composition strictly, so that the companion and conjoint identities of (the horizontal dual of) the previous lemma are preserved.
	\end{proof}
	
	\begin{example} \label{necessary condition for the existence of companions and conjoints in (V, V')-Prof}
		In \exref{restrictions in (V, V')-Prof} we saw that a $\V'$"/functor \mbox{$\map fAC$} has a companion $f_*$ in $\enProf{(\V, \V')}$ as soon as all hom-objects $C(fx, z)$ are isomorphic to $\V$-objects. Using the previous lemma we can prove the converse, as follows. If the companion $f_*$ exists in $\enProf{(\V, \V')}$, as a $\V$"/profunctor $\hmap{f_*}AC$, then consider cells $\cell\psi{f_*}C$ and $\cell\phi A{f_*}$ as in the lemma. It is straightforward to check that the companion identities for $\phi$ and $\psi$ imply that the composites below are inverses for the components $f_*(x, z) \to C(fx, z)$ of $\psi$, thus showing that $C(fx, z) \iso f_*(x, z)$, the latter of which are $\V$"/objects for all $x \in A$ and $z \in C$.
		\begin{displaymath}
			C(fx, y) \iso I \tens' C(fx, y) \xrar{\phi_x \tens' \id} J(x, fx) \tens' C(fx, y) \xrar\rho J(x, y)
		\end{displaymath}
	
		Horizontally dual, the conjoint $\hmap{f^*}CA$ exists in $\enProf{(\V, \V')}$ if and only if the hom"/objects $C(z, fx)$ are isomorphic to $\V$"/objects.
	\end{example}
	
	The companion identities of the lemma above, together with the conjoint identities, directly imply the following. The analogous result for unital virtual equipments was proved as Theorem~7.20 of \cite{Cruttwell-Shulman10}.
	\begin{corollary} \label{horizontal cells}
		Let $\map f{A_0}C$ and $\map g{A_n}D$ be morphisms such that the conjoint $f^*$ and the companion $g_*$ exist. Horizontally composing with the cartesian cells defining $f^*$ and $g_*$ gives a bijection between cells of the form
		\begin{displaymath}
			\begin{tikzpicture}[textbaseline]
				\matrix(m)[math35]{A_0 & A_n \\ C & D \\};
				\path[map]	(m-1-1) edge[barred] node[above] {$\ul J$} (m-1-2)
														edge node[left] {$f$} (m-2-1)
										(m-1-2) edge node[right] {$g$} (m-2-2)
										(m-2-1) edge[barred] node[below] {$\ul K$} (m-2-2);
				\path[transform canvas={xshift=1.75em}]	(m-1-1) edge[cell] node[right] {$\phi$} (m-2-1);
			\end{tikzpicture} \quad\qquad \qquad \text{and} \quad\qquad \qquad \begin{tikzpicture}[textbaseline]
				\matrix(m)[math35]{C & D \\ C & D. \\};
				\path[map]	(m-1-1) edge[barred] node[above, inner sep=7pt] {$f^* \conc \ul J \conc g_*$} (m-1-2)
										(m-2-1) edge[barred] node[below] {$\ul K$} (m-2-2);
				\path				(m-1-1) edge[eq] (m-2-1)
										(m-1-2) edge[eq] (m-2-2);
				\path[transform canvas={xshift=1.75em}]	(m-1-1) edge[cell] node[right] {$\psi$} (m-2-1);
			\end{tikzpicture}
		\end{displaymath}
	\end{corollary}
	
	\begin{example} \label{restrictions of internal relations}
		As was recalled in \exref{internal relations}, a relation $J$ internal to a category $\E$ is a span \mbox{$\hmap JAB$} in $\E$ such that any two horizontal cells $\cell{\phi, \psi}HJ$ in $\Span\E$ coincide. Since $\Span\E$ has all companions and conjoints (\exref{companions and conjoints in Span(E)}), by the corollary the latter is equivalent to asking that any two cells $\cell{\phi, \psi}HJ$, of the same shape but not necessarily horizontal, coincide in $\Span\E$.
		
		Now consider a unary restriction $K(f, g)$ of a relation $K$ in $\Span\E$ (\exref{restrictions of spans}). By the universal property of $K(f, g)$ and the preceding it follows that $K(f, g)$ is again a relation and thus, using \lemref{locally full and faithful functors reflect cartesian cells}, forms the restriction of $K$ in $\Rel(\E)$. Since $\Rel(\E)$ has all nullary restrictions as well (\exref{restrictions of spans}), we conclude that $\Rel(\E)$ is a unital virtual equipment.
	\end{example}
	
	Since horizontal units $I_C$ are a special kind of companions, i.e.\ $I_C \dfn C(\id, \id) = (\id_C)_*$ (see the discussion following \defref{cartesian cells}), they too are defined by pairs $(\psi, \phi)$ of cells satisfying two `horizontal unit identities', as the lemma below explains. It also shows that the cells $\psi$ and $\phi$ are both cartesian as well as weakly cocartesian; in \lemref{cocartesian unit identities} below we will see that they are `cocartesian' as well, in the sense of \secref{composition section}.
	\begin{lemma} \label{unit identities}
		Consider cells $\psi$ and $\phi$ as in the identities below, and assume that either identity holds. The following conditions are equivalent: \textup{(a)} $\psi$ is cartesian; \textup{(b)} $\psi$ is weakly cocartesian; \textup{(c)} both identities hold; \textup{(d)} $\phi$ is cartesian; \textup{(e)} $\phi$ is  weakly cocartesian.
		\begin{displaymath}
			\begin{tikzpicture}[textbaseline]
				\matrix(m)[math35, column sep={1.75em,between origins}]{A \\ A \\};
				\path				(m-1-1) edge[bend left = 55, eq] (m-2-1)
														edge[bend right = 55, eq] (m-2-1);
				\path[transform canvas={xshift=-0.6em}]	(m-1-1) edge[cell] node[right] {$\id_A$} (m-2-1);
			\end{tikzpicture} \quad \overset{\textup{(A)}} = \quad \begin{tikzpicture}[textbaseline]
    		\matrix(m)[math35, column sep={1.75em,between origins}]{& A & \\ A & & A \\ & A & \\};
    		\path[map]	(m-2-1) edge[barred] node[below, inner sep=1.5pt] {$J$} (m-2-3);
    		\path				(m-1-2) edge[transform canvas={xshift=2pt}, eq] (m-2-3)
    								(m-2-3)	edge[transform canvas={xshift=2pt}, eq] (m-3-2)
    								(m-1-2) edge[eq, transform canvas={xshift=-2pt}] (m-2-1)
    								(m-2-1) edge[eq, transform canvas={xshift=-2pt}] (m-3-2);
    		\path				(m-1-2) edge[cell, transform canvas={yshift=-0.25em}] node[right] {$\phi$} (m-2-2)
    								(m-2-2) edge[cell, transform canvas={yshift=0.1em}]	node[right] {$\psi$} (m-3-2);
  		\end{tikzpicture} \qquad\qquad\qquad\quad \begin{tikzpicture}[textbaseline]
  			\matrix(m)[math35]{A & A \\ A & A \\};
  			\path[map]	(m-1-1) edge[barred] node[above] {$J$} (m-1-2)
  									(m-2-1) edge[barred] node[below] {$J$} (m-2-2);
  			\path				(m-1-1) edge[eq] (m-2-1)
  									(m-1-2) edge[eq] (m-2-2);
  			\path[transform canvas={xshift=1.75em, xshift=-5.5pt}]	(m-1-1) edge[cell] node[right] {$\id_J$} (m-2-1);
  		\end{tikzpicture} \quad \overset{\textup{(J)}} = \quad \begin{tikzpicture}[textbaseline]
  			\matrix(m)[math35, column sep={1.75em,between origins}]{A & & A \\ & A & \\ A & & A \\};
  			\path[map]	(m-1-1) edge[barred] node[above] {$J$} (m-1-3)
  									(m-3-1) edge[barred] node[below] {$J$} (m-3-3);
  			\path				(m-1-1) edge[eq, transform canvas={xshift=-2pt}] (m-2-2)
  									(m-1-3) edge[eq, transform canvas={xshift=2pt}] (m-2-2)
  									(m-2-2) edge[eq, transform canvas={xshift=-2pt}] (m-3-1)
  													edge[eq, transform canvas={xshift=2pt}] (m-3-3);
  			\path				(m-1-2) edge[cell, transform canvas={yshift=0.25em}] node[right] {$\psi$} (m-2-2)
  									(m-2-2) edge[cell, transform canvas={yshift=-0.25em}] node[right] {$\phi$} (m-3-2);
  		\end{tikzpicture}
  	\end{displaymath}
  	Consequently any cell of the form as $\psi$ or $\phi$ above is cartesian if and only if it is weakly cocartesian.
	\end{lemma}
	\begin{proof}
		We will show that under the assumption of the identity (A) the implications (a)~$\Leftrightarrow$~(c)~$\Leftrightarrow$~(e) and (a) $\Rightarrow$ (d) $\Rightarrow$ (c) hold while, under the assumption of (J), either (a) or (d) implies (A). Vertically dual, one similarly shows that (e)~$\Rightarrow$~(b)~$\Rightarrow$~(c) under the assumption of (A), while (e) or (b) implies (J) $\Rightarrow$ (A). Together these complete the proof of the main assertion.
		
		Assuming (A) first notice that (a), (c) and (e) are equivalent by \lemref{companion identities lemma}, using the fact that $\phi \of \psi = \phi \hc \psi$ by the interchange axioms (\lemref{horizontal composition}). Applying the pasting lemma to (A) shows that (a) $\Rightarrow$ (d).
		
		Next we show that under the assumption of (d) the identities (A) and (J) are equivalent so that, in particular, (d) $\Rightarrow$ (c) follows from (A). If (d) holds, that is $\phi$ is cartesian, then there exists a unique cell $\psi'$ such that $\id_J = \phi \of \psi'$. Because $\phi \of \psi' \of \phi = \phi$ and $\phi$ is cartesian, $\psi' \of \phi = \id_A$ follows. If (A) holds then $\psi = \psi \of \phi \of \psi' = \psi'$ follows, so that $\id_J = \phi \of \psi' = \phi \of \psi$, which is (J). On the other hand if (J) then $\psi = \psi' \of \phi \of \psi = \psi'$, so that $\id_A = \psi' \of \phi = \psi \of \phi$, which is (A).
		
		It remains to prove that (a) implies (J) $\Rightarrow$ (A). If (a) holds then $\id_A$ factors as $\id_A = \psi \of \phi'$; assuming (J) we then have $\phi = \phi \of \psi \of \phi' = \phi'$ so that (A) follows. For the final assertion notice that any (weakly co)cartesian cell of the form $\phi$ or $\psi$ can be used to obtain a factorisation of either form (A) or (J), so that the equivalence follows from applying the main assertion.
	\end{proof}
	
	The following is similar to \cororef{horizontal cells}.
	\begin{corollary} \label{vertical cells}
		Let $A$ and $C$ be unital objects. Vertically composing with the cartesian cells $I_A \Rightarrow A$ and $C \Rightarrow I_C$ that define the horizontal units $I_A$ and $I_C$ gives a bijection between cells of the form
		\begin{displaymath}
			\begin{tikzpicture}[textbaseline]
			\matrix(m)[math35]{A \\ C \\};
			\path[map]	(m-1-1) edge[bend right=45] node[left] {$f$} (m-2-1)
													edge[bend left=45] node[right] {$g$} (m-2-1);
			\path				(m-1-1) edge[cell] node[right] {$\phi$} (m-2-1);
		\end{tikzpicture} \quad \qquad \qquad \text{and} \quad \qquad \qquad \begin{tikzpicture}[textbaseline]
				\matrix(m)[math35]{A & A \\ C & C \\};
				\path[map]	(m-1-1) edge[barred] node[above] {$I_A$} (m-1-2)
														edge node[left] {$f$} (m-2-1)
										(m-1-2) edge node[right] {$g$} (m-2-2)
										(m-2-1) edge[barred] node[below] {$I_C$} (m-2-2);
				\path[transform canvas={xshift=1.75em}]	(m-1-1) edge[cell] node[right] {$\psi$} (m-2-1);
			\end{tikzpicture}
		\end{displaymath}
		which preserves cartesian cells.
	\end{corollary}
	Restricting to vertical identity cells $\phi = \id_f$ in the above we find that choosing a horizontal unit $I_A$ for each unital object $A$ in an augmented virtual double category extends uniquely to a functorial assignment
	\begin{displaymath}
		(\map fAC) \quad \qquad \mapsto \quad \qquad \begin{tikzpicture}[textbaseline]
				\matrix(m)[math35]{A & A \\ C & C, \\};
				\path[map]	(m-1-1) edge[barred] node[above] {$I_A$} (m-1-2)
														edge node[left] {$f$} (m-2-1)
										(m-1-2) edge node[right] {$f$} (m-2-2)
										(m-2-1) edge[barred] node[below] {$I_C$} (m-2-2);
				\path[transform canvas={xshift=1.75em}]	(m-1-1) edge[cell] node[right] {$I_f$} (m-2-1);
			\end{tikzpicture}
	\end{displaymath}
	where $f$ is full and faithful (\defref{full and faithful morphism}) if and only if $I_f$ is cartesian.
		
	The remainder of this section records some useful properties of companions, conjoints and horizontal units. The first of these is an immediate consequence of the pasting lemma for cartesian cells (\lemref{pasting lemma for cartesian cells}).
	\begin{lemma} \label{companion of a composite}
		Let $\map fAC$ and $\map hCE$ be morphisms and assume that the companion $\hmap{h_*}CE$ exists. The companion $(h \of f)_*$ exists if and only if the restriction $h_*(f, \id)$ does, and in that case they are isomorphic.
		\begin{displaymath}
			\begin{tikzpicture}[textbaseline]
  			\matrix(m)[math35, column sep={0.875em,between origins}]
  				{	A & & & & E \\
  					& C & & & \\
  					& & E & & \\ };
  			\path[map]	(m-1-1) edge[barred] node[above] {$J$} (m-1-5)
  													edge[transform canvas={xshift=-1pt}] node[left] {$f$} (m-2-2)
  									(m-2-2) edge[transform canvas={xshift=-1pt}] node[left] {$h$} (m-3-3);
  			\path				(m-1-3) edge[cell, transform canvas={yshift=-0.5em}] node[right] {$\chi$} (m-2-3)
  									(m-1-5) edge[eq, transform canvas={xshift=1pt}] (m-3-3);
  		\end{tikzpicture} \quad = \quad \begin{tikzpicture}[textbaseline]
  			\matrix(m)[math35, column sep={1.875em,between origins}]{A & & E \\ C & & E \\ & E & \\};
  			\path[map]	(m-1-1) edge[barred] node[above] {$J$} (m-1-3)
  													edge node[left] {$f$} (m-2-1)
  									(m-2-1) edge[barred] node[below] {$h_*$} (m-2-3)
  													edge[transform canvas={xshift=-2pt}] node[left] {$h$} (m-3-2);
  			\path				(m-1-3) edge[eq] (m-2-3)
  									(m-2-3) edge[eq, transform canvas={xshift=2pt}] (m-3-2)
  									(m-1-2) edge[cell] node[right] {$\psi$} (m-2-2);
        \draw[font=\scriptsize]	([yshift=0.25em]$(m-2-2)!0.5!(m-3-2)$) node {$\cart$};
  		\end{tikzpicture}
		\end{displaymath}
		In detail, the cell $\chi$ above is cartesian if and only if the cell $\psi$ is.
	\end{lemma}
	
	\begin{lemma} \label{pasting lemma for full and faithful morphisms}
	  Let $\map fAC$, $\map gBC$ and $\map hCE$ be morphisms and assume that $h$ is full and faithful. The nullary restriction $C(f, g)$ exists if and only if \mbox{$\hmap{E(h \of f, h \of g)}AC$} does, and in that case they are isomorphic.
		\begin{displaymath}
			\begin{tikzpicture}[textbaseline]
  			\matrix(m)[math35, column sep={0.875em,between origins}]
  				{	A & & & & B \\
  					& C & & C & \\
  					& & E & & \\ };
  			\path[map]	(m-1-1) edge[barred] node[above] {$J$} (m-1-5)
  													edge[transform canvas={xshift=-1pt}] node[left] {$f$} (m-2-2)
  									(m-2-2) edge[transform canvas={xshift=-1pt}] node[left] {$h$} (m-3-3)
  									(m-1-5) edge[transform canvas={xshift=1pt}] node[right] {$g$} (m-2-4)
  									(m-2-4) edge[transform canvas={xshift=1pt}] node[right] {$h$} (m-3-3);
  			\path				(m-1-3) edge[cell, transform canvas={yshift=-0.5em}] node[right] {$\chi$} (m-2-3);
  		\end{tikzpicture} \quad = \quad \begin{tikzpicture}[textbaseline]
				\matrix(m)[math35, column sep={1.875em,between origins}]{A & & B \\ & C & \\ & E & \\};
				\path[map]	(m-1-1) edge[barred] node[above] {$J$} (m-1-3)
														edge[transform canvas={xshift=-2pt}] node[left] {$f$} (m-2-2)
										(m-1-3) edge[transform canvas={xshift=2pt}] node[right] {$g$} (m-2-2)
										(m-2-2) edge[bend right=45] node[left] {$h$} (m-3-2)
														edge[bend left=45] node[right] {$h$} (m-3-2);
				\path				(m-1-2) edge[cell, transform canvas={yshift=0.25em}] node[right] {$\psi$} (m-2-2)
										(m-2-2) edge[cell, transform canvas={xshift=-0.35em}] node[right] {$\id$} (m-3-2);
			\end{tikzpicture}
		\end{displaymath}
		In detail, the cell $\chi$ above is cartesian if and only if the cell $\psi$ is.
	\end{lemma}
	\begin{proof}
		Because $h$ is full and faithful its identity cell is cartesian by \defref{full and faithful morphism}. The proof follows immediately from applying the pasting lemma (\lemref{pasting lemma for cartesian cells}) to the factorisation above.  
	\end{proof}

	Recall that any isomorphism $\map hCE$ is full and faithful (\exref{isomorphisms are full and faithful}), so that taking $g = \inv h$ in previous lemma gives the following.

	\begin{corollary} \label{companions of morphisms composed with an isomorphism}
		Let $\map fAC$ and $\map hCE$ be morphisms and assume that $h$ is an isomorphism. The companion $(h \of f)_*$ exists if and only if the nullary restriction $C(f, \inv h)$ does, and in that case they are isomorphic.
	\end{corollary}
	
	Together with \lemref{unit identities}, \lemref{pasting lemma for full and faithful morphisms} implies the following.
	\begin{lemma} \label{full and faithfulness and horizontal units}
		Consider the factorisation on the left below. Any two of the following properties imply the third:
		\begin{enumerate}[label=\textup{(\alph*)}]
			\item the cell $\chi$ is cartesian (defining $J$ as the nullary restriction $E(h, h)$);
			\item the cell $\psi$ is cartesian (defining $J$ as the horizontal unit of $C$);
			\item the morphism $h$ is full and faithful.
		\end{enumerate}
		Moreover if \textup{(a)} holds then \textup{(b)} is equivalent to
		\begin{enumerate}
			\item[\textup{(b')}] the factorisation $\id'$, as on the right below, is cartesian.
		\end{enumerate}
		\begin{displaymath}
			\begin{tikzpicture}[textbaseline]
  			\matrix(m)[math35, column sep={1.75em,between origins}]{C & & C \\ & E & \\};
				\path[map]	(m-1-1) edge[barred] node[above] {$J$} (m-1-3)
														edge[transform canvas={xshift=-1pt}] node[left] {$h$} (m-2-2)
										(m-1-3) edge[transform canvas={xshift=1pt}] node[right] {$h$} (m-2-2);
				\path[transform canvas={yshift=0.25em}]	(m-1-2) edge[cell] node[right, inner sep=3pt] {$\chi$} (m-2-2);
			\end{tikzpicture} \quad = \quad \begin{tikzpicture}[textbaseline]
				\matrix(m)[math35, column sep={1.75em,between origins}]{C & & C \\ & C & \\ & E & \\};
				\path[map]	(m-1-1) edge[barred] node[above] {$J$} (m-1-3)
										(m-2-2) edge[bend right=45] node[left] {$h$} (m-3-2)
														edge[bend left=45] node[right] {$h$} (m-3-2);
				\path				(m-1-1) edge[eq, transform canvas={xshift=-2pt}] (m-2-2)
										(m-1-3) edge[eq, transform canvas={xshift=2pt}] (m-2-2)
										(m-1-2)	edge[cell, transform canvas={yshift=0.25em}] node[right] {$\psi$} (m-2-2)
										(m-2-2) edge[cell, transform canvas={xshift=-0.35em}] node[right] {$\id$} (m-3-2);
			\end{tikzpicture} \qquad \qquad \quad \qquad \begin{tikzpicture}[textbaseline]
						\matrix(m)[math35]{C \\ E \\};
						\path[map]	(m-1-1) edge[bend right=45] node[left] {$h$} (m-2-1)
																edge[bend left=45] node[right] {$h$} (m-2-1);
						\path[transform canvas={xshift=-0.35em}]	(m-1-1) edge[cell] node[right] {$\id$} (m-2-1);
			\end{tikzpicture} \quad = \quad \begin{tikzpicture}[textbaseline]
    		\matrix(m)[math35, column sep={1.75em,between origins}]{& C & \\ C & & C \\ & E & \\};
    		\path[map]	(m-2-1) edge[barred] node[below, inner sep=2pt] {$J$} (m-2-3)
    												edge[transform canvas={xshift=-2pt}] node[left] {$h$} (m-3-2)
    								(m-2-3) edge[transform canvas={xshift=2pt}] node[right] {$h$} (m-3-2);
    								
    		\path				(m-1-2) edge[eq, transform canvas={xshift=-2pt}] (m-2-1)
    												edge[eq, transform canvas={xshift=2pt}] (m-2-3);
    		\path				(m-1-2) edge[cell, transform canvas={yshift=-0.25em, xshift=-0.45em}] node[right] {$\id'$} (m-2-2)
    								(m-2-2) edge[cell, transform canvas={yshift=0.1em}]	node[right] {$\chi$} (m-3-2);
  		\end{tikzpicture}
  	\end{displaymath}
		Consequently any full and faithful morphism $\map hCE$ in an augmented virtual equipment (\defref{augmented virtual equipment}) `reflects unitality': if the target $E$ is unital then so is the source $C$.
	\end{lemma}
	\begin{proof}
		We first prove the second assertion: if $\chi$ is cartesian then (b) $\Leftrightarrow$ (b'). Combining both identities above gives $\chi \of \id' \of \psi = \chi$, so that $\id' \of \psi = \id_J$ by uniqueness of factorisations through $\chi$. But this is identity (J) in \lemref{unit identities}, which asserts that $\psi$ is cartesian precisely if $\id'$ is, that is (b) $\Leftrightarrow$ (b').
		
		The main assertion now follows easily. Taking $f = \id_C = g$ in \lemref{pasting lemma for full and faithful morphisms} we find that (c) implies \mbox{(a) $\Leftrightarrow$ (b)}. Conversely assume (a) and (b): by the previous (b') follows so that both $\chi$ and $\id'$ in the identity on the right above are cartesian. Applying the pasting lemma (\lemref{pasting lemma for cartesian cells}) we find that $\id_h$ is cartesian, showing that $h$ is full and faithful.
				
		For the final assertion notice that $E$ being unital in an augmented virtual equipment implies that the nullary restriction $E(h, h)$ exists, by \cororef{restrictions in terms of units}. Applying the main assertion we find that a cartesian cell defining the horizontal unit of $C$ can be obtained by factorising the cartesian cell that defines $E(h, h)$ through $\id_h$.
	\end{proof}
	
	\begin{corollary}
		Let $\map hCE$ be a full and faithful morphism in an augmented virtual double category $\K$ and assume that the nullary restriction $E(h, h)$ exists. A functor $\map F\K\L$ preserves the full and faithful morphism $h$ if and only if it preserves the cartesian cell defining $E(h, h)$.
	\end{corollary}
	In \cororef{functors preserving cartesian cells} we will see that $F$ preserves $E(h, h)$ whenever the companion $h_*$ and the conjoint $h^*$ exist in $\K$.
	\begin{proof}
		Factorising the cartesian cell $\chi$ that defines $E(h, h)$ through the vertical identity cell of $h$ we obtain $\chi = \id_h \of \psi$ as in the previous lemma, where $\psi$ is the cartesian cell defining $E(h, h)$ as the horizontal unit of $C$. Applying $F$ to both factorisations considered in the previous lemma we obtain
		\begin{displaymath}
			F\chi = \id_{Fh} \of F\psi \qquad \qquad \qquad \text{and} \qquad \qquad \qquad \id_{Fh} = F\chi \of F\id_h',
		\end{displaymath}
		where $F$ preserves the cartesian cells $\psi$ and $\id_h'$ by \cororef{functors preserve companions and conjoints}. Applying the pasting lemma for cartesian cells (\lemref{pasting lemma for cartesian cells}) to these identities we conclude that $F\chi$ is cartesian precisely if $\id_{Fh}$ is.
	\end{proof}
	
	Recall from \exref{vertical 2-category} that the objects, vertical morphisms and vertical cells of any augmented virtual double category $\K$ form a $2$-category $V(\K)$. The next lemmas reformulate the notions of adjunction and \emph{absolute left lifting} (see Section~1 of \cite{Street-Walters78} or Section~2.4 of \cite{Weber07}) in $V(\K)$ in terms of companions in $\K$.
	\begin{lemma} \label{adjunctions}
		In an augmented virtual double category $\K$ let $\map fAC$ be a vertical morphism whose companion $f_*$ exists. Consider vertical cells $\eta$ and $\eps$ below as well as their factorisations through $f_*$, as shown.
		\begin{displaymath}
			\begin{tikzpicture}[textbaseline]
				\matrix(m)[math35, column sep={1.75em,between origins}]{& A & \\ & & C \\ & A & \\};
				\path[map]	(m-1-2) edge[bend left = 18] node[right] {$f$} (m-2-3)
										(m-2-3) edge[bend left = 18] node[right] {$g$} (m-3-2);
				\path				(m-1-2) edge[bend right = 45, eq] (m-3-2);
				\path[transform canvas={yshift=-1.625em}]	(m-1-2) edge[cell] node[right] {$\eta$} (m-2-2);
			\end{tikzpicture} \quad = \quad \begin{tikzpicture}[textbaseline]
    		\matrix(m)[math35, column sep={1.75em,between origins}]{& A & \\ A & & C \\ & A & \\};
    		\path[map]	(m-1-2) edge[transform canvas={xshift=2pt}] node[right] {$f$} (m-2-3)
    								(m-2-1) edge[barred] node[below, inner sep=2pt] {$f_*$} (m-2-3)
    								(m-2-3)	edge[transform canvas={xshift=2pt}] node[right] {$g$} (m-3-2);
    		\path				(m-1-2) edge[eq, transform canvas={xshift=-2pt}] (m-2-1)
    								(m-2-1) edge[eq, transform canvas={xshift=-2pt}] (m-3-2);
    		\path				(m-2-2) edge[cell, transform canvas={yshift=0.1em}]	node[right, inner sep=3pt] {$\eta'$} (m-3-2);
    		\draw				([yshift=-0.5em]$(m-1-2)!0.5!(m-2-2)$) node[font=\scriptsize] {$\cocart$};
  		\end{tikzpicture} \qquad\qquad\qquad \begin{tikzpicture}[textbaseline]
				\matrix(m)[math35, column sep={1.75em,between origins}]{& C & \\ A & & \\ & C & \\};
				\path[map]	(m-1-2) edge[bend right = 18] node[left] {$g$} (m-2-1)
										(m-2-1) edge[bend right = 18] node[left] {$f$} (m-3-2);
				\path				(m-1-2) edge[bend left = 45, eq] (m-3-2);
				\path[transform canvas={yshift=-1.625em}]	(m-1-2) edge[cell] node[right] {$\eps$} (m-2-2);
			\end{tikzpicture} \quad = \quad \begin{tikzpicture}[textbaseline]
    		\matrix(m)[math35, column sep={1.75em,between origins}]{& C & \\ A & & C \\ & C & \\};
    		\path[map]	(m-1-2) edge[transform canvas={xshift=-2pt}] node[left] {$g$} (m-2-1)
    								(m-2-1) edge[barred] node[below, inner sep=2pt] {$f_*$} (m-2-3)
    												edge[transform canvas={xshift=-2pt}] node[left] {$f$} (m-3-2);
    		\path				(m-1-2) edge[eq, transform canvas={xshift=2pt}] (m-2-3)
    								(m-2-3) edge[eq, transform canvas={xshift=2pt}] (m-3-2);
    		\path				(m-1-2) edge[cell, transform canvas={yshift=-0.4em}] node[right, inner sep=3pt] {$\eps'$} (m-2-2);
    		\draw				([yshift=0.25em]$(m-2-2)!0.5!(m-3-2)$) node[font=\scriptsize] {$\cart$};
  		\end{tikzpicture}
		\end{displaymath}
		The following are equivalent:
		\begin{enumerate}[label=\textup{(\alph*)}]
			\item	$(\eta, \eps)$ defines an adjunction $f \ladj g$ in $V(\K)$;
			\item $(\eta', \eps')$ satisfies the conjoint identities (horizontally dual to the companion identities of \lemref{companion identities lemma}), thus defining $f_*$ as the conjoint of $g$ in $\K$.
		\end{enumerate}
	\end{lemma}
	\begin{proof}
		We claim that the triangle identities for $\eta$ and $\eps$ in $V(\K)$ are equivalent to the conjoint identities $\eta' \of \eps' = \id_g$ and $\eta' \hc \eps' = \id_{f_*}$ in $\K$. Indeed we have
		\begin{align*}
			(f \of \eta) \hc (\eps \of f) = \id_f \; &\Leftrightarrow \; (f \of \eta' \of \cocart) \hc (\cart \of \eps' \of f) = \id_f \\
			&\Leftrightarrow \; \cart \of (\eta' \hc \eps') \of \cocart = \id_f \; \Leftrightarrow \; \eta' \hc \eps' = \id_{f_*},
		\end{align*}
		where the second equivalence follows from the interchange axioms (\lemref{horizontal composition}), and the third from the vertical companion identity $\cart \of \id_{f_*} \of \cocart = \id_f$ together with the uniqueness of factorisations through (co)cartesian cells. Likewise
		\begin{align*}
			(\eta \of g) \hc (g \of \eps) = \id_g \; &\Leftrightarrow \; (\eta' \of \cocart \of g) \hc (g \of \cart \of \eps') = \id_g \\
			&\Leftrightarrow \; \eta' \of (\cocart \hc \cart) \of \eps' = \id_g \; \Leftrightarrow \; \eta' \of \eps' = \id_g,
		\end{align*}
		where we used the horizontal companion identity $\cocart \hc \cart = \id_{f_*}$.
	\end{proof}
	
	The converse of the following holds whenever $\K$ has `all weakly cocartesian paths of $(0,1)$"/ary cells', see \propref{converses with weakly cocartesian paths of (0,1)-ary cells} below.
	\begin{lemma} \label{restrictions and absolute left liftings}
		In an augmented virtual double category $\K$ consider the factorisation below. The vertical cell $\psi$ defines $j$ as the absolute left lifting of $f$ along $g$ in $V(\K)$ whenever its factorisation $\psi'$ is cartesian in $\K$.
		\begin{displaymath}
			\begin{tikzpicture}[textbaseline]
				\matrix(m)[math35, column sep={1.75em,between origins}]{& A & \\ & & B \\ & C & \\};
				\path[map]	(m-1-2) edge[bend left = 18] node[above right] {$j$} (m-2-3)
														edge[bend right = 45] node[left] {$f$} (m-3-2)
										(m-2-3) edge[bend left = 18] node[below right] {$g$} (m-3-2);
				\path[transform canvas={yshift=-1.625em}]	(m-1-2) edge[cell] node[right] {$\psi$} (m-2-2);
			\end{tikzpicture} \quad = \quad \begin{tikzpicture}[textbaseline]
    		\matrix(m)[math35, column sep={1.75em,between origins}]{& A & \\ A & & B \\ & C & \\};
    		\path[map]	(m-1-2) edge[transform canvas={xshift=2pt}] node[right] {$j$} (m-2-3)
    								(m-2-1) edge[barred] node[below, inner sep=2pt] {$j_*$} (m-2-3)
    												edge[transform canvas={xshift=-2pt}] node[left] {$f$} (m-3-2)
    								(m-2-3)	edge[transform canvas={xshift=2pt}] node[right] {$g$} (m-3-2);
    		\path				(m-1-2) edge[eq, transform canvas={xshift=-2pt}] (m-2-1);
    		\path				(m-2-2) edge[cell, transform canvas={yshift=0.1em}]	node[right, inner sep=3pt] {$\psi'$} (m-3-2);
    		\draw				([yshift=-0.5em]$(m-1-2)!0.5!(m-2-2)$) node[font=\scriptsize] {$\cocart$};
  		\end{tikzpicture}
  	\end{displaymath}
	\end{lemma}
	\begin{proof}
		Consider the diagram of assignments between collections of cells in $\K$ below, where $\cart$ denotes the cartesian cell that defines $j_*$. That it commutes follows from the identity above and the horizontal companion identity (\lemref{companion identities lemma}).
		\begin{displaymath}
			\begin{tikzpicture}
				\matrix(m)[math35, column sep={1.75em,between origins}, xshift=-12em]{& X & \\ A & & \\ & B & \\};
				\path[map]	(m-1-2) edge[bend right = 18] node[left] {$h$} (m-2-1)
														edge[bend left = 45] node[right, inner sep=2pt] {$k$} (m-3-2)
										(m-2-1) edge[bend right = 18] node[left] {$j$} (m-3-2);
				\path[transform canvas={yshift=-1.625em}]	(m-1-2) edge[cell] node[right] {$\xi$} (m-2-2);
				
				\matrix(m)[math35, column sep={1.75em,between origins}]{& X & \\ A & & B \\};
				\path[map]	(m-1-2) edge[transform canvas={xshift=-2pt}] node[left] {$h$} (m-2-1)
														edge[transform canvas={xshift=2pt}] node[right] {$k$} (m-2-3)
										(m-2-1) edge[barred] node[below] {$j_*$} (m-2-3);
				\path				(m-1-2) edge[cell, transform canvas={yshift=-0.25em}] node[right] {$\phi$} (m-2-2);
				
				\matrix(m)[math35, column sep={1.75em,between origins}, xshift=12em]{& X & \\ A & & B \\ & C & \\};
				\path[map]	(m-1-2) edge[bend left = 18] node[right] {$k$} (m-2-3)
														edge[bend right = 18] node[left] {$h$} (m-2-1)
										(m-2-1) edge[bend right = 18] node[left] {$f$} (m-3-2)
										(m-2-3) edge[bend left = 18] node[right] {$g$} (m-3-2);
				\path[transform canvas={yshift=-1.625em}]	(m-1-2) edge[cell] node[right] {$\chi$} (m-2-2);
				
				\draw[font=\Large]	(-2.5em,0em) node {$\lbrace$}
										(2.5em,0) node {$\rbrace$}
										(-14.5em,0) node {$\lbrace$}
										(-9em,0) node {$\rbrace$}
										(9.25em,0) node {$\lbrace$}
										(14.75em,0) node {$\rbrace$};
				\path[map]	(-4.25em,0) edge node[above] {$\cart \of\, \dash$} (-7.25em,0)
										(4.25em,0) edge node[above] {$\psi' \of \dash$} (7.5em,0)
										(-8.5em,-2.5em) edge[bend right=15] node[below] {$(\psi \of h) \hc (g \of \dash)$} (8.75em,-2.5em);
			\end{tikzpicture}
		\end{displaymath}
		By definition the vertical cell $\psi$ defines $j$ as the absolute left lifting of $f$ along $g$ in $V(\K)$ when the bottom assignment is a bijection, so that the proof follows from the fact that both top assignments are bijections whenever $\psi'$ is cartesian.
	\end{proof}
	
	\section{Representable horizontal morphisms} \label{representable morphism section}
	In this section we study horizontal morphisms $\hmap JAB$ that are `represented' by vertical morphisms $\map fAB$ in the sense that $J \iso f_*$; see the definition below. Given an augmented virtual double category $\K$, the main result (\thmref{lower star}) of this section characterises the sub"/augmented virtual double category of $\K$ generated by representable horizontal morphisms, in terms of the strict double category \mbox{$(Q \of V)(\K)$} of `quintets' in the vertical $2$"/category $V(\K)$ (\exref{vertical 2-category}); see \exref{quintets} below.
	
	Generalising the fact that lax monoidal profunctors (as described in the Introduction) that are representable can be identified with colax monoidal functors, \thmref{lower star} can be used to obtain a correspondence between representable `horizontal $T$"/morphisms' and `colax $T$"/morphisms', where $T$ is any `monad' on an augmented virtual double category; this is done in Section~6.4 of \cite{Koudenburg19b}.
	
	\begin{definition} \label{representable horizontal morphism}
		A vertical morphism $\map jAB$ is said to \emph{represent} the horizontal morphism $\hmap JAB$ if there exists a cartesian cell as on the left below, that is $J$ forms the companion of $j$; in this case we say that $J$ is \emph{representable}. Horizontally dual, $J$ is called \emph{oprepresentable} whenever there exists a cartesian cell as on the right.
		\begin{displaymath}
			\begin{tikzpicture}[baseline]
				\matrix(m)[math35, column sep={1.75em,between origins}]{A & & B \\ & B & \\};
				\path[map]	(m-1-1) edge[barred] node[above] {$J$} (m-1-3)
														edge[transform canvas={xshift=-1pt}] node[left] {$j$} (m-2-2);
				\path				(m-1-3) edge[transform canvas={xshift=2pt}, eq] (m-2-2);
				\draw[font=\scriptsize]	([yshift=0.333em]$(m-1-2)!0.5!(m-2-2)$) node {$\cart$};
			\end{tikzpicture} \qquad\qquad\qquad\qquad\qquad \begin{tikzpicture}[baseline]
				\matrix(m)[math35, column sep={1.75em,between origins}]{A & & B \\ & A & \\};
				\path[map]	(m-1-1) edge[barred] node[above] {$J$} (m-1-3)
										(m-1-3)	edge[transform canvas={xshift=1pt}] node[right] {$g$} (m-2-2);
				\path				(m-1-1) edge[eq, transform canvas={xshift=-2pt}] (m-2-2);
				\draw				([yshift=0.333em]$(m-1-2)!0.5!(m-2-2)$) node[font=\scriptsize] {$\cart$};
			\end{tikzpicture}
    \end{displaymath}
	\end{definition}
	
	For an augmented virtual double category $\K$ we write $\Rep(\K) \subseteq \K$ for the sub"/augmented virtual double category that consists of all objects, all vertical morphisms, only those horizontal morphisms that are representable, and all cells between them. The subcategory $\opRep(\K)$ generated by the oprepresentable horizontal morphisms is defined analogously; notice that $\opRep(\K) = \co{\pars{\Rep(\co\K)}}$ where $\co\K$ denotes the horizontal dual of $\K$ (\defref{horizontal dual}). Because functors of augmented virtual double categories preserve companions and conjoints (\cororef{functors preserve companions and conjoints}), they preserve (op)representable horizontal morphisms as well; whence the following.
	\begin{proposition} \label{2-functor Rep}
		The assignments $\K \mapsto \Rep(\K)$ and $\K \mapsto \opRep(\K)$ extend to strict $2$-endofunctors $\Rep$ and $\opRep$ on $\AugVirtDblCat$.
	\end{proposition}
	
	In \cite{Ehresmann63} Ehresmann defined for any $2$-category $\mathcal C$ a strict double category $Q(\mathcal C)$ of `quintets' in $\mathcal C$. The following example describes $Q(\mathcal C)$ as an augmented virtual double category. 
	\begin{example} \label{quintets}
		Let $\mathcal C$ be a $2$-category. The augmented virtual double category $Q(\mathcal C)$ of \emph{quintets in $\mathcal C$} has as objects those of $\mathcal C$ while both its vertical and horizontal morphisms are morphisms in $\mathcal C$. A unary cell $\phi$ in $Q(\mathcal C)$, as in the middle below, is a cell $\phi$ in $\mathcal C$ as on the left, while the nullary cells of $Q(\mathcal C)$ are cells in $\mathcal C$ as on the left but with $k = \id_C$. Composition in $Q(\mathcal C)$ is induced by that of $\mathcal C$ in the evident way.
		\begin{displaymath}
			\begin{tikzpicture}[baseline]
				\matrix(m)[math35, column sep={1.75em,between origins}]
					{ & A_0 & & & \\
						C & & A_1 & & \\
						& \phantom{.} & & A_{n'} &\\
						& & & & A_n \\
						& & & D & \\ };
				\path[map]	(m-1-2) edge node[above left] {$f$} (m-2-1)
														edge node[above right] {$j_1$} (m-2-3)
										(m-2-1) edge node[below left] {$k$} (m-5-4)
										(m-3-4) edge node[above right] {$j_n$} (m-4-5)
										(m-4-5) edge node[below right] {$g$} (m-5-4);
				\path				(m-3-2) edge[cell, transform canvas={xshift=3pt}] node[above] {$\phi$} (m-3-4)
										(m-2-3)	edge[white] node[sloped, black] {$\dotsb$} (m-3-4);
			\end{tikzpicture} \qquad\quad \begin{tikzpicture}[baseline]
				\matrix(m)[math35, yshift=-4em]{A_n & A_{n'} & A_1 & A_0 \\ D & & & C \\};
				\path[map]	(m-1-1) edge[barred] node[above] {$\co j_n$} (m-1-2)
														edge node[left] {$g$} (m-2-1)
										(m-1-3) edge[barred] node[above] {$\co j_1$} (m-1-4)
										(m-1-4) edge node[right] {$f$} (m-2-4)
										(m-2-1) edge[barred] node[below] {$\co k$} (m-2-4);
				\path[transform canvas={xshift=1.75em}]	(m-1-2) edge[cell] node[right] {$\co\psi$} (m-2-2);
				\draw				($(m-1-2)!0.5!(m-1-3)$) node {$\dotsb$};
				
				\matrix(n)[math35, yshift=4em]{A_0 & A_1 & A_{n'} & A_n \\ C & & & D \\};
				\path[map]	(n-1-1) edge[barred] node[above] {$j_1$} (n-1-2)
														edge node[left] {$f$} (n-2-1)
										(n-1-3) edge[barred] node[above] {$j_n$} (n-1-4)
										(n-1-4) edge node[right] {$g$} (n-2-4)
										(n-2-1) edge[barred] node[below] {$k$} (n-2-4);
				\path[transform canvas={xshift=1.75em}]	(n-1-2) edge[cell] node[right] {$\phi$} (n-2-2);
				\draw				($(n-1-2)!0.5!(n-1-3)$) node {$\dotsb$};
			\end{tikzpicture} \qquad\quad \begin{tikzpicture}[baseline]
				\matrix(m)[math35, column sep={1.75em,between origins}]
					{ & & & A_0 & \\
						& & A_1 & & C \\
						& A_{n'} & & \phantom{.} & \\
						A_n & & & & \\
						& D & & & \\ };
				\path[map]	(m-1-4) edge node[above right] {$f$} (m-2-5)
														edge node[above left] {$j_1$} (m-2-3)
										(m-2-5) edge node[below right] {$k$} (m-5-2)
										(m-3-2) edge node[above left] {$j_n$} (m-4-1)
										(m-4-1) edge node[below left] {$g$} (m-5-2);
				\path				(m-3-2) edge[cell] node[above] {$\psi$} (m-3-4)
										(m-2-3)	edge[white] node[sloped, black] {$\dotsb$} (m-3-2);
			\end{tikzpicture}
		\end{displaymath}

		We abbreviate $\co Q(\mathcal C) \dfn \co{(Q(\co{\mathcal C}))}$, where $\co{\mathcal C}$ denotes the $2$"/category obtained by reversing the direction of the cells in $\mathcal C$. Thus, to each morphism $\map jAB$ in $\mathcal C$ there is a horizontal morphism $\hmap{\co j}BA$ in $\co Q(\mathcal C)$, and to each cell $\phi$ as on the left below there is a unary cell $\co\phi$ in $\co Q(\mathcal C)$ as on the right.
	\end{example}
	
	\begin{proposition} \label{2-functor Q}
		The assignments $\mathcal C \mapsto Q(\mathcal C)$ and $\mathcal C \mapsto \co Q(\mathcal C)$ above extend to strict $2$-functors $\map Q{\twoCat}{\AugVirtDblCat}$ and $\map{\co Q}{\twoCat}{\AugVirtDblCat}$.
	\end{proposition}
	\begin{proof}
		The image $\map{QF}{Q\mathcal C}{Q\mathcal D}$ of a strict $2$-functor $\map F{\mathcal C}{\mathcal D}$ is simply given by letting $F$ act on objects, morphisms and cells. The image $\nat{Q\xi}{QF}{QG}$ of a $2$"/natural transformation $\nat\xi FG$ is given by $(Q\xi)_A \dfn \xi_A$ on objects, while the naturality cell $\cell{(Q\xi)_j}{Fj}{Gj}$ in $Q(\mathcal D)$, for $\hmap jAB$ in $Q(\mathcal C)$, is the quintet given by the naturality square $Gj \of \xi_A = \xi_B \of Fj$. Finally $\mathcal C \mapsto \co Q(\mathcal C)$ is extended by the composite of strict $2$-functors $\co Q \dfn \co{(\dash)} \of Q \of \co{(\dash)}$.
	\end{proof}
	
	Remember that any augmented virtual double category $\K$ contains a $2$-category $V(\K)$ of vertical morphisms and cells; see \exref{vertical 2-category}. We denote by \mbox{$(Q \of V)_*(\K) \subseteq (Q \of V)(\K)$} the sub"/augmented virtual double category generated by all vertical morphisms, those horizontal morphisms $\hmap jAB$ that correspond to morphisms $\map jAB$ that admit companions in $\K$, and all quintets between them. Because functors between augmented virtual double categories preserve cartesian cells that define companions (\cororef{functors preserve companions and conjoints}), this gives a sub"/2"/endofunctor \mbox{$(Q \of V)_* \subseteq Q \of V$} on $\AugVirtDblCat$. The sub-2-endofunctor $(\co Q \of V)^*$ is defined likewise, by mapping each $\K$ to the sub"/augmented virtual double category $(\co Q \of V)^*(\K) \subseteq (\co Q \of V)(\K)$ that is generated by horizontal morphisms $\hmap{\co j}BA$ corresponding to vertical morphism $\map jAB$ that admit conjoints in $\K$.
	\begin{theorem} \label{lower star}
		Let $\K$ be an augmented virtual double category. Choosing for each \mbox{$\hmap jAB$} in $(Q \of V)_*(\K)$, corresponding to $\map jAB$ in $\K$, a cartesian cell
		\begin{displaymath}
  		\begin{tikzpicture}
				\matrix(m)[math35, column sep={1.75em,between origins}]{A & & B \\ & B & \\};
				\path[map]	(m-1-1) edge[barred] node[above] {$j_*$} (m-1-3)
														edge node[left] {$j$} (m-2-2);
				\path				(m-1-3) edge[eq, transform canvas={xshift=1pt}] (m-2-2);
				\path[transform canvas={yshift=0.25em, xshift=-0.3em}]	(m-1-2) edge[cell] node[right, inner sep=3pt] {$\eps_j$} (m-2-2);
			\end{tikzpicture}
		\end{displaymath}
		in $\K$ that defines the companion of $j$ induces and equivalence $(\dash)_*\colon (Q \of V)_*(\K) \xrar{\simeq} \Rep(\K)$
		of augmented virtual double categories as follows. Restricting to the identity on objects and vertical morphisms, $(\dash)_*$ maps each horizontal morphism $\hmap jAB$ in $(Q \of V)_*(K)$ to its chosen companion $j_*$, while a cell $\phi$ of $(Q \of V)_*(K)$, as in the left-hand side below, is mapped to the unique factorisation $\phi_*$ as shown; here $\ul{\eps_k} \dfn \eps_k$ if $\phi$ is unary and $\ul{\eps_k} \dfn \id_C$ otherwise.
		
		Letting $\K$ vary in the above, the functors $(\dash)_*$ combine to form a pseudonatural transformation $\nat{(\dash)_*}{(Q \of V)_*}\Rep$ of strict $2$-endofunctors on $\AugVirtDblCat$.
		
		Analogously, choosing cartesian cells in $\K$ that define conjoints induces an equivalence $(\co Q \of V)^*(\K)\simeq\opRep(\K)$. Their underlying functors too combine to form a pseudonatural transformation $\nat{(\dash)^*}{(\co Q \of V)^*}\opRep$.
		\begin{equation} \label{lower star on quintets}
			\begin{tikzpicture}[textbaseline]
				\matrix(m)[math35, column sep={1.75em,between origins}]
					{ & A_0 & & A_1 & & A_{n'} & & A_n \\
						C & & A_1 & & & & & \\
						& & & A_{n'} & & A_n & \\
						& & & & A_n & & \\
						& & & D & & & \\ };
				\path[map]	(m-1-2) edge[barred] node[above] {$j_{1*}$} (m-1-4)
														edge node[left] {$f$} (m-2-1)
														edge node[left, inner sep=2pt] {$j_1$} (m-2-3)
										(m-1-6) edge[barred] node[above] {$j_{n*}$} (m-1-8)
										(m-2-1) edge node[below left] {$\ul k$} (m-5-4)
										(m-3-4) edge[barred] node[above] {$j_{n*}$} (m-3-6)
														edge node[left] {$j_n$} (m-4-5)
										(m-4-5) edge node[right] {$g$} (m-5-4);
				\path[transform canvas={xshift=1pt}]	(m-1-4) edge[eq] (m-2-3)
										(m-1-6) edge[eq] (m-3-4)
										(m-1-8) edge[eq] (m-3-6)
										(m-3-6) edge[eq] (m-4-5);
				\path				(m-2-3) edge[cell, transform canvas={shift={(-0.25em,-1.625em)}}] node[right] {$\phi$} (m-3-3)
										(m-1-3) edge[cell, transform canvas={shift={(-0.55em,0.333em)}}] node[right] {$\eps_{j_1}$} (m-2-3)
										(m-3-5) edge[cell, transform canvas={shift={(-0.55em,0.333em)}}] node[right] {$\eps_{j_n}$} (m-4-5)
										(m-2-3)	edge[white] node[sloped, black] {$\dotsb$} (m-3-4)
										(m-1-4) edge[white] node[sloped, black] {$\dotsb$} (m-1-6);
			\end{tikzpicture} \mspace{-9mu}= \quad \begin{tikzpicture}[textbaseline]
				\matrix(m)[math35, yshift=1.625em]{A_0 & A_1 & A_{n'} & A_n \\ C & & & D \\};
				\draw	([yshift=-6.5em]$(m-1-1)!0.5!(m-1-4)$) node (D) {$D$};
				\path[map]	(m-1-1) edge[barred] node[above] {$j_{1*}$} (m-1-2)
														edge node[left] {$f$} (m-2-1)
										(m-1-3) edge[barred] node[above] {$j_{n*}$} (m-1-4)
										(m-1-4) edge node[right] {$g$} (m-2-4)
										(m-2-1) edge[transform canvas={yshift=-1pt}] node[below left] {$\ul k$} (D)
										(m-2-1) edge[barred] node[below] {$\ul{k_*}$} (m-2-4);
				\path				(m-2-4)	edge[eq, transform canvas={yshift=-1pt}] (D)
										(m-1-2) edge[cell, transform canvas={xshift=1.75em}] node[right] {$\phi_*$} (m-2-2)
										($(m-2-1.south)!0.5!(m-2-4.south)$) edge[cell, transform canvas={yshift=-0.125em}] node[right] {$\ul{\eps_k}$} (D);
				\draw				($(m-1-2)!0.5!(m-1-3)$) node {$\dotsb$};
			\end{tikzpicture}
		\end{equation}
	\end{theorem}
	\begin{proof}
		We will construct the functors $\map{(\dash)_*}{(Q \of V)_*(\K)}{\Rep(\K)}$; show that they are full, faithful and essentially surjective, so that they are part of equivalences by \propref{equivalences}; and prove that they are pseudonatural in $\K$. Horizontally dual, the functors \mbox{$\map{(\dash)^*}{(\co Q \of V)^*(\K)}{\opRep(\K)}$} can then be defined as the composites $(\dash)^* \dfn \co{(\dash)} \of (\dash)_* \of \co{(\dash)}$, where we use that companions in $\co\K$ correspond to conjoints in $\K$, so that $\co{\pars{(Q \of V)_*(\co\K)}} = (\co Q \of V)^*(\K)$ and $\co{\pars{\Rep(\co\K)}} = \opRep(\K)$.
		
		It is clear that $\phi \mapsto \phi_*$ preserves identities. To see that it preserves composites $\psi \of (\phi_1, \dotsc, \phi_n)$ too consider the following equation where, as in \eqref{lower star on quintets} above, each cell denoted $\ul \eps$ is either an identity or one of the chosen cartesian cells $\ul \eps = \eps_k$. The identities follow from \eqref{lower star on quintets} above and the definition of composition in $(Q \of V)(\K)$. We conclude that $\psi_* \of (\phi_{1*}, \dotsc, \phi_{n*})$ and $\bigpars{\psi \of (\phi_1, \dotsc, \phi_n)}_*$ coincide after composition with the cell $\ul \eps$ used in the definition of $\psi_*$. By uniqueness of factorisations through cartesian cells we conclude that these composites themselves coincide, showing that the composite $\psi \of (\phi_1, \dotsc, \phi_n)$ is preserved by $(\dash)_*$.
		\begin{align*}
			\begin{tikzpicture}[scheme]
				\draw	(1,3) -- (0,3) -- (0,1) -- (4.5,0) -- (9,1) -- (9,3) -- (8,3) (2,3) -- (3,3) -- (3,2) -- (0,2) (7,3) -- (6,3) -- (6,2) -- (9,2) (0,1) -- (9,1);
				\draw[shift={(0.5,0.5)}]	(1,2) node {$\phi_{1*}$}
							(7,2) node {$\phi_{n*}$}
							(4,1) node {$\psi_*$}
							(4,0) node {$\ul\eps$}
							(1.05,2.5) node[font=\tiny] {$\dotsb$}
							(7.05,2.5) node[font=\tiny] {$\dotsb$}
							(4,2) node {$\dotsb$};
			\end{tikzpicture} \quad & = \quad \begin{tikzpicture}[scheme, yshift=1.6em]
				\draw	(1,3) -- (0,3) -- (0,2) -- (-0.666,1) -- (3.833,-2) -- (4.5,-1) -- (9,2) -- (9,3) -- (8,3) (2,3) -- (3,3) -- (3,2) -- (1.5,1) -- (0,2) (7,3) -- (6,3) -- (6,2) -- (3,0) -- (4.5,-1) (0,2) -- (3,2) (6,2) -- (9,2) (3,0) -- (6,0);
				\draw[shift={(0.5,0.5)}]	(1,2) node {$\phi_{1*}$}
							(7,2) node {$\phi_{n*}$}
							(1,1.125) node {$\ul\eps$}
							(4,-0.875) node {$\ul\eps$}
							(1.584,-0.5) node {$\psi$}
							(1.05,2.5) node[font=\tiny] {$\dotsb$}
							(7.05,2.5) node[font=\tiny] {$\dotsb$}
							(4,2) node {$\dotsb$}
							(2.5,0.5) node[rotate=-33.75] {$\dotsb$};
			\end{tikzpicture} \\[-0.975em]
			& \hspace{-5cm} = \begin{tikzpicture}[scheme, yshift=1.6em]
				\draw	(0,2) -- (0.5,1.666) -- (1,2) -- (0,2) -- (-1.333,0) -- (3.166,-3) -- (4.5,-1) -- (9,2) -- (8,2) -- (4,-0.666) -- (4.5,-1) (-0.666,1) -- (0.833,0) -- (1.5,1) -- (3,2) -- (2,2) -- (1,1.333) -- (1.5,1) (3,0) -- (3.5,-0.333) -- (7,2) -- (6,2) -- (3,0) -- (2.333,-1) -- (3.833,-2) (1,1.333) -- (2,1.333) (3,0) -- (4,0) (4,-0.666) -- (5,-0.666);
				\draw	(0.5,1.85) node {$\eps$}
							(1.5,1.19) node {$\eps$}
							(3.5,-0.15) node {$\eps$}
							(4.5,-0.815) node {$\eps$}
							(0.41625,1) node {$\phi_1$}
							(3.45,-1) node {$\phi_n$}
							($(-1.333,0)!0.5!(3.825,-2)$) node {$\psi$}
							($(0.42,1)!0.5!(3.45,-1)$) node[rotate=-33.75] {$\dotsb$}
							(3.5,1.333) node {$\dotsb$}
							(1,1.666) node[font=\tiny, rotate=-33.75] {$\dotsb$}
							(5.5,0.666) node[font=\tiny, rotate=-33.75] {$\dotsb$};
			\end{tikzpicture} \mspace{-37mu} \begin{tikzpicture}[scheme, yshift=-1.6em]
				\draw	(1,2) -- (0,2) -- (0,1) -- (4.5,0) -- (9,1) -- (9,2) -- (8,2) (2,2) -- (3,2) (7,2) -- (6,2) (0,1) -- (9,1);
				\draw[shift={(0.5,0.5)}]	(4,1) node {$\bigpars{\psi \of (\phi_1, \dotsc, \phi_n)}_*$}
							(4,0) node {$\ul\eps$}
							(1.05,1.5) node[font=\tiny] {$\dotsb$}
							(7.05,1.5) node[font=\tiny] {$\dotsb$}
							(4,1.5) node {$\dotsb$}
							(-1.4,0.75) node[font=] {$\mspace{9mu}=$};
			\end{tikzpicture}
		\end{align*}
		
		To prove that $(\dash)_*$ is part of an equivalence it suffices by \propref{equivalences} to show that it is full, faithfull and essentially surjective. That it is essentially surjective and full and faithful on vertical morphisms is clear; we have to show that it is locally full and faithful, that is full and faithful on cells. To see this we denote, for each $\hmap jAB$ in $(Q \of V)_*(\K)$, by $\eta_j$ the weakly cocartesian cell that corresponds to $\eps_j$ as in \lemref{companion identities lemma}, so that the pair $(\eps_j, \eta_j)$ satisfies the companion identities. To show faithfulness, consider cells $\psi$ and $\cell\chi{\ul j}{\ul k}$ in $(Q \of V)_*(\K)$ such that $\psi_* = \chi_*$. It follows that the left-hand sides of \eqref{lower star on quintets} coincide for $\phi = \psi$ and $\phi = \chi$ so that, by precomposing both with the cells $\eta_{j_1}, \dotsc, \eta_{j_n}$, $\psi = \chi$ follows from the vertical companion identities. To show fullness on unary cells, consider \mbox{$\cell\psi{(j_{1*}, \dotsc, j_{n*})}{k_*}$} in \mbox{$(Q \of V)_*(\K)$}. We claim that the composite
		\begin{displaymath}
			\phi \dfn \eps_k \of \psi \of (\eta_{j_1}', \dotsc, \eta_{j_n}'),
		\end{displaymath}
		where $\eta_{j_i}' \dfn \eta_{j_i} \of j_{i'} \of \dotsb \of j_1$ for each $i = 1, \dotsc, n$, is mapped to $\psi$ by $(\dash)_*$. Indeed, plugging $\phi$ into the left-hand side of \eqref{lower star on quintets} we find $\eps_k \of \psi = \eps_k \of \phi_*$ by using the horizontal companion identities, so that $\psi = \phi_*$ follows. The case of a nullary cell $\cell\psi{(j_{1*}, \dotsc, j_{n*})}C$ is similar: simply take $\phi \dfn \psi \of (\eta_{j_1}', \dotsc, \eta_{j_n}')$ instead.
		
		We now turn to proving that the functors $(\dash)_*$ combine to form a pseudonatural transformation $(Q \of V)_* \Rar \Rep$ of strict $2$-endofunctors on $\AugVirtDblCat$. We have to supply an invertible transformation $\nu_F$ as on the left below, for each functor $\map F\K\L$ of augmented virtual double categories. We take $\nu_F$ to consist of identities $(\nu_F)_A = \id_{FA}$ on objects and, for each $\hmap jAB$ in $(Q \of V)_*(\K)$, the unique factorisation \mbox{$\cell{(\nu_F)_j}{F(j_*)}{(Fj)_*}$} as on the right below. The latter is invertible since $F\eps_j$, on the left-hand side, is cartesian by \cororef{functors preserve companions and conjoints}.
		\begin{displaymath}
			\begin{tikzpicture}[textbaseline]
				\matrix(m)[math35, column sep=1.75em]{(Q \of V)_* (\K) & \Rep(\K) \\ (Q \of V)_* (\L) & \Rep(\L) \\};
				\path[map]	(m-1-1) edge node[above] {$(\dash)_*$} (m-1-2)
														edge node[left] {$(Q \of V)_*(F)$} (m-2-1)
										(m-1-2) edge node[right] {$\Rep(F)$} (m-2-2)
										(m-2-1) edge node[below] {$(\dash)_*$} (m-2-2);
				\path				(m-1-2) edge[cell, shorten >= 1.5em, shorten <= 1.5em] node[below right] {$\nu_F$} (m-2-1);				
			\end{tikzpicture} \qquad\qquad \begin{tikzpicture}[textbaseline]
				\matrix(m)[math35, column sep={1.75em,between origins}]{FA & & FB \\ & FB & \\};
				\path[map]	(m-1-1) edge[barred] node[above] {$F(j_*)$} (m-1-3)
														edge[transform canvas={xshift=-2pt}] node[left] {$Fj$} (m-2-2);
				\path				(m-1-3) edge[eq, transform canvas={xshift=1pt}] (m-2-2);
				\path[transform canvas={shift={(-0.7em,0.25em)}}]	(m-1-2) edge[cell] node[right, inner sep=2.5pt] {$F\eps_j$} (m-2-2);
			\end{tikzpicture} \quad = \quad \begin{tikzpicture}[textbaseline]
  			\matrix(m)[math35, column sep={1.75em,between origins}]{FA & & FB \\ FA & & FB \\ & FB & \\};
  			\path[map]	(m-1-1) edge[barred] node[above] {$F(j_*)$} (m-1-3)
  									(m-2-1) edge[barred] node[above] {$(Fj)_*$} (m-2-3)
  													edge[transform canvas={xshift=-2pt}] node[left] {$Fj$} (m-3-2);
  			\path				(m-1-1) edge[eq] (m-2-1)
  									(m-1-3) edge[eq] (m-2-3)
  									(m-2-3) edge[eq, transform canvas={xshift=1pt}] (m-3-2);
  			\path				(m-1-2) edge[cell, transform canvas={shift={(-0.95em,0.25em)}}] node[right] {$(\nu_F)_j$} (m-2-2)
  									(m-2-2) edge[cell, transform canvas={shift={(-0.7em,0.25em)}}] node[right] {$\eps_{Fj}$} (m-3-2);
  		\end{tikzpicture}
		\end{displaymath}
		We have to show that the components of $\nu_F$ are natural with respect to the cells of $(Q \of V)_*(\K)$, in the sense of \defref{transformation}. We will do so in the case of a unary cell $\cell\phi{(j_1, \dotsc, j_n)}k$; the case of nullary cells is similar. Consider the following equation, where $\eps_{Fj_i}' \dfn Fg \of Fj_n \of \dotsb \of Fj_{i+1} \of \eps_{Fj_i}$ and $\eps_{j_i}' = g \of j_n \of \dotsb \of j_{i+1} \of \eps_{j_i}$ for each $i = 1, \dotsc, n$, as in the left-hand side of \eqref{lower star on quintets}. The identities follow from \eqref{lower star on quintets} for $F\phi$, the identity above, $F$ preserves composition, the $F$-image of \eqref{lower star on quintets} for $\phi$ and the identity above again. Since factorisations through $\eps_{Fk}$ in the left and right-hand side below are unique the naturality of the components of $\nu_F$ with respect to $\phi$ follows. This completes the definition of the transformation $\nu_F$.
		\begin{multline*}
			\eps_{Fk} \of (F\phi)_* \of \bigpars{(\nu_F)_{j_1}, \dotsc, (\nu_F)_{j_n}} = (F\phi \hc \eps_{Fj_1}' \hc \dotsb \hc \eps_{Fj_n}') \of \bigpars{(\nu_F)_{j_1}, \dotsc, (\nu_F)_{j_n}} \\ = F\phi \hc F\eps_{j_1}' \dotsb \hc F\eps_{j_n}' = F(\phi \hc \eps_{j_1}' \hc \dotsb \hc \eps_{j_n}') = F(\eps_k \of \phi_*) = \eps_{Fk} \of (\nu_F)_k \of F(\phi_*)
		\end{multline*}
		
		Finally we have to show that the transformations $\nu_F$ are natural with respect to the transformations $\nat\xi FG$ in $\AugVirtDblCat$, and that they are compatible with compositions and identities, that is $\nu_{\id} = \id$ and $\nu_GF \of G\nu_F = \nu_{G \of F}$. The latter is a direct consequence of the uniqueness of the components of $\nu$. To prove the former we have to show that $(\nu_G)_j \of \xi_{(j_*)} = (\xi_j)_* \of (\nu_F)_j$, for each $\hmap jAB$ in $(Q \of V)_*(\K)$. To do so consider the equation
		\begin{displaymath}
			\eps_{Gj} \of (\nu_G)_j \of \xi_{(j_*)} = G\eps_j \of \xi_{(j_*)} = \xi_B \of F\eps_j = \xi_B \of \eps_{Fj} \of (\nu_F)_j = \eps_{Gj} \of (\xi_j)_* \of (\nu_F)_j,
		\end{displaymath}
		where we have used the defining identities for $(\nu_G)_j$ and $(\nu_F)_j$, the naturality of $\xi$, identity~\eqref{lower star on quintets} for $\xi_j$, and the fact that the latter is simply the quintet given by the naturality square $Gj \of \xi_A = \xi_B \of Fj$; see the proof of \propref{2-functor Q}. Using the cartesianess of $\eps_{Gj}$ we conclude that $(\nu_G)_j \of \xi_{(j_*)} = (\xi_j)_* \of (\nu_F)_j$, proving the naturality of the transformations $\nu_F$. This concludes the proof.
	\end{proof}
	
	\section{Composition of horizontal morphisms}\label{composition section}
	We now turn to compositions of horizontal morphisms in augmented virtual double categories. Analogous to the case of virtual double categories (see Section~2 of \cite{Dawson-Pare-Pronk06} or Section~5 of \cite{Cruttwell-Shulman10}) such composites are defined by horizontal `cocartesian cells', whose universal property strengthens that of horizontal weakly cocartesian cell (\defref{cartesian cells}), as in the following definition. Generalising the latter notions in the obvious way, it also defines (weakly) cocartesian \emph{paths} of cells that are not necessarily horizontal.
	
	\begin{definition} \label{cocartesian paths}
		A path of cells $\ul\phi = (\phi_1, \dotsc, \phi_n)$, as in the right-hand side below, is called \emph{weakly cocartesian} if any cell $\chi$, as on the left-hand side, factors uniquely through $\ul\phi$ as shown.
		\begin{multline*}
			\begin{tikzpicture}[textbaseline, ampersand replacement=\&]
				\matrix(m)[minimath, column sep={0.75em}]
					{	X_{10} \& X_{11} \& X_{1m'_1} \& X_{1m_1} \&[4em] X_{n0} \& X_{n1} \& X_{nm'_n} \& X_{nm_n} \\
						C \& \& \& \& \& \& \& D \\};
				\path[map]	(m-1-1) edge[barred] node[above] {$H_{11}$} (m-1-2)
														edge node[left] {$h \of f_0$} (m-2-1)
										(m-1-3) edge[barred] node[above, xshift=-1pt] {$H_{1m_1}$} (m-1-4)
										(m-1-5)	edge[barred] node[above] {$H_{n1}$} (m-1-6)
										(m-1-7) 	edge[barred] node[above, xshift=-2pt] {$H_{nm_n}$} (m-1-8)
										(m-1-8)	edge node[right] {$k \of f_n$} (m-2-8)
										(m-2-1) edge[barred] node[below] {$\ul K$} (m-2-8);
				\path				($(m-1-1.south)!0.5!(m-1-8.south)$) edge[cell] node[right] {$\chi$} ($(m-2-1.north)!0.5!(m-2-8.north)$);
				\draw[transform canvas={xshift=-1pt}]	($(m-1-2)!0.5!(m-1-3)$) node {$\dotsb$}
										($(m-1-6)!0.5!(m-1-7)$) node {$\dotsb$};
				\draw				($(m-1-4)!0.5!(m-1-5)$) node {$\dotsb$};
			\end{tikzpicture} \\
			= \begin{tikzpicture}[textbaseline, ampersand replacement=\&]
				\matrix(m)[minimath, column sep={0.75em}]
					{	X_{10} \& X_{11} \& X_{1m'_1} \& X_{1m_1} \&[4em] X_{n0} \& X_{n1} \& X_{nm'_n} \& X_{nm_n} \\
						A_0 \& \& \& A_1 \& A_{n'} \& \& \& A_n \\
						C \& \& \& \& \& \& \& D \\};
				\path[map]	(m-1-1) edge[barred] node[above] {$H_{11}$} (m-1-2)
														edge node[left] {$f_0$} (m-2-1)
										(m-1-3) edge[barred] node[above, xshift=-1pt] {$H_{1m_1}$} (m-1-4)
										(m-1-4) edge node[right] {$f_1$} (m-2-4)
										(m-1-5)	edge[barred] node[above] {$H_{n1}$} (m-1-6)
														edge node[left] {$f_{n'}$} (m-2-5)
										(m-1-7) 	edge[barred] node[above, xshift=-2pt] {$H_{nm_n}$} (m-1-8)
										(m-1-8)	edge node[right] {$f_n$} (m-2-8)
										(m-2-1) edge[barred] node[below] {$\ul J_1$} (m-2-4)
														edge node[left] {$h$} (m-3-1)
										(m-2-5)	edge[barred] node[below] {$\ul J_n$} (m-2-8)
										(m-2-8)	edge node[right] {$k$} (m-3-8)
										(m-3-1) edge[barred] node[below] {$\ul K$} (m-3-8);
				\path				($(m-1-1.south)!0.5!(m-1-4.south)$) edge[cell] node[right] {$\phi_1$} ($(m-2-1.north)!0.5!(m-2-4.north)$)
										($(m-1-5.south)!0.5!(m-1-8.south)$) edge[cell] node[right] {$\phi_n$} ($(m-2-5.north)!0.5!(m-2-8.north)$)
										($(m-2-1.south)!0.5!(m-2-8.south)$) edge[cell] node[right] {$\chi'$} ($(m-3-1.north)!0.5!(m-3-8.north)$);
				\draw[transform canvas={xshift=-1pt}]	($(m-1-2)!0.5!(m-1-3)$) node {$\dotsb$}
										($(m-1-6)!0.5!(m-1-7)$) node {$\dotsb$};
				\draw				($(m-1-4)!0.5!(m-2-5)$) node {$\dotsb$};
			\end{tikzpicture}
		\end{multline*}

		A weakly cocartesian path $\ul\phi = (\phi_1, \dotsc, \phi_n)$ is called \emph{cocartesian} if any path of the form below, where $p, q \geq 0$, is weakly cocartesian. (If $p = 0$ then $\phi_1$ is the first cell in the path below; similarly if $q = 0$ then $\phi_n$ is the last cell.)
		\begin{displaymath}
			\begin{tikzpicture}
				\matrix(m)[math35, xshift=-2.5em, column sep={4em,between origins}]
					{	X'_0 & X'_1 & X'_{p''} & X'_{p'} & X_{10} & X_{11} & X_{1m_1'} & X_{1m_1} \\
						X'_0 & X'_1 & X'_{p''} & X'_{p'} & A_0 & & & A_1 \\ };
				\path[map]	(m-1-1) edge[barred] node[above] {$H'_1$} (m-1-2)
										(m-1-3) edge[barred] node[above] {$H'_{p'}$} (m-1-4)
										(m-1-4) edge[barred] node[above] {$H'_p(\id, f_0)$} (m-1-5)
										(m-1-5) edge[barred] node[above] {$H_{11}$} (m-1-6)
														edge node[right] {$f_0$} (m-2-5)
										(m-1-7) edge[barred] node[above, xshift=-1pt] {$H_{1m_1}$} (m-1-8)
										(m-1-8) edge node[right] {$f_1$} (m-2-8)
										(m-2-1) edge[barred] node[below] {$H'_1$} (m-2-2)
										(m-2-3) edge[barred] node[below] {$H'_{p'}$} (m-2-4)
										(m-2-4) edge[barred] node[below] {$H'_p$} (m-2-5)
										(m-2-5) edge[barred] node[below] {$\ul J_1$} (m-2-8);
				\path				(m-1-1) edge[eq] (m-2-1)
										(m-1-2) edge[eq] (m-2-2)
										(m-1-3) edge[eq] (m-2-3)
										(m-1-4) edge[eq] (m-2-4)
										(m-1-6) edge[transform canvas={xshift=1.875em}, cell] node[right] {$\phi_1$} (m-2-6);
				\draw				($(m-1-6)!0.5!(m-1-7)$) node[xshift=-1pt] {$\dotsb$}
										($(m-1-2)!0.5!(m-2-3)$) node {$\dotsb$}
										($(m-1-4)!0.5!(m-2-5)$) node[font=\scriptsize] {$\cart$};
				
				\matrix(m)[math35, yshift=-6.5em, xshift=2.5em, column sep={4em,between origins}]
					{ X_{n0} & X_{n1} & X_{nm_n'} & X_{nm_n} & X''_1 & X''_2 & X''_{q'} & X''_q \\
						A_{n'} & & & A_n & X''_1 & X''_2 & X''_{q'} & X''_q \\ };
				\path[map]	(m-1-1) edge[barred] node[above] {$H_{n1}$} (m-1-2)
														edge node[left, inner sep=1pt] {$f_{n'}$} (m-2-1)
										(m-1-3)	edge[barred] node[above, xshift=-2pt] {$H_{nm_n}$} (m-1-4)
										(m-1-4) edge[barred] node[above, xshift=-6pt] {$H''_1(f_n, \id)$} (m-1-5)
														edge node[left, inner sep=1.5pt] {$f_n$} (m-2-4)
										(m-1-5) edge[barred] node[above] {$H''_2$} (m-1-6)
										(m-1-7) edge[barred] node[above] {$H''_q$} (m-1-8)
										(m-2-1) edge[barred] node[below] {$\ul J_n$} (m-2-4)
										(m-2-4) edge[barred] node[below] {$H''_1$} (m-2-5)
										(m-2-5) edge[barred] node[below] {$H''_2$} (m-2-6)
										(m-2-7) edge[barred] node[below] {$H''_q$} (m-2-8);
				\path				(m-1-5) edge[eq] (m-2-5)
										(m-1-6) edge[eq] (m-2-6)
										(m-1-7) edge[eq] (m-2-7)
										(m-1-8) edge[eq] (m-2-8)
										(m-1-2) edge[transform canvas={xshift=1.875em}, cell] node[right] {$\phi_n$} (m-2-2);
				\draw				($(m-1-1)!0.5!(m-2-1)$) node[xshift=-2.5em] {$\dotsb$}
										($(m-1-2)!0.5!(m-1-3)$) node[xshift=-1pt] {$\dotsb$}
										($(m-1-6)!0.5!(m-2-7)$) node {$\dotsb$}
										($(m-1-4)!0.5!(m-2-5)$) node[font=\scriptsize] {$\cart$};
			\end{tikzpicture}
		\end{displaymath}
	\end{definition}
	Notice that cocartesianness of the path $\ul\phi$ depends on the existence of restrictions along $f_0$ and $f_n$. If no restrictions along $f_0$ and $f_n$ exist then $\ul\phi$ is cocartesian precisely if it is weakly cocartesian.

	Given a cocartesian horizontal cell of the form
	\begin{equation} \label{horizontal cocartesian cell}
		\begin{tikzpicture}[textbaseline]
			\matrix(m)[math35]{X_0 & X_1 & X_{n'} & X_n \\ X_0 & & & X_n \\};
			\path[map]	(m-1-1) edge[barred] node[above] {$H_1$} (m-1-2)
									(m-1-3) edge[barred] node[above] {$H_n$} (m-1-4)
									(m-2-1) edge[barred] node[below] {$J$} (m-2-4);
			\path				(m-1-1) edge[eq] (m-2-1)
									(m-1-4) edge[eq] (m-2-4);
			\path[transform canvas={xshift=1.75em}]	(m-1-2) edge[cell] node[right] {$\phi$} (m-2-2);
			\draw				($(m-1-2)!0.5!(m-1-3)$) node {$\dotsb$};
		\end{tikzpicture}
	\end{equation}
	we call, in the case that $n \geq 1$, $J$ the \emph{(horizontal) composite} of $(H_1, \dotsc, H_n)$ and write $(H_1 \hc \dotsb \hc H_n) \dfn J$. If $n = 0$ then, by \lemref{unit identities}, $\cell\phi{X_0}J$ corresponds to a horizontal cartesian cell $\cell\psi J{X_0}$ that defines $J$ as the (horizontal) unit $I_{X_0} = J$ of $X_0$, in the sense of \secref{restriction section}. Conversely, in \lemref{cocartesian unit identities} below we will see that any horizontal cartesian cell $J \Rightarrow X_0$ corresponds to a horizontal cocartesian cell $X_0 \Rightarrow J$.
	
	By their universal property any two cocartesian horizontal cells defining the same composite or unit factor through each other as invertible horizontal cells. The same property also ensures that composites of composites and units are associative and unital up to isomorphisms, as we will see after \lemref{pasting lemma for cocartesian paths} below. Like weakly cocartesian cells, in diagrams we denote single cocartesian cells simply by ``cocart''.
	
	Recall from the discussion following \defref{cartesian cells} that, when restricting its universal property to unary cells, the notion of weakly cocartesian cell in augmented virtual double categories coincides with the corresponding notion for virtual double categories. From this it follows that the notions of horizontal composite and horizontal unit likewise restrict to the corresponding notions for virtual double categories considered in Section~2 of \cite{Dawson-Pare-Pronk06} or Section 5 of \cite{Cruttwell-Shulman10}.

	Notice that the concatenation $\ul \phi \conc \ul \psi$ of two cocartesian paths $\ul \phi = (\phi_1, \dotsc, \phi_n)$ and $\ul \psi = (\psi_1, \dotsc, \psi_m)$ is again cocartesian whenever the common vertical target of $\phi_n$ and vertical source of $\psi_1$ is an identity morphism. That this is not true in general is shown in \exref{concatenation does not preserve cocartesian paths} below.
	
	\begin{example} \label{compositions of enriched profunctors}
		In \exref{weakly cocartesian cells in V-Prof} we characterised weakly cocartesian cells $\phi$ in $\enProf\V$ (\exref{enriched profunctors}) as in \eqref{horizontal cocartesian cell}, but with $X_0 = I = X_n$ the unit $\V$"/category, as those defining $J$ as the coend $\int^{u_1 \in X_1} \!\dotsb \int^{u_{n'} \in X_{n'}} H_1(*, u_1) \tens \dotsb \tens H_n(u_{n'}, *)$. Such a weakly cocartesian cell $\phi$ is cocartesian if and only if its coend is preserved by all functors $x \tens \dash$ and $\dash \tens x$, for all $x \in \V$.
	\end{example}
	
	\begin{example} \label{Span(E) has composites}
		In $\Span\E$ (\exref{internal profunctors}) all weakly cocartesian cells, as characterised in \exref{weakly cocartesian cells in Span(E)}, are cocartesian. Thus $\Span\E$ has all horizontal composites, besides having all horizontal units (see \exref{restrictions of spans}).
	\end{example}
	
	\begin{example}
		In \exref{weakly cocartesian cells in Rel(E)} we saw that for a cell $\phi$ as in \eqref{horizontal cocartesian cell} to be weakly cocartesian in $\Rel(\E)$ it suffices that its underlying morphism \mbox{$\map\phi{H_1 \times_{X_1} \dotsb \times_{X_{n'}} H_n}J$} is a strong epimorphism. In that case $\phi$ is cocartesian as soon as all pullbacks of $\phi$ are again strong epimorphisms. In particular $\Rel(\E)$ has all horizontal composites and units whenever $\E$ is \emph{regular}, in the sense of Section~3 of \cite{Carboni-Kasangian-Street84}.
	\end{example}
	
	\begin{example} \label{composites of internal profunctors}
		If $\E$ has reflexive coequalisers preserved by pullback then the augmented virtual double category $\inProf\E$ of profunctors internal to $\E$ (\exref{internal profunctors}) has all horizontal composites. The composite of internal profunctors is an ``internal coend''. That $\inProf\E$ has all horizontal units follows from \exref{restrictions of bimodules}.
		
		Next let $\K$ be a finitely complete $2$"/category that has reflexive coequalisers preserved by pullback. Since the embedding $\spFib\K \hookrightarrow \inProf{\K_0}$ (\exref{internal split fibrations}) is surjective on horizontal morphisms as well as locally full and faithful, it follows from \lemref{locally full and faithful functors reflect cocartesian paths} below that $\spFib\K$ too has all units and composites.
	\end{example}
	
	The following lemma shows that in conditions (b) and (e) of \lemref{unit identities} `weakly cocartesian' can be replaced by `cocartesian'.
	\begin{lemma} \label{cocartesian unit identities}
		Let $(\psi, \phi)\mspace{-3mu}$ be a pair of cells that satisfies identities \textup{(A)} and \textup{(J)} of \lemref{unit identities}. Then, both $\psi$ and $\phi$ are cocartesian.
	\end{lemma}
	\begin{proof}[(sketch)]
		Identities (A) and (J) imply that the unique factorisation of a cell $\chi$ through a path of the form $(\id_{H_1'}, \dotsc, \id_{H_p'}, \phi, \id_{H_1''}, \dotsc, \id_{H_q''})$, as in \defref{cocartesian paths}, is given by $\chi' = \chi \of (\id_{H_1'}, \dotsc, \id_{H_p'}, \psi, \id_{H_1''}, \dotsc, \id_{H_q''})$. Likewise factorisations through $(\id_{H_1'}, \dotsc, \id_{H_p'}, \psi, \id_{H_1''}, \dotsc, \id_{H_q''})$ are given by composing with $\phi$.
	\end{proof}
	
	(Weakly) cocartesian paths, like cartesian cells, satisfy a pasting lemma as follows. In proving the assertions (a) and (b) use the pasting lemma for cartesian cells (\lemref{pasting lemma for cartesian cells}).
	\begin{lemma}[Pasting lemma] \label{pasting lemma for cocartesian paths}
		In the configuration of cells below denote by $\ul\phi_j$ the path $\ul\phi_j \dfn (\phi_{j1}, \dotsc, \phi_{jn_j})$, for each $1 \leq j \leq n$, and assume that the path $(\phi_{11}, \dotsc, \phi_{nm_n})$ is weakly cocartesian. Then the path $\ul\psi \dfn (\psi_1, \dotsc, \psi_n)$  is weakly cocartesian if and only if the path of composites $\bigpars{\psi_1 \of \ul \phi_1, \dotsc, \psi_n \of \ul\phi_n}$ is so.
		\begin{displaymath}
			\begin{tikzpicture}[scheme, x=0.40cm]
				\draw	(1,2) -- (0,2) -- (0,0) -- (18,0) -- (18,2) -- (17,2)
							(2,2) -- (3,2) -- (3,1) -- (0,1)
							(7,2) -- (6,2) -- (6,1) -- (12,1) -- (12,2) -- (11,2)
							(8,2) -- (10,2)
							(9,2) -- (9,0)
							(16,2) -- (15,2) -- (15,1) -- (18,1)
							(28,2) -- (27,2) -- (27,0) -- (36,0) -- (36,2) -- (35,2)
							(29,2) -- (30,2) -- (30,1) -- (27,1)
							(34,2) -- (33,2) -- (33,1) -- (36,1);
				\draw[shift={(0.3pt, -0.33pt)}] (1.5,2) node {\scalebox{.6}{$\dotsb$}}
							(7.5,2) node {\scalebox{.6}{$\dotsb$}}
							(10.5,2) node {\scalebox{.6}{$\dotsb$}}
							(16.5,2) node {\scalebox{.6}{$\dotsb$}}
							(28.5,2) node {\scalebox{.6}{$\dotsb$}}
							(34.5,2) node {\scalebox{.6}{$\dotsb$}};
				\draw[shift={(0.165cm, 0.3cm)}] (1,1) node {$\phi_{11}$}
							(7,1) node {$\phi_{1m_1}$}
							(10,1) node {$\phi_{21}$}
							(16,1) node {$\phi_{2m_2}$}
							(28,1) node {$\phi_{n1}$}
							(34,1) node {$\phi_{nm_n}$}
							(4,0) node {$\psi_1$}
							(13,0) node {$\psi_2$}
							(31,0) node {$\psi_n$}
							(4,1) node {$\dotsb$}
							(13,1) node {$\dotsb$}
							(31,1) node {$\dotsb$};
				\draw[font=]	(22.5,1) node {$\dotsb$};
			\end{tikzpicture}
		\end{displaymath}
		
		Next denote the vertical sources of $\phi_{11}$ and $\psi_1$ by $\map{f_{10}}{X_{110}}{A_{10}}$ and $\map{h_0}{A_{10}}{C_0}$, and the vertical targets of $\phi_{nm_n}$ and $\psi_n$ by $\map{f_{nm_n}}{X_{nm_nk_{m_n}}}{A_{nm_n}}$ and $\map{h_n}{A_{nm_n}}{C_n}$. If the path $(\phi_{11}, \dotsc, \phi_{nm_n})$ is cocartesian then the following hold:
		\begin{enumerate}[label=\textup{(\alph*)}]
			\item	if $\ul\psi$ is cocartesian then so is $\bigpars{\psi_1 \of \ul \phi_1, \dotsc, \psi_n \of \ul\phi_n}$ provided that for any horizontal morphisms $\hmap{K'}{C'}{C_0}$ and $\hmap{K''}{C_n}{C''}$ the following holds: if the restrictions $K'(\id, h_0 \of f_{10})$ and $K''(h_n \of f_{nm_n}, \id)$ exist then so do $K'(\id, h_0)$ and $K''(h_n, \id)$;

			\item if $\bigpars{\psi_1 \of \ul \phi_1, \dotsc, \psi_n \of \ul\phi_n}$ is cocartesian then so is $\ul\psi$ provided that for any horizontal morphisms $\hmap{K'}{C'}{C_0}$ and $\hmap{K''}{C_n}{C''}$ the following holds: if the restrictions $K'(\id, h_0)$ and $K''(h_n, \id)$ exist then so do $K'(\id, h_0 \of f_{10})$ and $K''(h_n \of f_{nm_n}, \id)$.
		\end{enumerate}
	\end{lemma}
	
	Applying the pasting lemma to compositions $\psi \of (\phi_1, \dotsc, \phi_n)$ of horizontal cells shows that the collection of horizontal composites and units in an augmented virtual double category is coherent as follows. Let $(\ul J_1, \dotsc, \ul J_n)$ be a path of paths $\ul J_i = (J_{i1}, \dotsc, J_{im_i})$ of horizontal morphisms. If all composites $(J_{11} \hc \dotsb \hc J_{1m_1})$, \dots, $(J_{n1} \hc \dotsb \hc J_{nm_n})$ exist then the composite $(J_{11} \hc \dotsb \hc J_{nm_n})$ of the concatenation $J_{11} \conc \dotsb \conc J_{nm_n}$ exists if and only if
	\begin{displaymath}
		\bigpars{(J_{11} \hc \dotsb \hc J_{1m_1}) \hc \dotsb \hc (J_{n1} \hc \dotsb \hc J_{nm_n})}
	\end{displaymath}
	does, in which case they are canonically isomorphic. Notice that this also includes isomorphisms of the form $(I_A \hc J) \iso J \iso (J \hc I_B)$, for any $\hmap JAB$, and similar.	
	
	Remember that any functor between unital augmented virtual double categories preserves horizontal units by \cororef{functors preserve companions and conjoints}. We follow \cite{Cruttwell-Shulman10} in calling a functor \mbox{$\map F\K\L$} \emph{strong} if it preserves horizontal composites too, that is its image of any horizontal cocartesian cell in $\K$ is cocartesian in $\L$.
	
	To complete the picture we now briefly describe the classical notion of `pseudo double category' as introduced by Grandis and Par\'e in the Appendix to \cite{Grandis-Pare99}; see also Section~2 of \cite{Shulman08}. In our terms, a \emph{pseudo double category} is a virtual double category that contains $(1,1)$-ary cells only and which is equipped with a horizontal composition
	\begin{align*}
		A \xbrar J B \xbrar H E \mspace{-4.25mu}\quad\qquad&\mapsto\mspace{-4.25mu}\quad\qquad A \xbrar{J \hc H} E; \\
		\begin{tikzpicture}[textbaseline, ampersand replacement=\&]
			\matrix(m)[math35]{A \& B \& E \\ C \& D \& F \\};
			\path[map]	(m-1-1) edge[barred] node[above] {$J$} (m-1-2)
													edge node[left] {$f$} (m-2-1)
									(m-1-2) edge[barred] node[above] {$H$} (m-1-3)
													edge node[right] {$g$} (m-2-2)
									(m-1-3) edge node[right] {$h$} (m-2-3)
									(m-2-1) edge[barred] node[below] {$K$} (m-2-2)
									(m-2-2) edge[barred] node[below] {$L$} (m-2-3);
			\path[transform canvas={xshift=1.75em}]	(m-1-1) edge[cell] node[right] {$\phi$} (m-2-1)
									(m-1-2) edge[cell] node[right] {$\psi$} (m-2-2);
		\end{tikzpicture} \qquad&\mapsto\qquad \begin{tikzpicture}[textbaseline, ampersand replacement=\&]
			\matrix(m)[math35, column sep={5em,between origins}]{A \& E \\ C \& F, \\};
			\path[map]	(m-1-1) edge[barred] node[above] {$J \hc H$} (m-1-2)
													edge node[left] {$f$} (m-2-1)
									(m-1-2) edge node[right] {$h$} (m-2-2)
									(m-2-1) edge[barred] node[below] {$K \hc L$} (m-2-2);
			\path[transform canvas={xshift=1.5em}]	(m-1-1) edge[cell] node[right] {$\phi \hc \psi$} (m-2-1);
		\end{tikzpicture}
	\end{align*}
	as well as horizontal units $\hmap{I_A}AA$; $\cell{I_f}{I_A}{I_C}$. These come with horizontal coherence cells of the forms $(J \hc H) \hc M \iso J \hc (H \hc M)$, $I_A \hc J \iso J$ and $J \hc I_B \iso J$ which satisfy the usual coherence axioms, analogous to those for a monoidal category or bicategory; see e.g.\ Section~VII.1 of \cite{MacLane98}. A pseudo double category with identity cells as coherence cells is called a \emph{strict} double category.
	
	Any pseudo double category gives rise to a virtual double category with the same objects and morphisms, whose cells $(J_1, \dotsc, J_n) \Rar K$ correspond to cells \mbox{$J_1 \hc \dotsb \hc J_n \Rar K$} in the double category. The following result, which is Proposition~2.8 of \cite{Dawson-Pare-Pronk06} and Theorem~5.2 of \cite{Cruttwell-Shulman10}, characterises the virtual double categories obtained in this way.
	\begin{proposition}[MacG. Dawson, Par\'e and Pronk] \label{pseudo double categories}
		A virtual double category is induced by a pseudo double category if and only if it has all horizontal composites and units.
	\end{proposition}
	
	In \thmref{unital virtual double categories} below we will see that in the presence of horizontal units the notions of augmented virtual double category and virtual double category coincide. In view of this and the proposition above, by a \emph{pseudo double category} we shall mean either an (augmented) virtual double category that has all horizontal composites and units or, equivalently, a pseudo double category in the classical sense. Following \cite{Cruttwell-Shulman10}, by an \emph{equipment} we shall mean a pseudo double category that has all restrictions. Table~\ref{notions} includes most of the double categories and equipments considered in this paper.
	
	\begin{example} \label{strict double category of quintets}
		The augmented virtual double category $Q(\catvar C)$ of quintets in a $2$"/category $\catvar C$ (\exref{quintets}) clearly is a strict double category: the composite $(j \hc k)$ of two horizontal morphisms in $Q(\catvar C)$ is simply given by their composite $k \of j$ in $\catvar C$.
	\end{example}
	
	In closing this section we consider augmented virtual double categories $\K$ that have all weakly cocartesian paths of $(0,1)$"/ary cells as in the definition below. In such $\K$ general cells $(J_1, \dotsc, J_n) \Rightarrow \ul K$ can be identified with cells into $\ul K$ with empty horizontal source. This can be used to show that certain notions in $\K$ are equivalent to the corresponding notions in the vertical $2$"/category $V(\K)$. The proposition below asserts such equivalences for the notions of full and faithful morphism and absolute left lifting; for the case of pointwise Kan extension see Section~5.5 of \cite{Koudenburg14} and Section~4.6 of \cite{Koudenburg19b}.
	\begin{definition}
		An augmented virtual double category is said to \emph{have all weakly cocartesian paths of $(0,1)$"/ary cells} if, for every path $\hmap{\ul J = (J_1, \dotsc, J_n)}{A_0}{A_n}$ of horizontal morphisms, there exists a weakly cocartesian path $\ul\phi = (\phi_1, \dotsc, \phi_n)$ of $(0,1)$"/ary cells
		\begin{displaymath}
			\begin{tikzpicture}
  			\matrix(m)[math35, column sep={1.75em,between origins}]{& X & \\ A_{i'} & & A_i. \\};
				\path[map]	(m-1-2) edge[transform canvas={xshift=-2pt}] node[left] {$f_{i'}$} (m-2-1)
														edge[transform canvas={xshift=2pt}] node[right] {$f_i$} (m-2-3)
										(m-2-1) edge[barred] node[below] {$J_i$} (m-2-3);
				\path				(m-1-2) edge[cell, transform canvas={yshift=-0.25em}] node[right, inner sep=2.5pt] {$\phi_i$} (m-2-2);
			\end{tikzpicture}
		\end{displaymath}
	\end{definition}
	
	\begin{example}
		The \emph{graph} $\tab{\ul J}$ of a path of $\Set'$"/profunctors $\hmap{\ul J}{A_0}{A_n}$ in $\enProf{\Set'}$ (\exref{enriched profunctors}) is the category whose objects are tuples \mbox{$\ul u = (x_0, u_1, x_1, \dotsc, u_n, x_n)$} of, alternatingly, objects $x_i \in A_i$ and elements $u_i \in J_i(x_{i'}, x_i)$, while its morphisms $\ul u \to \ul u'$ are tuples $(s_1, \dotsc, s_n)$ of morphisms $\map{s_i}{x_i}{x'_i}$ in $A_i$ such that $\lambda(s_{i'}, u'_i) = \rho(u_i, s_i)$.
		\begin{displaymath}
			\begin{tikzpicture}
  			\matrix(m)[math35, column sep={1.75em,between origins}]{& \tab{\ul J} & \\ A_{i'} & & A_i \\};
				\path[map]	(m-1-2) edge[transform canvas={xshift=-2pt}] node[left] {$p_{i'}$} (m-2-1)
														edge[transform canvas={xshift=2pt}] node[right] {$p_i$} (m-2-3)
										(m-2-1) edge[barred] node[below] {$J_i$} (m-2-3);
				\path				(m-1-2) edge[cell, transform canvas={yshift=-0.25em}] node[right, inner sep=2.5pt] {$\pi_i$} (m-2-2);
			\end{tikzpicture}
		\end{displaymath}
		Writing $\map{p_i}{\tab{\ul J}}{A_i}$ for the projections, consider the $(0,1)$"/ary cells $\pi_i$ above, which map $\ul u = (x_0, u_1, x_1, \dotsc, u_n, x_n)$ to $u_i \in J_i(x_{i'}, x_i)$. It is straightforward to check that the path $(\pi_1, \dotsc, \pi_n)$ is cocartesian. Cocartesian paths of $(0,1)$"/ary cells in $\enProf{\Set}$ and $\enProf{(\Set, \Set')}$ (\exref{(Set, Set')-Prof}) can be obtained analogously.
		
		Restricting to the case $n=1$ the single cell $\cell{\pi_1}{\tab{J_1}}{J_1}$ above is universal with respect to all $(0,1)$"/ary cells $\cell\phi X{J_1}$, exhibiting $\tab{J_1}$ as the `tabulation' of $J_1$; see e.g.\ Section~4.5 of \cite{Koudenburg19b}. There it is also shown that, in general, `cocartesian tabulations' can be combined to obtain cocartesian paths $(\pi_1, \dotsc, \pi_n)$, similar to the construction above.
	\end{example}
	
	\begin{proposition} \label{converses with weakly cocartesian paths of (0,1)-ary cells}
		In an augmented virtual double category $\K$ that has all weakly cocartesian paths of $(0,1)$"/ary cells the converses of \lemref{full and faithful in V(K)} and \lemref{restrictions and absolute left liftings} hold.
	\end{proposition}
	\begin{proof}[(sketch)]
		For any path $\hmap{\ul J}{A_0}{A_n}$ in $\K$ consider a weakly cocartesian path $\ul \phi = (\phi_1, \dotsc, \phi_n)$ of $(0, 1)$"/ary cells, as in the definition above. Composing with $\ul\phi$ gives a bijection between nullary cells $\psi$ with horizontal source $\ul J$ and vertical cells $\chi$ with source $X$ and, in the case of both lemmas, the universal property for the cells $\psi$ under consideration (defining a notion in $\K$) is equivalent to that for the vertical cells $\chi$ (defining the corresponding notion in $V(\K)$) under this bijection.
	\end{proof}
	
	\section{Restrictions and extensions in terms of companions and conjoints} \label{restrictions and extensions in terms of companions and conjoints section}
	Here we make precise the fact that restrictions and extensions of a horizontal morphism can be obtained by composing it with companions and conjoints, as anticipated in the discussion following \defref{cartesian cells}.
	
	We start with restrictions. In the setting of unital virtual equipments the `only if'-part of the first assertion of the following lemma was proved as Theorem~7.16 of~\cite{Cruttwell-Shulman10}; notice that here we do not have to assume the existence of horizontal units. In \lemref{cocartesian cell defining restriction is pointwise} below we will see that the composite of $f_* \conc \ul K \conc g^*$ considered below is in fact `pointwise'.
	\begin{lemma} \label{restrictions and composites}
		In an augmented virtual double category $\K$ assume that the companion $\hmap{f_*}AC$ and the conjoint $\hmap{g^*}DB$ exist. For each path $\hmap{\ul K}CD$ of length $\leq 1$ the restriction $\ul K(f, g)$ exists if and only if the horizontal composite of the path $f_* \conc \ul K \conc g^*$ does, and in that case they are isomorphic.
		\begin{displaymath}
			\begin{tikzpicture}[textbaseline]
				\matrix(m)[math35, column sep={1.75em,between origins}]{A & & C & & D & & B \\ & C & & & & D & \\};
				\path[map]	(m-1-1) edge[barred] node[above] {$f_*$} (m-1-3)
														edge node[below left] {$f$} (m-2-2)
										(m-1-3) edge[barred] node[above] {$\ul K$} (m-1-5)
										(m-1-5) edge[barred] node[above] {$g^*$} (m-1-7)
										(m-1-7) edge node[below right] {$g$} (m-2-6)
										(m-2-2) edge[barred] node[below] {$\ul K$} (m-2-6);
				\path				(m-1-3) edge[eq] (m-2-2)
										(m-1-5) edge[eq] (m-2-6);
				\draw				([yshift=0.333em]$(m-1-2)!0.5!(m-2-2)$) node[font=\scriptsize] {$\cart$}
										([yshift=0.333em]$(m-1-6)!0.5!(m-2-6)$) node[font=\scriptsize] {$\cart$};
			\end{tikzpicture} \quad = \quad \begin{tikzpicture}[textbaseline]
				\matrix(m)[math35, column sep={5em,between origins}]{A & B \\ A & B \\ C & D \\};
				\path[map]	(m-1-1) edge[barred] node[above] {$f_* \conc \ul K \conc g^*$} (m-1-2)
										(m-2-1) edge[barred] node[below] {$J$} (m-2-2)
														edge node[left] {$f$} (m-3-1)
										(m-2-2) edge node[right] {$g$} (m-3-2)
										(m-3-1) edge[barred] node[below] {$\ul K$} (m-3-2);
				\path				(m-1-1) edge[eq] (m-2-1)
										(m-1-2) edge[eq] (m-2-2);
				\path[transform canvas={xshift=2.5em}]	(m-1-1) edge[cell] node[right] {$\phi$} (m-2-1)
										(m-2-1) edge[cell, transform canvas={yshift=-0.25em}] node[right] {$\psi$} (m-3-1);
			\end{tikzpicture}
		\end{displaymath}
		In detail, for a factorisation as above (where the empty cell is the vertical identity cell $\id_C$ if $\ul K$ is empty) the following are equivalent: $\psi$ is cartesian; $\phi$ is cocartesian; the identity below holds. Moreover in this case the path $(\cocart, \psi, \cocart)$, making up the top row of the left"/hand side below, is cocartesian.
		\begin{displaymath}
			\begin{tikzpicture}[textbaseline]
				\matrix(m)[math35, column sep={1.75em,between origins}]{& A & & & & B & \\ A & & C & & D & & B \\ & A & & & & B & \\};
				\path[map]	(m-1-2) edge[barred] node[above] {$J$} (m-1-6)
														edge[transform canvas={xshift=2pt}] node[right] {$f$} (m-2-3)
										(m-1-6) edge[transform canvas={xshift=-2pt}] node[left] {$g$} (m-2-5)
										(m-2-1) edge[barred] node[below] {$f_*$} (m-2-3)
										(m-2-3) edge[barred] node[below] {$\ul K$} (m-2-5)
										(m-2-5) edge[barred] node[below] {$g^*$} (m-2-7)
										(m-3-2) edge[barred] node[below] {$J$} (m-3-6);
				\path				(m-1-2) edge[eq, transform canvas={xshift=-2pt}] (m-2-1)
										(m-1-6) edge[eq, transform canvas={xshift=2pt}] (m-2-7)
										(m-2-1) edge[eq] (m-3-2)
										(m-2-7) edge[eq] (m-3-6);
				\path				(m-1-4) edge[cell] node[right] {$\psi$} (m-2-4)
										(m-2-4) edge[cell, transform canvas={yshift=-0.25em}] node[right] {$\phi$} (m-3-4);
				\draw[font=\scriptsize]	([xshift=1pt]$(m-1-2)!0.666!(m-2-2)$) node {$\cocart$}
										([xshift=-1pt]$(m-1-6)!0.666!(m-2-6)$) node {$\cocart$};
			\end{tikzpicture} \quad = \quad \begin{tikzpicture}[textbaseline]
  			\matrix(m)[math35]{A & C \\ A & C \\};
  			\path[map]	(m-1-1) edge[barred] node[above] {$J$} (m-1-2)
  									(m-2-1) edge[barred] node[below] {$J$} (m-2-2);
  			\path				(m-1-1) edge[eq] (m-2-1)
  									(m-1-2) edge[eq] (m-2-2);
  			\path[transform canvas={xshift=1.75em, xshift=-5.5pt}]	(m-1-1) edge[cell] node[right] {$\id_J$} (m-2-1);
  		\end{tikzpicture}
  	\end{displaymath}
  	
  	Analogous assertions hold for one-sided restrictions. In particular $K(f, \id)$ exists precisely if $f_* \hc K$ does, while $K(\id, g)$ exists if and only if $K \hc g^*$ does.
	\end{lemma}
	\begin{proof}
		Assuming that the top identity above holds, it follows from the companion and conjoint identities (see \lemref{companion identities lemma} and its horizontal dual) that vertically precomposing the composite on the left"/hand side of the bottom equation with $\phi$ again results in $\phi$, while postcomposing it with $\psi$ gives back $\psi$. Using the uniqueness of factorisations through (co)cartesian cells we conclude that either $\psi$ or $\phi$ being (co)cartesian implies the bottom identity.
		
		Conversely, assume that both identities above hold; we will prove that $\psi$ is cartesian and that $\phi$  and $(\cocart, \psi, \cocart)$ are cocartesian. For the first it suffices to show that the following assignment of cells is a bijection. Indeed the identities imply that its inverse can be given by $\chi \mapsto \phi \of (\cocart \of h, \chi, \cocart \of k)$, where the weakly cocartesian cells define $f_*$ and $g^*$ respectively.
		\begin{displaymath}
			\begin{tikzpicture}
				\matrix(m)[minimath, xshift=-9em]{X_0 & X_1 & X_{n'} & X_n \\ A & & & B \\};
				\path[map]	(m-1-1) edge[barred] node[above] {$H_1$} (m-1-2)
														edge node[left] {$h$} (m-2-1)
										(m-1-3) edge[barred] node[above] {$H_n$} (m-1-4)
										(m-1-4) edge node[right] {$k$} (m-2-4)
										(m-2-1) edge[barred] node[below] {$J$} (m-2-4);
				\draw				($(m-1-2)!0.5!(m-1-3)$) node {$\dotsc$};
				\path				($(m-1-1.south)!0.5!(m-1-4.south)$) edge[cell] node[right] {$\chi'$} ($(m-2-1.north)!0.5!(m-2-4.north)$);
				
				\matrix(m)[minimath, xshift=9em]{X_0 & X_1 & X_{n'} & X_n \\ C & & & D \\};
				\path[map]	(m-1-1) edge[barred] node[above] {$H_1$} (m-1-2)
														edge node[left] {$f \of h$} (m-2-1)
										(m-1-3) edge[barred] node[above] {$H_n$} (m-1-4)
										(m-1-4) edge node[right] {$g \of k$} (m-2-4)
										(m-2-1) edge[barred] node[below] {$\ul K$} (m-2-4);
				\draw				($(m-1-2)!0.5!(m-1-3)$) node {$\dotsc$};
				\path				($(m-1-1.south)!0.5!(m-1-4.south)$) edge[cell] node[right] {$\chi$} ($(m-2-1.north)!0.5!(m-2-4.north)$);
				
				\draw[font=\Large]	(-15.25em,0) node {$\lbrace$}
										(-2.75em,0) node {$\rbrace$}
										(1.75em,0) node {$\lbrace$}
										(16.25em,0) node {$\rbrace$};
				\path[map]	(-1.5em,0) edge node[above] {$\psi \of \dash$} (0.5em,0);
			\end{tikzpicture}
		\end{displaymath}
		Similarly that $\phi$ and $(\cocart, \psi, \cocart)$ are cocartesian follows from the fact that, for any paths $\hmap{\ul J'}{A'_0}{A'_p = A}$ and $\hmap{\ul J''}{B= B'_0}{B'_q}$, the assignments of cells
		\begin{displaymath}
			\begin{tikzpicture}
				\matrix(m)[math35, column sep={22.5em,between origins}]
					{ \lbrace\map{\xi'}{\ul J' \conc J \conc \ul J''}{\ul L}\rbrace &
						\lbrace\map\xi{\ul J' \conc f_* \conc \ul K \conc g^* \conc \ul J''}{\ul L}\rbrace \\ };
				\path[map, transform canvas={yshift=3pt}]	(m-1-1) edge node[above] {$\dash \of (\ul\id, \phi, \ul\id)$} (m-1-2);
				\path[map, transform canvas={yshift=-3pt}]	(m-1-2) edge node[below] {$\dash \of (\ul\id, \cocart, \psi, \cocart, \ul\id)$} (m-1-1);
			\end{tikzpicture}
		\end{displaymath}
		are inverses whenever both identities hold.
	\end{proof}
	
	The remainder of this section consists of corollaries of the lemma above. The first of these shows that functors of augmented virtual double categories behave well with respect to restrictions along morphisms that admit companions/conjoints. This is a variation on the corresponding result for functors between double categories; see Proposition 6.8 of \cite{Shulman08}.
	\begin{corollary} \label{functors preserving cartesian cells}
		Let $\map F\K\L$ be a functor between augmented virtual double categories. Consider morphisms $\map fAC$ and $\map gBD$ in $\K$ and let $\hmap{\ul K}CD$ be a path of length $\leq 1$. If the companion $\hmap{f_*}AC$ and the conjoint $\hmap{g^*}DB$ exist then $F$ preserves both the cartesian cell defining the restriction $\ul K(f, g)$ as well as the cocartesian cell defining the horizontal composite of the path $f_* \conc \ul K \conc g^*$.
		
		Under the same conditions the cartesian cells defining the restrictions of the form $K(f, \id)$ and $K(\id, g)$, as well as the cocartesian ones defining the horizontal composites of the form $(f_* \hc K)$ and $(K \hc g^*)$, are preserved by $F$.
	\end{corollary}
	\begin{proof}
		This follows from the fact that $F$ preserves the identities of the previous lemma as well as the (weakly co)cartesian cells that define $f_*$ and $g^*$; the latter by \cororef{functors preserve companions and conjoints}.
	\end{proof}
	
	\begin{corollary} \label{cocartesian cells for companions and conjoints}
		Weakly cocartesian cells that define companions or conjoints, as in the discussion preceding \lemref{companion identities lemma}, are cocartesian.
	\end{corollary}
	\begin{proof}
		Let $\map fAC$ be a vertical morphism. We will prove that any weakly cocartesian cell defining the companion $f_*$, as in the composite below, is cocartesian; the proof for the conjoint $f^*$ is horizontally dual. By \defref{cocartesian paths} it suffices to prove that for any $\hmap KCD$ the path
		\begin{displaymath}
			\begin{tikzpicture}
				\matrix(m)[math35, column sep={1.75em,between origins}]
					{	& A & & D & \\
						A & & C & & D, \\ };
				\path[map]	(m-1-2) edge[barred] node[above] {$K(f, \id)$} (m-1-4)
														edge[transform canvas={xshift=2pt}] node[above right, inner sep=1pt] {$f$} (m-2-3)
										(m-2-1) edge[barred] node[below] {$f_*$} (m-2-3)
										(m-2-3) edge[barred] node[below] {$K$} (m-2-5);
				\path				(m-1-2) edge[eq, transform canvas={xshift=-1pt}] (m-2-1)
										(m-1-4) edge[eq, transform canvas={xshift=1pt}] (m-2-5)
										(m-1-3) edge[cell, transform canvas={xshift=0.875em}] node[right] {$\psi$} (m-2-3);
				\draw[font=\scriptsize]	($(m-1-2)!0.666!(m-2-2)$) node {$\cocart$};
			\end{tikzpicture}
		\end{displaymath}
		where $\psi$ is the cartesian cell defining $K(f, \id)$, is cocartesian. But this follows directly from the second assertion of \lemref{restrictions and composites} for $K(f, \id)$.
	\end{proof}
	
	\begin{example} \label{concatenation does not preserve cocartesian paths}
		To show that a path of cocartesian cells need not be cocartesian itself in general consider a morphism $\map fAC$ such that $f_*$, $f^*$ and $C(f, f)$ exist. We claim that the path
		\begin{displaymath}
			\begin{tikzpicture}
				\matrix(m)[math35]{& A & \\ A & C & A \\};
				\path[map]	(m-1-2) edge node[right, inner sep=0.25pt] {$f$} (m-2-2)
										(m-2-1) edge[barred] node[below] {$f_*$} (m-2-2)
										(m-2-2) edge[barred] node[below] {$f^*$} (m-2-3);
				\path				(m-1-2) edge[eq, transform canvas={shift={(-2.5pt,1.5pt)}}] (m-2-1)
														edge[eq, transform canvas={shift={(2.5pt,1.5pt)}}] (m-2-3);
				\draw[font=\scriptsize]	($(m-1-1)!0.57!(m-2-2)$) node[rotate=42.5] {$\cocart$}
										($(m-1-3)!0.57!(m-2-2)$) node[rotate=-42.5] {$\cocart$};
			\end{tikzpicture}
		\end{displaymath}
		is weakly cocartesian only if $f$ is full and faithful (\defref{full and faithful morphism}). To see this let the cocartesian cell $\cell\phi{(f_*, f^*)}{C(f, f)}$ and the cartesian cell $\cell\psi{C(f, f)}C$ be as in \lemref{restrictions and composites}. If the path above is weakly cocartesian then so is the composite $\phi\of (\cocart, \cocart)$ by the pasting lemma for cocartesian paths (\lemref{pasting lemma for cocartesian paths}); that is cartesian too then follows from \lemref{unit identities}. Using the pasting lemma for cartesian cells (\lemref{pasting lemma for cartesian cells}) it follows that $\psi \of \phi \of (\cocart, \cocart)$ is cartesian which, by the first identity of \lemref{restrictions and composites} and the vertical companion and conjoint identities (\lemref{companion identities lemma}), equals $\id_f$. We conclude that $f$ is full and faithful.
	\end{example}
	
	Together with the pasting lemma for cocartesian paths (\lemref{pasting lemma for cocartesian paths}) \cororef{cocartesian cells for companions and conjoints} allows us to describe extensions along vertical morphisms in terms of compositions with their companions and conjoints as follows; this is a variation on the corresponding result for unital virtual equipments Theorem~7.20 of \cite{Cruttwell-Shulman10}.
	\begin{corollary} \label{extensions and composites}
		Any composite of the form below is cocartesian, so that it defines $J$ as the extension of $(H_1, \dotsc, H_n)$ along $h$. Cocartesian cells that define extensions on the right or that define two"/sided extensions can be constructed analogously.
		\begin{displaymath}
			\begin{tikzpicture}
				\matrix(m)[math35, column sep={1.75em,between origins}]{& X_0 & & X_1 & & X_{n'} & & X_n & \\ A & & X_0 & & X_1 & & X_{n'} & & X_n \\ & A & & & & & & X_n & \\};
				\path[map]	(m-1-2) edge[barred] node[above] {$H_1$} (m-1-4)
														edge[transform canvas={xshift=-2pt}] node[above left] {$h$} (m-2-1)
										(m-1-6) edge[barred] node[above] {$H_n$} (m-1-8)
										(m-2-1) edge[barred] node[below] {$h^*$} (m-2-3)
										(m-2-3) edge[barred] node[above] {$H_1$} (m-2-5)
										(m-2-7) edge[barred] node[above] {$H_n$} (m-2-9)
										(m-3-2) edge[barred] node[below] {$J$} (m-3-8);
				\path				(m-1-2) edge[eq, transform canvas={xshift=2pt}] (m-2-3)
										(m-1-4) edge[eq, transform canvas={xshift=1pt}] (m-2-5)
										(m-1-6) edge[eq] (m-2-7)
										(m-1-8) edge[eq] (m-2-9)
										(m-2-1) edge[eq] (m-3-2)
										(m-2-9) edge[eq] (m-3-8);
				\draw				($(m-1-6)!0.5!(m-2-5)$) node {$\dotsb$};
				\draw[font=\scriptsize]	([xshift=-1pt]$(m-1-2)!0.666!(m-2-2)$) node {$\cocart$}
										($(m-2-1)!0.5!(m-3-9)$) node {$\cocart$};
			\end{tikzpicture}
		\end{displaymath}
	\end{corollary}

	\section{Pointwise horizontal composites} \label{pointwise horizontal composites section}
	Consider a path $\hmap{(H_1, \dotsc, H_n)}{X_0}{X_n}$ in the augmented virtual double category $\enProf\V$ of $\V$"/profunctors (\exref{enriched profunctors}). In \exref{compositions of enriched profunctors} we have seen that, in the special case where $X_0 = I = X_n$ is the unit $\V$"/category, the horizontal composite $(H_1 \hc \dotsb \hc H_n)$ is given by the coend $\int^{u_1 \in X_1} \!\dotsb \int^{u_{n'} \in X_{n'}} H_1(*, u_1) \tens \dotsb \tens H_n(u_{n'}, *)$, provided that it is preserved by the monoidal product $\tens$ of $\V$ on both sides. Recall that in the general case, where $X_0$ and $X_n$ are any $\V$"/categories, the composite $(H_1 \hc \dotsb \hc H_n)$ can be built up ``pointwise'' from such coends, by taking
	\begin{displaymath}
		(H_1 \hc \dotsb \hc H_n)(x, y) = \int^{u_1 \in X_1} \!\dotsb \int^{u_{n'} \in X_{n'}} H_1(x, u_1) \tens \dotsb \tens H_n(u_{n'}, y)
	\end{displaymath}
	for each pair $x \in X_0$ and $y \in X_n$. The definition of `pointwise horizontal composite' below formalises the pointwise character of this composite inside an augmented virtual double category; informally it captures that ``any restriction of a pointwise horizontal composite $(H_1 \hc \dotsb \hc H_n)$ is again a horizontal composite''. Pointwise horizontal composites are important in the study of ``pointwise Kan extensions'' in augmented virtual double categories; see Section~4 of \cite{Koudenburg19b}. While the definition below is stated in terms of a path $(\phi_1, \dotsc, \phi_n)$ of unary cells we will mostly apply it to single horizontal cocartesian cells $\cell{\phi_1}{(H_1, \dotsc, H_n)}{(H_1 \hc \dotsb \hc H_n)}$.
	
	\begin{definition} \label{pointwise cocartesian path}
		Consider a path $\ul \phi = (\phi_1, \dotsc, \phi_n)$ of unary cells whose last cell $\phi_n$ has non"/empty horizontal source and trivial vertical target, as in the composite on the left-hand side below. Let $\map fY{X_{nm_n}}$ be any morphism such that both restrictions $H_{nm_n}(\id, f)$ and $J_n(\id, f)$ exist.
		
		We call $\ul \phi$ \emph{right pointwise cocartesian with respect to $f$} if the path $(\phi_1, \dotsc, \phi_n')$ is cocartesian, where $\phi_n'$ is the unique factorisation as below. We call $\ul \phi$ \emph{right pointwise cocartesian} if it is right pointwise cocartesian with respect to all such morphisms $f$.
		\begin{equation} \label{pointwise cocartesian factorisation}
			\begin{tikzpicture}[textbaseline]
				\matrix(m)[math35, column sep={2.1em,between origins}]
					{ X_{n0} & & X_{n1} & & X_{n(m_n)'} &[0.25em] & Y \\
						X_{n0} & & X_{n1} & & X_{n(m_n)'} & & X_{nm_n} \\
						& A_{n'} & & & & X_{nm_n} & \\ };
				\path[map]	(m-1-1) edge[barred] node[above] {$H_{n1}$} (m-1-3)
										(m-1-5) edge[barred] node[above, xshift=-12pt] {$H_{nm_n}(\id, f)$} (m-1-7)
										(m-1-7) edge node[right] {$f$} (m-2-7)
										(m-2-1) edge[barred] node[below] {$H_{n1}$} (m-2-3)
														edge node[left] {$f_{n'}$} (m-3-2)
										(m-2-5) edge[barred] node[below] {$H_{nm_n}$} (m-2-7)
										(m-3-2) edge[barred] node[below] {$J_n$} (m-3-6);
				\path				(m-1-3) edge[eq] (m-2-3)
										(m-1-5) edge[eq] (m-2-5)
										(m-1-1) edge[eq] (m-2-1)
										(m-2-7) edge[eq] (m-3-6)
										(m-2-4) edge[cell, transform canvas={xshift=0.125em}] node[right] {$\phi_n$} (m-3-4);
				\draw[font=\scriptsize]	($(m-1-6)!0.5!(m-2-6)$) node {$\cart$};
				\draw				($(m-1-4)!0.5!(m-2-4)$) node {$\dotsb$};
			\end{tikzpicture} \quad = \quad \begin{tikzpicture}[textbaseline]
				\matrix(m)[math35, column sep={2.1em,between origins}]
					{ X_{n0} & & X_{n1} & & X_{n(m_n)'} & & Y \\
						A_{n'} & & & & & & Y \\
						& A_{n'} & & & & X_{nm_n} & \\ };
				\path[map]	(m-1-1) edge[barred] node[above] {$H_{n1}$} (m-1-3)
														edge node[left] {$f_{n'}$} (m-2-1)
										(m-1-5) edge[barred] node[above, xshift=-12pt] {$H_{nm_n}(\id, f)$} (m-1-7)
										(m-2-1) edge[barred] node[below] {$J_n(\id, f)$} (m-2-7)
										(m-2-7) edge node[right] {$f$} (m-3-6)
										(m-3-2) edge[barred] node[below] {$J_n$} (m-3-6);
				\path				(m-2-1) edge[eq] (m-3-2)
										(m-1-7) edge[eq] (m-2-7)
										(m-1-4) edge[cell] node[right] {$\phi_n'$} (m-2-4);
				\draw[font=\scriptsize]	([yshift=-0.25em]$(m-2-4)!0.5!(m-3-4)$) node {$\cart$};
				\draw				(m-1-4) node[xshift=-3pt] {$\dotsb$};
			\end{tikzpicture}
		\end{equation}
		
		The notion of \emph{left pointwise cocartesian} path is horizontally dual. A path that is both left and right pointwise cocartesian is called \emph{pointwise cocartesian}.
	\end{definition}
	Notice that any right (or left) pointwise cocartesian path is cocartesian, by taking $f = \id_{X_{nm_n}}$ in the above. Conversely, in \lemref{right pointwise cocartesian with respect to a morphism that has a conjoint} below we will see that any cocartesian path is pointwise with respect to morphisms $f$ that admit conjoints. A single horizontal cocartesian cell $\cell\phi{(H_1, \dotsc, H_n)}J$ is called \emph{pointwise cocartesian} whenever the singleton path $(\phi)$ is pointwise cocartesian; in that case we call $J$ the \emph{pointwise composite} of $(H_1, \dotsc, H_n)$.
	
	\begin{example} \label{horizontal composites in (V, V')-Prof}
		Let $\hmap{(H_1, \dotsc, H_n)}{X_0}{X_n}$ be a path of $\V$"/profunctors. As anticipated in the introduction to this section, a horizontal cell $\cell\phi{(H_1, \dotsc, H_n)}J$ is pointwise cocartesian in $\enProf\V$ (\exref{enriched profunctors}) if and only if, for all pairs $x \in X_0$ and $y \in X_n$, the components \mbox{$\map\phi{H_1(x, u_1) \tens \dotsb H_n(u_{n'}, y)}{J(x,y)}$} define $J{(x,y)}$ as the coend
		\begin{displaymath}
			J(x, y) = \intl^{u_1 \in X_1} \dotsb \intl^{u_{n'} \in X_{n'}} H_1(x, u_1) \tens \dotsb \tens H_n(u_{n'}, y)
		\end{displaymath}
		which is preserved by the monoidal product $\tens$ of $\V$ on both sides. The `only if'"/part follows from applying \exref{compositions of enriched profunctors} to the restrictions of $\phi$ along $\V$"/functors of the form $\map xI{X_0}$ and $\map yI{X_n}$. The `if'"/part follows from the ``functoriality of coends'', dual to that of ends as described in Section~2.1 of \cite{Kelly82}. We conclude that $\enProf\V$ is a pseudo double category whenever $\V$ has large colimits that are preserved by $\tens$ on both sides.
		
		Now let $\V \subset \V'$ be a universe enlargement as in \exref{(V, V')-Prof}. Here $\V'$ is large cocomplete and closed, so that $\enProf{\V'}$ is a pseudo double category by the above.  Since the embedding $\enProf{(\V, \V')} \hookrightarrow \enProf{\V'}$ preserves cartesian cells, \lemref{locally full and faithful functors reflect cocartesian paths} below implies that the pointwise composite $(H_1 \hc \dotsb \hc H_n)$ exists in $\enProf{(\V, \V')}$ whenever the coends above, which exist in $\V'$, are isomorphic to $\V$"/objects.
	\end{example}
	
	\begin{example} \label{composites of small profunctors}
		Let $\hmap JAB$ and $\hmap HBE$ be small $\V$"/profunctors between (possibly large) $\V$"/categories; see \exref{small V-profunctors}. If the monoidal product $\tens$ of $\V$ preserves colimits (large ones, if $B$ is large) on both sides then, as we will show, the composite $J \hc H$ can be computed as the family of small colimits
		\begin{displaymath}
			(J \hc H)(x, z) = \intl^{y' \in B_z} J(x, y') \tens H(y', z)
		\end{displaymath}
		where $B_z \subseteq B$ are the small sub"/$\V$"/categories that exhibit $H$ as small (see \exref{small V-profunctors}). Moreover these colimits, if they exist, form a small $\V$"/profunctor which forms the pointwise composite of $J$ and $H$ in $\ensProf\V$. We conclude that $\ensProf\V$, which has horizontal units by \exref{restrictions of small profunctors}, is a pseudo double category whenever $\V$ is small cocomplete and $\tens$ preserves large colimits on both sides.
		
		To see the above choose any universe enlargement $\V \subset \V'$ (\exref{(V, V')-Prof}). By the previous example the pointwise composite $(J \hc H)$ exists in $\enProf{\V'}$: it is defined by the coends on the left below. The cascade of isomorphisms below shows that $(J \hc H)$ can be computed as above. Here we have used the smallness of $H$, the assumption that $\tens$ preserves large colimits on both sides, the ``interchange of coends'' theorem (see e.g.\ Formula~2.9 of \cite{Kelly82}), while the last isomorphism follows from the enriched Yoneda's lemma, see e.g.\ Formula~3.71 of \cite{Kelly82}.
		\begin{multline*}
			(J \hc H)(x, z) = \intl^{y \in B} J(x, y) \tens H(y, z) \iso \intl^{y \in B} J(x, y) \tens \Bigpars{\intl^{y' \in B_z} B(y, y') \tens H(y', z)} \\
			\iso \intl^{y' \in B_z} \Bigpars{\intl^{y \in B} J(x, y) \tens B(y, y')} \tens H(y', z) \iso \intl^{y' \in B_z} J(x, y') \tens H(y', z)
		\end{multline*}
		Now assume that the small colimit above (and thus all colimits above) exists in $\V$. To see that, in this case, $(J \hc H)$ is again a small $\V$"/profunctor take, for each $z \in E$, $A_z \subseteq A$ to be the smallest full sub"/$\V$"/category containing all $A_y$, where $y$ ranges over the objects of $B_z$. Then $A_z$ is small and we have
		\begin{multline*}
			\intl^{x' \in A_z} A(x, x') \tens (J \hc H)(x', z) = \intl^{x' \in A_z} A(x, x') \tens \Bigpars{\intl^{y \in B} J(x', y) \tens H(y, z)}\\
			\iso \intl^{y \in B} \Bigpars{\!\intl^{x' \in A_z} A(x, x') \tens J(x', y)} \tens H(y, z') \iso \intl^{y \in B} J(x, y) \tens H(y, z) = (J \hc H)(x, z),
		\end{multline*}
		which shows that $(J \hc H)$ is small. For the second isomorphism here recall from \exref{small V-profunctors} that each $J(\dash, y)$ is a left Kan extension along $A_y \subseteq A$: the isomorphism follows from the fact that the latter factors as a Kan extension along $A_z \subseteq A$ as a consequence of the ``pasting lemma'' for Kan extensions, see e.g.\ Theorem~4.47 of \cite{Kelly82}. We can now conclude that $(J \hc H)$, as defined above, exists in $\ensProf\V$; that it forms the pointwise composite of $J$ and $H$ there follows from applying the lemma below to the locally full embedding $\ensProf\V \hookrightarrow \enProf{\V'}$ which, as follows from \exref{restrictions of small profunctors}, preserves cartesian cells.
	\end{example}
	
	Besides reflecting restrictions (\lemref{locally full and faithful functors reflect cartesian cells}), locally full and faithful functors reflect horizontal composites.
	\begin{lemma} \label{locally full and faithful functors reflect cocartesian paths}
		Any locally full and faithful functor $\map F\K\L$ (\defref{full and faithful functor}) reflects weakly cocartesian paths, that is a path $(\phi_1, \dotsc, \phi_n) \in \K$ is weakly cocartesian whenever its image $(F\phi_1, \dotsc, F\phi_n)$ is weakly cocartesian in $\L$. Likewise $F$ reflects horizontal cocartesian cells, i.e.\ horizontal composites.
		
		If moreover $F$ preserves unary cartesian cells then it reflects (right/left) (pointwise) cocartesian paths as well.
	\end{lemma}
	
	Pointwise cocartesian paths are coherent in the following sense.
	\begin{lemma} \label{coherence of pointwise cocartesian paths}
		If the path $\ul \phi = (\phi_1, \dotsc, \phi_n)$ is right pointwise cocartesian then any path of the form $(\phi_1, \dotsc, \phi_n')$, as in \defref{pointwise cocartesian path}, is again right pointwise cocartesian. An analogous result holds for (left) pointwise cocartesian paths.
	\end{lemma}
	\begin{proof}
		Let $\map fY{X_{nm_n}}$ be as in \defref{pointwise cocartesian path}; that is $H_{nm_n}(\id, f)$ and $J_n(\id, f)$ exist. Let $\map gZY$ be any morphism such that $H_{nm_n}(\id, f)(\id, g) \iso H_{nm_n}(\id, f \of g)$ and $J_n(\id, f)(\id, g) \iso J_n(\id, f \of g)$ exist, where the isomorphisms follow from the pasting lemma for cartesian cells (\lemref{pasting lemma for cartesian cells}). Consider the unique factorisation $\phi_n''$ in \mbox{$\phi'_n \of (\id, \dotsc, \id, \cart) = \cart \of \phi_n''$}, as in \defref{pointwise cocartesian path} but for $\phi'$, where the cartesian cells define $H_{nm_n}(\id, f)(\id, h)$ and $J_n(\id, f)(\id, h)$ respectively; we have to show that $(\phi_1, \dotsc, \phi_n'')$ is cocartesian. To see this consider the following equation where, in each composite, the bottom cartesian cell (denoted `c') defines a restriction along $f$ and the top cartesian cell (also denoted `c') defines a restriction along $g$, and where the identities follow from the definitions of $\phi_n''$ and $\phi_n'$ respectively.
		\begin{displaymath}
			\begin{tikzpicture}[scheme]
				\draw	(1,3) -- (0,3) -- (0,1) -- (0.5,0) -- (2.5,0) -- (3,1) -- (3,3) -- (2,3) (0,1) -- (3,1) (0,2) -- (3,2);
				\draw[shift={(0.5,0.5)}]	(1,2.5) node[xshift=0.5pt] {$\dotsb$}
							(1,2) node {$\phi_n''$}
							(1,1) node {c}
							(1,0) node {c};
			\end{tikzpicture} \quad = \quad \begin{tikzpicture}[scheme]
				\draw	(0,2) -- (1,2) -- (1,3) -- (0,3) -- (0,1) -- (0.5,0) -- (2.5,0) -- (3,1) -- (3,3) -- (2,3) -- (2,2) -- (3,2) (0,1) -- (3,1);
				\draw[shift={(0.5,0.5)}]	(1,2) node[xshift=0.5pt] {$\dotsb$}
							(2,2) node {c}
							(1,1) node {$\phi_n'$}
							(1,0) node {c};
			\end{tikzpicture} \quad = \quad \begin{tikzpicture}[scheme]
				\draw	(0,1) -- (1,1) -- (1,3) -- (0,3) -- (0,1) -- (0.5,0) -- (2.5,0) -- (3,1) -- (3,3) -- (2,3) -- (2,1) -- (3,1) (0,2) -- (1,2) (2,2) -- (3,2);
				\draw[shift={(0.5,0.5)}]	(1,2) node[xshift=0.5pt] {$\dotsb$}
							(2,2) node {c}
							(1,1) node[xshift=0.5pt] {$\dotsb$}
							(1,0) node {$\phi_n$}
							(2,1) node {c};
			\end{tikzpicture}
		\end{displaymath}
		The composites of cartesian cells in the left-hand and right-hand sides above are again cartesian by the pasting lemma, so that $(\phi_1, \dotsc, \phi_n'')$ is cocartesian because $(\phi_1, \dotsc, \phi_n)$ is right pointwise cocartesian. This concludes the proof.
	\end{proof}
	
	The pasting lemma for cocartesian paths (\lemref{pasting lemma for cocartesian paths}) induces one for pointwise cocartesian paths as follows.
	\begin{lemma}[Pasting lemma] \label{pasting lemma for right pointwise cocartesian paths}
		Consider the configuration of cells of \lemref{pasting lemma for cocartesian paths}. Assume that all its cells $\psi_i$ and $\phi_{jk}$ are unary and that the vertical targets of the last cells $\psi_n$ and $\phi_{nm_n}$ are both the identity morphism on the object $C_n$. The assertions \textup{(a)} and \textup{(b)} of \lemref{pasting lemma for cocartesian paths} also hold after replacing `cocartesian' by `right pointwise cocartesian with respect to $f$', where $\map fY{C_n}$ is any morphism. Similarly these assertions also apply to (left) pointwise cocartesian paths.
	\end{lemma}
	\begin{proof}
		We prove that assertion \lemref{pasting lemma for cocartesian paths}(a) holds for the `right pointwise cocartesian with respect to $\map fY{C_n}$' case; the proofs for the other assertions are analogous. Assume that the paths $\ul\psi$ and $(\phi_{11}, \dotsc, \phi_{nm_n})$ are right pointwise cocartesian with respect to $f$ so that, by \defref{pointwise cocartesian path}, the following restrictions along $f$ exist: those of the horizontal targets of $\psi_n$ and $\phi_{nm_n}$ as well as that of the last morphism in the horizontal source of $\phi_{nm_n}$. Using these restrictions we obtain factorisations $\psi_n'$ and $\phi_{nm_n}'$, as in \defref{pointwise cocartesian path}, such that the following equation holds, where `c' denotes any cartesian cell defining one of the restrictions along $f$.
		\begin{displaymath}
			\begin{tikzpicture}[scheme, x=0.5cm]
				\draw (0,2) -- (1,2) -- (1,3) -- (0,3) -- (0,0) -- (9,0) -- (9,3) -- (8,3) -- (8,2) -- (9,2) (3,2) -- (2,2) -- (2,3) -- (3,3) -- (3,1) -- (0,1) (6,2) -- (7,2) -- (7,3) -- (6,3) -- (6,1) -- (9,1);
				\draw	(1.5,2.5) node[xshift=0.75pt] {$\dotsb$}
							(1.5,1.5) node {$\phi_{n1}$}
							(4.5,2) node[font=] {$\dotsb$}
							(4.5,0.5) node {$\psi_n$}
							(7.5,2.5) node[xshift=0.75pt] {$\dotsb$}
							(7.5,1.5) node {$\phi_{nm_n}$}
							(8.5,2.5) node {c};
			\end{tikzpicture} \mspace{12mu} = \mspace{12mu} \begin{tikzpicture}[scheme, x=0.5cm]
				\draw (1,3) -- (0,3) -- (0,0) -- (9,0) -- (9,3) -- (8,3) (2,3) -- (3,3) -- (3,1) -- (0,1) (7,3) -- (6,3) -- (6,1) -- (9,1) (0,2) -- (3,2) (6,2) -- (9,2);
				\draw	(1.5,3) node[xshift=0.75pt, yshift=-0.25pt] {$\dotsb$}
							(1.5,2.5) node {$\phi_{n1}$}
							(4.5,2.5) node[font=] {$\dotsb$}
							(4.5,1.5) node[font=] {$\dotsb$}
							(4.5,0.5) node {$\psi_n$}
							(7.5,3) node[xshift=0.75pt, yshift=-0.25pt] {$\dotsb$}
							(7.5,2.5) node[yshift=1pt] {$\phi_{nm_n}'$}
							(7.5,1.5) node {c};
			\end{tikzpicture} \mspace{12mu} = \mspace{12mu} \begin{tikzpicture}[scheme, x=0.5cm]
				\draw (1,3) -- (0,3) -- (0,0) -- (9,0) -- (9,3) -- (8,3) (2,3) -- (3,3) -- (3,2) -- (0,2) (7,3) -- (6,3) -- (6,2) -- (9,2) (0,1) -- (9,1);
				\draw	(1.5,3) node[xshift=0.75pt, yshift=-0.25pt] {$\dotsb$}
							(1.5,2.5) node {$\phi_{n1}$}
							(4.5,2.5) node[font=] {$\dotsb$}
							(4.5,1.5) node {$\psi_n'$}
							(7.5,3) node[xshift=0.75pt, yshift=-0.25pt] {$\dotsb$}
							(7.5,2.5) node[yshift=1pt] {$\phi_{nm_n}'$}
							(4.5,0.5) node {c};
			\end{tikzpicture}
		\end{displaymath}
		The above equation implies that the unique factorisation $\lbrack\psi_n \of (\phi_1, \dotsc, \phi_n)\rbrack'$ corresponding to $\psi_n \of (\phi_{n1}, \dotsc, \phi_{nm_n})$, as in \defref{pointwise cocartesian path} and with respect to $f$, coincides with \mbox{$\psi_n' \of (\phi_{n1}, \dotsc, \phi_{nm_n}')$}. By assumption $(\psi_1, \dotsc, \psi_n')$ and $(\phi_{11}, \dotsc, \phi_{nm_n}')$ are cocartesian so that $\bigpars{\psi_1 \of (\phi_{11}, \dotsc, \phi_{1m_1}), \dotsc, \lbrack\psi_n \of (\phi_1, \dotsc, \phi_n)\rbrack'}$ is cocartesian too by \lemref{pasting lemma for cocartesian paths}(a).
	\end{proof}
	
	Pointwise cocartesian cells can be obtained from the following lemmas.
	\begin{lemma} \label{cocartesian cell defining restriction is pointwise}
		Let $(\psi, \phi)$ be a pair of cells that satisfies both identities of \lemref{restrictions and composites}. The cocartesian cell $\phi$ is pointwise cocartesian.
	\end{lemma}
	\begin{proof}
		We will show that $\phi$ is right pointwise cocartesian; a horizontally dual argument shows that $\phi$ is left pointwise cocartesian too. Let $\map pYB$ be any morphism such that $g^*(\id, p) \iso (g \of p)^*$ (see \lemref{companion of a composite}) and $J(\id, p)$ exist. Let $\cell{\phi'}{f_* \conc \ul K \conc (g \of p)^*}{J(\id, p)}$ be the unique factorisation in $\phi \of (\id, \id, \cart) = \cart \of \phi'$, as in \defref{pointwise cocartesian path}, where the cartesian cells define the restrictions along $p$. We have to show that $\phi'$ is cocartesian. To see this compose the first identity of \lemref{restrictions and composites} with the cartesian cell defining $g^*(\id, p)$, giving the first identity in the equation below. The second identity follows from the definition of $\phi'$.
		\begin{displaymath}
			\begin{tikzpicture}[scheme, yshift=0.8em]
				\draw	(0,1) -- (0.5,0) -- (2.5,0) -- (3,1) -- (3,2) -- (0,2) -- (0,1) -- (3,1) (0.5,0) -- (1,1) -- (1,2) (2.5,0) -- (2,1) -- (2,2);
				\draw	(0.5,0.6) node {c}
							(2.5,0.6) node {c}
							(2.5,1.5) node {c};
			\end{tikzpicture} \quad = \quad \begin{tikzpicture}[scheme]
				\draw (0,2) -- (3,2) -- (3,3) -- (0,3) -- (0,2) -- (1,1) -- (1,0) -- (2,0) -- (2,1) -- (3,2) (1,1) -- (2,1) (1,2) -- (1,3) (2,2) -- (2,3);
				\draw	(1.5,0.5) node {$\psi$}
							(1.5,1.5) node {$\phi$}
							(2.5,2.5) node {c};
			\end{tikzpicture} \quad = \quad \begin{tikzpicture}[scheme]
				\draw (1,2) -- (1,0) -- (2,0) -- (2,2) -- (3,3) -- (0,3) -- (1,2) -- (2,2) (1,1) -- (2,1);
				\draw (1.5,0.5) node {$\psi$}
							(1.5,1.5) node {c}
							(1.5,2.5) node {$\phi'$};
			\end{tikzpicture}
		\end{displaymath}
		Since the composite of cartesian cells in the left"/hand side defines the companion of $g \of p$, the equation above is of the same form as the first identity of \lemref{restrictions and composites}. Moreover by the pasting lemma (\lemref{pasting lemma for cartesian cells}) the composite of the bottom two cells in the right"/hand side is cartesian, so that $\phi'$ is cocartesian by \lemref{restrictions and composites}.
	\end{proof}
	
	\begin{lemma} \label{right pointwise cocartesian with respect to a morphism that has a conjoint}
		Consider the path $\ul\phi = (\phi_1, \dotsc, \phi_n)$ and the morphism $\map fY{X_{nm_n}}$ of \defref{pointwise cocartesian path}. If the conjoint $f^*$ exists and $\ul\phi$ is cocartesian then $\ul\phi$ is right pointwise cocartesian with respect to $f$. An analogous result holds for (left) pointwise cocartesianness.
		
		Consequently in an augmented virtual equipment (\defref{augmented virtual equipment}) the horizontal composite of a path $\hmap{(H_1, \dotsc, H_n)}{X_0}{X_n}$ is pointwise whenever $X_0$ and $X_n$ are unital.
	\end{lemma}
	\begin{proof}
		As in \defref{pointwise cocartesian path} we assume that $H_{nm_n}(\id, f)$ and $J(\id, f)$ exist. By \lemref{restrictions and composites} we have $H_{nm_n}(\id, f) \iso H_{nm_n} \hc f^*$ and $J(\id, f) \iso J \hc f^*$ such that each pair of cartesian and cocartesian cells, defining the restriction and the composite, satisfy the identities of that lemma. Let $\phi_n'$ be the factorisation as in \defref{pointwise cocartesian path}; we have to show that $(\phi_1, \dotsc, \phi_n')$ is cocartesian. To see this consider the following equation of composites, where the cartesian cells defining $H_{nm_n}(\id, f)$, $J(\id, f)$ and $f^*$ are denoted `c' and the cocartesian cells defining $H_{nm_n} \hc f^*$ and $J \hc f^*$ are denoted `cc'. The identities follow from the definition of $\phi_n'$ and the first identity of \lemref{restrictions and composites}.
		\begin{displaymath}
			\begin{tikzpicture}[scheme]
				\draw	(0,2) -- (1,2) -- (1,3) -- (0,3) -- (0,1) -- (1,0) -- (2,0) -- (3,1) -- (3,3) -- (2,3) -- (2,2) -- (3,2) (0,1) -- (3,1);
				\draw (1.5,0.5) node {c}
							(1.5,1.5) node {$\phi_n'$}
							(1.5,2.5) node[xshift=.75pt] {$\dotsb$}
							(2.5,2.5) node {cc};							
			\end{tikzpicture} \quad = \quad \begin{tikzpicture}[scheme]
				\draw (0,1) -- (1,1) -- (1,3) -- (0,3) -- (0,1) -- (1,0) -- (2,0) -- (3,1) -- (3,3) -- (2,3) -- (2,1) -- (3,1) (0,1) -- (1,1) (2,1) -- (3,1) (0,2) -- (1,2) (2,2) -- (3,2);
				\draw (1.5,0.5) node {$\phi_n$}
							(1.5,1.5) node[xshift=.75pt] {$\dotsb$}
							(2.5,1.5) node {c}
							(1.5,2.5) node[xshift=.75pt] {$\dotsb$}
							(2.5,2.5) node {cc};
			\end{tikzpicture} \quad = \quad \begin{tikzpicture}[scheme, yshift=1.6em]
				\draw	(1,1) -- (0,1) -- (0.5,0) -- (3.5,0) -- (4,1) -- (2,1) (3.5,0) -- (3,1);
				\draw	(1.75,0.5) node {$\phi_n$}
							(1.5,1) node[xshift=.75pt] {$\dotsb$}
							(3.5,0.6) node {c};
			\end{tikzpicture} \quad = \quad \begin{tikzpicture}[scheme]
				\draw (1,3) -- (0,3) -- (0.5,2) -- (0.5,1) -- (1.5,0) -- (2.5,0) -- (3.5,1) -- (3.5,2) -- (4,3) -- (2,3) (0.5,1) -- (3.5,1) (0.5,2) -- (3.5,2) (2.5,2) -- (3,3);
				\draw (2,0.5) node {c}
							(2,1.5) node {cc}
							(1.5,2.5) node {$\phi_n$}
							(1.5,3) node[xshift=.75pt] {$\dotsb$};
			\end{tikzpicture}
		\end{displaymath}
		We conclude that $\phi_n' \of (\id, \dotsc, \id, \cocart) = \cocart \of (\phi_n, \id)$, by the uniqueness of factorisations through cartesian cells. It then follows from the pasting lemma (\lemref{pasting lemma for cocartesian paths}) that $\ul\phi$ being cocartesian implies that $\bigpars{\phi_1, \dotsc, \cocart \of (\phi_n, \id)} = \bigpars{\phi_1, \dotsc, \phi_n' \of (\id, \dotsc, \id, \cocart)}$ is cocartesian which in turn means that $(\phi_1, \dotsc, \phi_n')$ is cocartesian. This proves the first assertion. The final assertion follows by recalling from \cororef{restrictions in terms of units} that, in an augmented virtual equipment, all morphisms into unital objects $X_0$ and $X_n$ admit companions and conjoints.
	\end{proof}
	
	\section{The equivalence of unital augmented virtual double categories and unital virtual double categories}\label{unital virtual double categories section}
	In this last section we will show that the notions of augmented virtual double category and virtual double category are equivalent when all horizontal units exist. We denote by $\VirtDblCat_\textup u$ the locally full sub"/$2$"/category of $\VirtDblCat$ consisting of virtual double categories that have all horizontal units, \emph{normal} functors---that preserve the cocartesian cells defining horizontal units---, and all transformations between them. Likewise $\AugVirtDblCat_\textup u \subset \AugVirtDblCat$ denotes the full sub"/$2$"/category generated by the augmented virtual double categories that have all horizontal units. Remember that any functor of augmented virtual double categories preserves horizontal units (\cororef{functors preserve companions and conjoints}).
	
	Recall the strict $2$-functor $\map U{\AugVirtDblCat}{\VirtDblCat}$ (\propref{2-category of augmented virtual double categories}) that maps any augmented virtual double category $\K$ to the underlying virtual double category $U(\K)$ consisting of the unary cells of $\K$. Clearly unary cocartesian cells in $\K$ are again cocartesian in $U(\K)$ so that $U$ restricts to a strict $2$"/functor \mbox{$\map U{\AugVirtDblCat_\textup u}{\VirtDblCat_\textup u}$}. The theorem of this section proves that the latter $2$"/functor, together with the assignment $\K \mapsto N(\K)$ of \exref{augmented virtual double category from a unital virtual double category}, extends to a $2$"/equivalence \mbox{$\AugVirtDblCat_\textup u \simeq \VirtDblCat_\textup u$}.
	
	In order to make the distinction between cells of $N(\K)$ and those of $\K$ clear, in this section only we will place a bar over those of $N(\K)$. Thus $N(\K)$ has the same objects and morphisms as $\K$ while each unary cell $\bar\phi$ of $N(\K)$ corresponds to a cell $\phi$ in $\K$ and each nullary cell $\bar\psi$ of $N(\K)$, of the shape on the left below, corresponds to a unary cell $\psi$ in $\K$ as on the right, where $I_C$ is the chosen horizontal unit for $C$.
	\begin{displaymath}
		\begin{tikzpicture}[baseline]
			\matrix(m)[math35, column sep={1.75em,between origins}]
				{A_0 & & A_1 & \dotsb & A_{n'} & & A_n \\ & & & C & & & \\};
			\path[map]	(m-1-1) edge[barred] node[above] {$J_1$} (m-1-3)
													edge node[below left] {$f$} (m-2-4)
									(m-1-5) edge[barred] node[above] {$J_n$} (m-1-7)
									(m-1-7) edge node[below right] {$g$} (m-2-4);
			\path				(m-1-4) edge[cell] node[right] {$\bar \psi$} (m-2-4);
		\end{tikzpicture} \qquad \qquad \qquad \qquad \begin{tikzpicture}[baseline]
			\matrix(m)[math35]{A_0 & A_1 & A_{n'} & A_n \\ C & & & C \\};
			\path[map]	(m-1-1) edge[barred] node[above] {$J_1$} (m-1-2)
													edge node[left] {$f$} (m-2-1)
									(m-1-3) edge[barred] node[above] {$J_n$} (m-1-4)
									(m-1-4) edge node[right] {$g$} (m-2-4)
									(m-2-1) edge[barred] node[below] {$I_C$} (m-2-4);
			\path[transform canvas={xshift=1.75em}]	(m-1-2) edge[cell] node[right] {$\psi$} (m-2-2);
			\draw				($(m-1-2)!0.5!(m-1-3)$) node {$\dotsb$};
		\end{tikzpicture}
	\end{displaymath}
	Recall that for each object $C \in \K$ we denote by $\cell{\eta_C}C{I_C}$ the cocartesian cell in $\K$ that defines the chosen horizontal unit $\hmap{I_C}CC$. Using the bar notation, composition in $N(\K)$ is defined as
	\begin{equation} \label{composition in N(K)}
		\bar\chi \of (\bar\xi_1, \dotsc, \bar\xi_n) \dfn \ol{\chi' \of (\xi_1, \dotsc, \xi_n)}
	\end{equation}
	where $\chi'$ is the unique factorisation of $\chi$ through the cocartesian path of cells $(\eta_{\bar\xi_1}, \dotsc, \eta_{\bar\xi_n})$ in $\K$, where $\eta_{\bar\xi_i} \dfn \eta_{C_{i'}}$ if $\bar\xi_i$ is nullary with horizontal target $C_{i'}$ and $\eta_{\bar\xi_i} \dfn \id_{K_i}$ if $\bar\xi_i$ is unary with horizontal target $\hmap{K_i}{C_{i'}}{C_i}$. The identity cells in $N(\K)$, for morphisms $\hmap JAB$ and $\map fAC$, are given by
	\begin{displaymath}
		\id_J \dfn \ol{\id_J} \qquad \qquad \text{and} \qquad \qquad \id_f \dfn \ol{\eta_C \of f}.
	\end{displaymath}
	
	\begin{theorem} \label{unital virtual double categories}
		The strict $2$-functor $\map U{\AugVirtDblCat_\textup u}{\VirtDblCat_\textup u}$ together with the assignment $\K \mapsto N(\K)$ (\exref{augmented virtual double category from a unital virtual double category}), both as recalled above, extend to a strict $2$"/equivalence $\AugVirtDblCat_\textup u \simeq \VirtDblCat_\textup u$.
	\end{theorem}
		\begin{proof}
		That the composition for $N(\K)$ as defined above satisfies the associativity and unit axioms is a straightforward consequence of those axioms in $\K$, combined with the uniqueness of the factorisations $\chi'$ in \eqref{composition in N(K)}.
		
		To show that $N(\K)$ has all horizontal units let $A$ be any object in $N(\K)$; we claim that $\cell{\bar{\eta_A}}A{I_A}$ defines $I_A$ as the horizontal unit of $A$ in $N(\K)$. To see this, consider the identity of $I_A$ as a nullary cell $\cell{\ol{\id_{I_A}}}{I_A}A$ in $N(\K)$; we will show that $\bar{\eta_A}$ and $\ol{\id_{I_A}}$ satisfy the horizontal unit identities of \lemref{unit identities}. Indeed, we have $\ol{\id_{I_A}} \of \bar{\eta_A} = \ol{(\id_{I_A} \of \eta_A)} = \bar{\eta_A} = \id_A$ (the identity cell of $A$ in $N(\K)$). On the other hand we have
		\begin{displaymath}
			\bar{\eta_A} \of \ol{\id_{I_A}} = \ol{\eta_A' \of \id_{I_A}} = \ol{\id_{I_A} \of \id_{I_A}} = \ol{\id_{I_A}} = \id_{I_A},
		\end{displaymath}
		where the right-hand side is the identity cell of $I_A$ in $N(\K)$ and where $\eta_A' = \id_{I_A}$ is the unique factorisation of $\eta_A$ through $\eta_{\ol{\id_{I_A}}} = \eta_A$.
		
		We conclude that $N(\K)$ forms a well-defined augmented virtual double category that has all horizontal units.	Next we extend the assignment $\K \mapsto N(\K)$ to a strict $2$-functor $\map N{\VirtDblCat_\textup u}{\AugVirtDblCat_\textup u}$. For the action of $N$ on morphisms consider a normal functor $\map F\K\L$ between unital virtual double categories. Since $F$ preserves the cocartesian cells $\eta_A$ of $\K$ we can obtain, for each object $A \in K$, an invertible horizontal cell $\cell{(F_I)_A}{FI_A}{I_{FA}}$ in $\L$ that is the unique factorisation in
		\begin{equation} \label{unitors}
			\begin{tikzpicture}[textbaseline]
				\matrix(m)[math35, column sep={1.75em,between origins}]{& FA & \\ FA & & FA \\};
				\path[map]	(m-2-1) edge[barred] node[below] {$I_{FA}$} (m-2-3);
				\path				(m-1-2) edge[transform canvas={xshift=-2pt}, eq] (m-2-1)
														edge[transform canvas={xshift=2pt},eq] (m-2-3)
														edge[transform canvas={shift={(-0.7em,-0.5em)}}, cell] node[right] {$\eta_{FA}$} (m-2-2);
			\end{tikzpicture} \quad = \quad \begin{tikzpicture}[textbaseline]
				\matrix(m)[math35, column sep={1.75em,between origins}]{& FA & \\ FA & & FA \\ FA & & FA \\};
				\path[map]	(m-2-1) edge[barred] node[below] {$FI_A$} (m-2-3)
										(m-3-1) edge[barred] node[below] {$I_{FA}$} (m-3-3);
				\path				(m-1-2) edge[transform canvas={xshift=-2pt}, eq] (m-2-1)
														edge[transform canvas={xshift=2pt},eq] (m-2-3)
														edge[transform canvas={shift={(-0.7em,-0.5em)}}, cell] node[right] {$F\eta_A$} (m-2-2)
										(m-2-1) edge[eq] (m-3-1)
										(m-2-2) edge[transform canvas={shift={(-1.05em,-0.15em)}}, cell] node[right] {$(F_I)_A$} (m-3-2)
										(m-2-3) edge[eq] (m-3-3);
			\end{tikzpicture}.
		\end{equation}
		We define $\map{NF}{N(\K)}{N(\L)}$ as follows. On objects and morphisms it simply acts as $F$ does. To define its action on cells we first define, for each $\bar\xi$ in $N(\K)$, the cell $\delta_{\bar\xi}$ in $\L$ by $\delta_{\bar\xi} \dfn (F_I)_C$ if $\bar\xi$ is nullary with horizontal target $C$, and $\delta_{\bar\xi} \dfn \id_{FK}$ if $\bar\xi$ is unary with horizontal target $\hmap KCD$; we then set $(NF)(\bar\xi) \dfn \ol{(\delta_{\bar\xi} \of F\xi)}$. That this assignment preserves identity cells is easily checked; that it preserves any composition $\bar\chi \of (\bar\xi_1, \dotsc, \bar\xi_n)$ in $N(\K)$, as in \eqref{composition in N(K)}, is shown by
		\begin{align*}
			(NF)(\bar\chi) \of \bigl(&(NF)(\bar\xi_1), \dotsc, (NF)(\bar\xi_n)\bigr) \\
			& = \ol{\delta_{\bar\chi} \of F\chi} \of \pars{\ol{\delta_{\bar\xi_1} \of F\xi_i}, \dotsc, \ol{\delta_{\bar\xi_n} \of F\xi_n}} \\
			& = \ol{\delta_{\bar\chi} \of (F\chi)' \of (\delta_{\bar\xi_1} \of F\xi_1, \dotsc, \delta_{\bar\xi_n} \of F\xi_n)} \\ 
			& = \ol{\delta_{\bar\chi} \of F(\chi') \of (F\xi_1, \dotsc, F\xi_n)} = \ol{\delta_{\bar\chi} \of F\bigpars{\chi' \of (\xi_1, \dotsc, \xi_n)}} \\
			& = (NF)\bigpars{\ol{\chi' \of (\xi_1, \dotsc, \xi_n)}} = (NF)\bigpars{\bar\chi \of (\bar\xi_1, \dotsc, \bar\xi_n)},
		\end{align*}
		where the third identity is shown as follows. The cells $(F\chi)'$ and $\chi'$, on either side, are the factorisations in $F\chi = (F\chi)' \of (\eta_{(NF)(\bar\xi_1)}, \dotsc, \eta_{(NF)(\bar\xi_n)})$ and $\chi = \chi' \of (\eta_{\bar\xi_1}, \dotsc, \eta_{\bar\xi_n})$ respectively. The identity follows from the fact that
		\begin{align*}
			(F\chi)' \of (\delta_{\bar\xi_1}, \dotsc, \delta_{\bar\xi_n}) \of (F\eta_{\bar\xi_1}, \dotsc, F\eta_{\bar\xi_n}) &= (F\chi)' \of \bigpars{\eta_{(NF)(\bar\xi_1)}, \dotsc, \eta_{(NF)(\bar\xi_n)}} \\
			&= F\chi = F(\chi') \of (F\eta_{\bar\xi_1}, \dotsc, F\eta_{\bar\xi_n})
		\end{align*}
		together with the uniqueness of factorisations through the path $(F\eta_{\bar\xi_1}, \dotsc, F\eta_{\bar\xi_n})$, which is cocartesian because $F$ is normal. This concludes the definition of $N$ on morphisms.
		
		Next consider a transformation $\nat\zeta FG$ between normal functors $F$ and $\map G\K\L$ of unital virtual double categories. We claim that the components of $\zeta$ again form a transformation $NF \Rar NG$, which we take to be the image $N\zeta$. For instance, that the components of $\zeta$ are natural with respect to a nullary cell $\cell{\bar\psi}{\ul J}C$ in $N(\K)$, with non"/empty horizontal source $\ul J$, is shown below, where $(\eta_{GC} \of \zeta_C)'$ is the unique factorisation of $\eta_{GC} \of \zeta_C$ through $\eta_{(NF)(\bar\psi)} = \eta_{FC}$.
		\begin{align*}
			(NG)(\bar\psi) \of (\bar\zeta_{J_1}, \dotsc, \bar\zeta_{J_n}) & = \ol{(G_I)_C \of G\psi} \of (\bar\zeta_{J_1}, \dotsc, \bar\zeta_{J_n}) \\
			& = \ol{(G_I)_C \of G\psi \of (\zeta_{J_1}, \dotsc, \zeta_{J_n})} = \ol{(G_I)_C \of \zeta_{I_C} \of F\psi} \\
			& = \ol{(\eta_{GC} \of  \zeta_C)' \of (F_I)_C \of F\psi} = \zeta_C \of (NF)(\bar\psi)
		\end{align*}
		Here the last identity follows from the definition of $\id_{\zeta_C}$ in $N(\L)$ while the penultimate identity follows from
		\begin{align*}
			(G_I)_C \of \zeta_{I_C} \of F\eta_C &= (G_I)_C \of G\eta_C \of \zeta_C = \eta_{GC} \of \zeta_C \\
			&= (\eta_{GC} \of \zeta_C)' \of \eta_{FC} = (\eta_{GC} \of \zeta_C)' \of (F_I)_C \of F\eta_C,
		\end{align*}
		by using that $F\eta_C$ is cocartesian. Naturality of the components of $\zeta$ with respect to cells in $N(\K)$ of other shapes can be shown similarly.
		
		That the assignments $\K \mapsto N(\K)$, $F \mapsto NF$ and $\zeta \mapsto N\zeta$ combine into a strict $2$-functor $\map N{\VirtDblCat_\textup u}{\AugVirtDblCat_\textup u}$ follows easily from the uniqueness of the factorisations \eqref{unitors}. It is also clear that the obvious isomorphism \mbox{$(U \of N)(\K) \iso \K$} of virtual double categories extends to an isomorphism $U \of N \iso \id$ of strict 2"/endofunctors on $\VirtDblCat$. Thus it remains to construct an invertible $2$"/natural transformation \mbox{$\tau\colon \id \xrar\iso N \of U$}. Given a unital augmented virtual double category $\K$ we define the functor $\map{\tau_\K}\K{(N \of U)(\K)}$ as follows. It is the identity on objects and morphisms, it is given by $\phi \mapsto \bar\phi$ on unary cells and by $\psi \mapsto \ol{\eta_C \of \psi}$ on nullary cells $\cell\psi{\ul J}C$. That these assignments preserve composites and identity cells is easily checked; that the family $\tau = (\tau_\K)_{\K}$ is $2$-natural is clear. Finally, the inverse functor $\map{\inv\tau}{(N \of U)(\K)}\K$ can be given as the identity on objects and morphisms, as $\bar\phi \mapsto \phi$ on unary cells and as $\bar\psi \mapsto \eps_C \of \psi$ on nullary cells $\cell{\bar\psi}{\ul J}C$, where $\cell{\eps_C}{I_C}C$ is the nullary cartesian cell that corresponds to $\cell{\eta_C}C{I_C}$ as in \lemref{unit identities}. This completes the proof.
	\end{proof}


\end{document}